\documentclass[11pt]{amsart}

\usepackage{lipsum}
\usepackage{amsfonts}
\usepackage{graphicx}
\usepackage{epstopdf}
\usepackage{amsmath}
\usepackage{amssymb}
\usepackage[utf8]{inputenc}
\usepackage[english]{babel}
\usepackage{tcolorbox}
\usepackage{pgfplots} 
\usepackage{algorithmic}

\usepackage[top=3.4cm,bottom=3.6cm,left=2.4cm,right=2.4cm]{geometry}
\usepackage{hyperref}
\hypersetup{
  colorlinks,
  citecolor= mybrown,
  linkcolor= link,
  urlcolor= link
}

\usepackage[algosection]{algorithm2e}
\makeatletter
\renewcommand{\@algocf@capt@plain}{above}
\makeatother

\definecolor{mybrown}{rgb}{0.47,0.12,0.12}
\definecolor{link}{rgb}{0.18,0.25,0.63}

\newcommand{\specificthanks}[1]{\@fnsymbol{#1}}

\numberwithin{equation}{section}

\newcommand{\norm}[1]{\left\| #1 \right\|}
\newcommand{\abs}[1]{ \left\vert #1 \right\vert}
\newcommand{\set}[1]{ \left\{ #1  \right\}}
\newcommand{\f}[1]{\mathbf{#1}}

\newcommand{\R}{ \mathbb{R}}
\newcommand{\C}{ \mathbb{C}}
\newcommand{\N}{ \mathbb{N}}
\newcommand{\om}{\Omega}

\makeatletter
\g@addto@macro{\endabstract}{\@setabstract}
\newcommand{\authorfootnotes}{\renewcommand\thefootnote{\@fnsymbol\c@footnote}}%
\makeatother

\newtheorem{proposition}{Proposition}[section]
\newtheorem*{prop*}{Proposition}
\newtheorem{assumption}[proposition]{Assumption}
\newtheorem{remark}[proposition]{Remark}
\newtheorem{theorem}[proposition]{Theorem}
\newtheorem{corollary}[proposition]{Corollary}
\newtheorem{lemma}[proposition]{Lemma}

\begin{document}

 \begin{center}
 \large
  \textbf{Quantitative Magnetic Resonance Imaging: From Fingerprinting to Integrated Physics-Based Models} \par \bigskip \bigskip
   \normalsize
\textsc{Guozhi Dong}\textsuperscript{\,$\ddagger$},  \textsc{Michael Hinterm\"uller}\textsuperscript{\,$\dagger\ddagger$} \textsc{and Kostas Papafitsoros}\textsuperscript{\,$\dagger$}
\let\thefootnote\relax\footnote{
\textsuperscript{$\dagger$}Weierstrass Institute for Applied Analysis and Stochastics (WIAS), Mohrenstrasse 39, 10117, Berlin, Germany
}
\let\thefootnote\relax\footnote{
\textsuperscript{$\ddagger$}Institute for Mathematics, Humboldt University of Berlin, Unter den Linden 6, 10099, Berlin, Germany}

\let\thefootnote\relax\footnote{
\hspace{3.2pt}Emails: \href{mailto:guozhi.dong@hu-berlin.de}{\nolinkurl{guozhi.dong@hu-berlin.de}}}
 \let\thefootnote\relax\footnote{
 \hspace{37pt}\href{mailto: Michael.Hintermueller@wias-berlin.de}{\nolinkurl{Michael.Hintermueller@wias-berlin.de}}
 }
 \let\thefootnote\relax\footnote{
 \hspace{37pt}\href{mailto:Kostas.Papafitsoros@wias-berlin.de}{\nolinkurl{Kostas.Papafitsoros@wias-berlin.de}}
 }

\end{center}

\begin{abstract}
Quantitative magnetic resonance imaging (qMRI) is concerned with estimating (in physical units) values of magnetic and tissue parameters e.g., relaxation times $T_1$, $T_2$,  or proton density $\rho$. 
Recently in [Ma et al., Nature, 2013], Magnetic Resonance Fingerprinting (MRF) was introduced as a technique being capable of simultaneously recovering such quantitative parameters by using a two step procedure: (i) given a probe, a series of magnetization maps are computed  and then (ii) matched to (quantitative) parameters with the help of a pre-computed dictionary which is related to the Bloch manifold.
In this paper, we first put MRF and its variants into a perspective with optimization and inverse problems to gain mathematical insights concerning identifiability of parameters under noise and interpretation in terms of optimizers. Motivated by the fact that the Bloch manifold is non-convex and that the accuracy of the MRF-type algorithms is limited by the ``discretization size'' of the dictionary, a novel physics-based method for qMRI is proposed.
In contrast to the conventional two step method, our model is dictionary-free and is rather governed by a single non-linear equation, which is studied analytically.
This non-linear equation is efficiently solved via robustified Newton-type methods.
The effectiveness of the new method for noisy and undersampled data is shown both analytically and via extensive numerical examples for which also improvement over MRF and its variants is documented.
\end{abstract}

\vspace{0.5cm}

\vspace{0.2cm}

\section{Introduction}

\subsection{Context}
The current routine of magnetic resonance imaging (MRI) examinations typically provides qualitative images of nuclear magnetization of tissue accompanied by contrast ``weights''.
Physicians then visually inspect these images, which, being qualitative only, may however not provide enough information for certain diagnostic purposes.
In order to remedy this, {\it quantitative MRI} (qMRI) seeks to not only visualize the structure of the imaged object, but also to provide accurate parameters values (in physical units) that characterize different tissue types.
Such quantities are typically the proton density $\rho$ of Hydrogen atoms in water molecules, and the longitudinal and the transverse relaxation parameters $T_1$ and $T_2$, respectively, among others. These magnetic parameters are related to the evolution of the net magnetization $\mathfrak{m}$ through the renowned Bloch equations \cite{Blo46}:
\begin{equation}
\label{eq:Bloch_intro}
\begin{array}{lll}
\frac{\partial m}{\partial t}(t)&= & m(t)\times \gamma B(t)-\left (\frac{m_{x}(t)}{T_{2}}, \frac{m_{y}(t)}{T_{2}}, \frac{m_{z}(t)-m_{eq}}{T_{1}} \right)^{\top},\\
m(0) & = & m_0.
\end{array}
\end{equation}
Here $m$, yielding $\mathfrak{m}=\rho m$, is the macroscopic magnetization of (Hydrogen) proton of some unitary density in the tissue under an external magnetic field $B$, and the relaxation rates $T_1$ and $T_2$ are associated model parameters. Further, $m_0$ represents an initial state. System \eqref{eq:Bloch_intro} is instrumental in our quantification process established below and will be further described in Section \ref{subsec:Bloch}.

Although qMRI techniques are still in their infancy, several interesting ideas and methods have already been conceived.
Early approaches \cite{Lar_etal89} are based on a set of spin echo or inversion recovery images that are reconstructed from $k$-space data with respect to various repetition times ($TR$) and echo times ($TE$). In that context, acquisitions are designed for each parameter individually. The overall technique is often referred to as parametric mapping method and consists of two steps: (i) reconstruct a sequence of images as in qualitative MRI, and (ii) for each pixel of those images fit its intensity to an ansatz curve characterized by the magnetic parameter associated to the tissue imaged at that pixel.
Based on this idea, many improvements have been suggested in the literature; see for instance \cite{HuaGraClaBilAlt12}.
The associated approaches aim to simplify the physical model and handle tissue parameters separately, as these are considered to be time consuming for the patient.

Another line of research, initiated by Ma et al. in \cite{Ma_etal13}  and named Magnetic Resonance Fingerprinting (MRF), has recently gained considerable attention.
First, in an offline phase, it builds a database (dictionary) consisting of all trajectories (fingerprints) of the evolution of the associated magnetization. Each of the latter is obtained by solving Bloch equations for some pre-selected combination of parameter values, typically those of $T_{1}$ and $T_{2}$ (but sometimes also others such as, e.g., the off-resonance frequency). 
The underlying parameter combinations stem from a (sufficiently large) selection within a region $C_{ad}$ which is meaningful for human tissue.
Hence, the outcome of this first step of the method is a physiologically informed dictionary $\mathrm{Dic}(C_{ad})$ (i.e., a look-up table) relating a set of feasible parameters to their associated solutions of Bloch equations. In a second phase, given a (sensed) magnetization trajectory that is assumed to be related to a solution of the Bloch equations, with the help of this look-up table, the method identifies the parameter values that fit best to this trajectory. This main principle behind MRF enables a simultaneous estimation of (quantitative) tissue parameters. 

As our new method is inspired by MRF we further detail the MRF-workflow. Focusing on a thin slice $\om$ of the tissue of interest, its first step is to  reconstruct a sequence of $L$ images $\{X^{(\ell)}\}_{\ell=1}^{L}$ from data $\{D^{(\ell)}\}_{\ell=1}^{L}$ as in qualitative MRI, using a sufficiently rich excitation process through $L$ fast radio pulses; see Section \ref{sec:Background} for more details. At every time step, the data consists of a sub-sampling of the Fourier coefficients of the magnetization. Sub-sampling occurs due to the short time between each excitation. In a typical MRF routine, the reconstruction of the magnetization relies on the pseudo-inverse of the Fourier transform. This leads to noticeable artifacts in the magnetization images. However, the evolution of the magnetization of a specific tissue element (voxel) along the series of the reconstructed magnetization images can be assumed to correspond (approximately) to the solution of Bloch equations with parameters that correspond to this specific voxel. Hence, the second step of MRF matches  the recorded trajectory of each voxel to a fingerprint in the pre-computed dictionary, typically through minimizing a least-squares distance. In this way, the parameter values that  correspond to the ``best'' fingerprint are then assigned to that very voxel.
Formally, the MRF procedure can be stated as follows:
\begin{align*}\label{eq:ori_MRF_1}
 &\text{- Compute }  X^{(\ell)}\in \underset{X}{\operatorname{argmin}} \norm{P^{(\ell)}\mathcal{F} X- D^{(\ell)}}_{2}^2,\quad \ell=1,\ldots,L \tag{MRF-step 1},\\
       \label{eq:ori_MRF_2}
  &\text{- Compute } m^* \in   \underset{m \in \text{Dic}(C_{ad})}{\operatorname{argmin}}\;  S(T_{x,y} m,X^*) \text{ with }  X^*:=(X^{(1)},\ldots, X^{(L)}). \; \tag{MRF-step 2}
  \end{align*}
In \eqref{eq:ori_MRF_1} $D=(D^{(1)},\ldots,D^{(L)})$ denotes the data obtained after each pulse, where $D^{(\ell)}$ is a sub-sampling of the Fourier coefficients of the magnetization, more precisely of its transverse component $T_{x,y}$, i.e., the first two components only. Here, $P^{\ell}$ is the $\ell$-th sub-sampling operator, and the Fourier transform is denoted by $\mathcal{F}$. Accordingly, the first step computes $L$ reconstructions of the magnetization of the tissue slice, i.e., $X^{(\ell)}=\mathcal{F}^{-1}(P^{(\ell)})^{\top} D^{(\ell)},$ where $\mathcal{F}^{-1}$ stands for the pseudo-inverse.
 In \eqref{eq:ori_MRF_2}, for every voxel $r$ (in practice, for every image pixel) the best approximation is obtained via exhaustive search over the dictionary $\text{Dic}(C_{ad})$. 
We recall that $\text{Dic}(C_{ad})=\{m^{\theta}:\; \theta\in C_{ad} \}$, where $\theta$ is a vector of tissue parameters -- here, for simplicity, $\theta=(T_{1}, T_{2})$ -- and $C_{ad}$ is the admissible domain for these parameters. By $m^{\theta}$ we denote the solution of  Bloch equations with parameter $\theta$, evaluated at the same time instances as for the magnetization responses.  Hence, every element of $\text{Dic}(C_{ad})$ is a vector sequence of length $L$.
The function $S(\cdot,\cdot)$ is a Euclidean distance of normalized quantities, in order to avoid the multiplicative effect that the  density $\rho$ has in the magnetization. 
Correspondingly, the minimization task in \eqref{eq:ori_MRF_2} has to be understood in a ``voxel-wise'' sense, i.e., it is performed as often as the number of voxels (in practice pixels). 
Finally, the spatial parameter maps are formed by assigning (to the corresponding voxels) the parameter values $\theta$ that correspond to the optimal matchings $m^{\theta}$.
For more details we refer to Section \ref{subsec:MRF}.

While first numerical results \cite{Ma_etal13} show that MRF is a promising qMRI approach, several issues remain open from a mathematical viewpoint which motivate our work. For instance,
with respect to stability one is interested in knowing whether two close trajectories yield similar parameter values.
Assuming that this is the case, i.e., the method is stable, and $L$ is fixed, then mainly two factors influence the accuracy of MRF: (i) the quality of the $L$ magnetization reconstructions, and (ii) the completeness (fineness) of the dictionary. Clearly, one is interested in analysing and optimizing both aspects. In this vein, the available literature mostly focuses on improving (i), 
 {\cite{Ass_etal18, DavPuyVanWia14, DonAmtKokSomBor17,McgPieMaJiaSayGulGri14,PieMaCheBadGri16,ZhaSetYeCauWal16,Zha_etal18} to name  a few, while for (ii) an interpolation of the dictionary entries for a better matching was recently considered \cite{MazWeiTalEld18}. However the limitation of the dictionary based matching remains there.}
Moreover, both steps of the MRF procedure may benefit from each other when combined. For example, in order to obtain improved reconstructions in \eqref{eq:ori_MRF_1}, it may be informed by the physics-based model built into \eqref{eq:ori_MRF_2}. This motivates our approach of integrating the Bloch equations already into \eqref{eq:ori_MRF_1}, yielding a single-step qMRI approach upon fixing the best approximation problem.

Davies et al. \cite{DavPuyVanWia14} proposed a first approach in this direction which exhibits advantages over the original MRF scheme. They coined the name {\bf BL}och response recovery through {\bf I}terated {\bf P}rojection (BLIP) for their solution scheme which relates to a projected Landweber-type iteration for reconstructing the magnetization. A key step of the procedure is to project, in every iteration, the current reconstruction onto a dictionary related to the \emph{Bloch manifold}. This leads to an improved solution for the magnetization especially in the case of strongly sub-sampled data.
The underlying constrained optimization problem reads
\begin{align}
\label{eq:BLIP}
& \min_{X }  \quad    \norm{P\mathcal{F} T_{x,y}X - D}_2^2  \quad
 \text{subject to (s.t.)}\quad X \in \R^+  \text{Dic}(C_{ad}). \tag{BLIP} 
\end{align}
While we defer more details on the BLIP method to Section \ref{subsec:MRF}, we mention already here that the potential non-convexity of the positive cone of Bloch manifolds $\R^+  \text{Dic}(C_{ad}) $ represents a major complication as the projection may become non-unique; compare Proposition \ref{prop:dis_bloch_non_convex} below. This problem is even more concerning when data is corrupted by noise. 
In addition, as the projection (matching) is still dictionary based, the method can be  memory consuming, especially when the dictionary is highly refined in order to have high accuracy.
 {During the preparation of the paper, some more advanced versions of the BLIP method and other generalizations of MRF have been  proposed (see e.g. \cite{Ass_etal18,DuaRepGomDavWia18,GolCheWiaDav18,MazWeiTalEld18,Zha_etal18}), which can improve the MRF or BLIP algorithms in different aspects.
In \cite{Ass_etal18,MazWeiTalEld18,Zha_etal18}  the low rank structure of the representation of the magnetization function was explored, therefore significant improvement on efficiency and accuracy of MRF (BLIP) algorithms can be achieved in both the reconstruction and matching steps. In \cite{GolCheWiaDav18} the authors arrange the dictionary entries using cover tree structures which can accelerate the convergence of the BLIP algorithm. Finally \cite{DuaRepGomDavWia18} deals with the partial volume effect of low resolution images, which, in the context of  continuous functions, is to enforce certain regularity on the variations of the parameter functions.}

\subsection{Our contribution }
Our work  has two major focus points. In a first part, contained in Section \ref{sec:Math_MRF},  we perform sensitivity analysis to show that the matching process is a well-posed inverse problem when it is restricted to the Bloch manifold. This fact partially explains why the concept of involving a dictionary in MRF has been so successful. In particular in Theorem \ref{thm:stability} we show that if two trajectories of the magnetization evolution as dictated by the Bloch equations are sufficiently close, then the same holds true for the associated parameters, i.e., $\|\theta -\theta^{\delta}\|\le C \delta$ if  $\|m-m^{\delta}\|\le \delta$.
Here $\theta,\;\theta^\delta$ are the inferred parameters given the Bloch trajectories $m,\;m^\delta$. The constant $C$ is independent of $\delta$, and the norms will be made precise later in the text.
Furthermore, we also establish a mathematical understanding of why a large number of frames $L$ yields a positive influence on the quality of the final result; compare Theorem \ref{thm:gaussian_data}. 

In the second part of our work, with the goal of avoiding the potentially ill-posed projection step in BLIP, we aim to solve the parameter identification problem directly subject to the Bloch manifold. The associated new single-step model reads: Find $(\rho,\theta)$ such that
\begin{equation}\label{eq:reduced_form}
Q(\rho,\theta):=P\mathcal{F}(\rho T_{x,y} m(\theta)) =D, \quad \text{with } (\rho(r),\theta(r))\in \mathbb{R}^{+}\times C_{ad}, \quad  \text{for every }r\in\om.
\end{equation}
Here the qMRI-operator $Q$ inserts the Bloch dynamics into the data acquisition, and by solving \eqref{eq:reduced_form} we can recover both $\rho$ and $\theta$. However, the non-linearity of $Q$ makes the problem rather challenging as also additional difficulties arise due to aspects like, e.g., sub-sampling and noise in MRI.
 
As a numerical remedy under such adverse circumstances we propose a projected Levenberg-Marquardt regularized variant of the Gauss-Newton method. 
Analytically, this requires a differentiability result for the map $\theta\mapsto m(\theta)$. 
Furthermore, as for many highly non-linear and non-convex problems, the initialization of the iteration turns out to be crucial. For this initialization, we suggest to use BLIP (or MRF) with a rather coarse dictionary for efficiency purposes, only. 
Overall it turns out that our approach allows to produce more accurate parameter maps in less CPU-time. 

We mention that similar single-step dictionary-free approaches can be found in the recent papers \cite{Sbr_etal18} and  \cite{Sbr_etal17}.
In particular, the model in \cite{Sbr_etal18} abandons the Fourier space character of the data and asks for a relatively large number of data frames which leads to solving a very large non-linear system. As a result, the method is memory and CPU intense. 
The work of \cite{Sbr_etal17} focuses mainly on the experimental design, aiming at optimizing the excitation pulse sequences as well as the repetition times.

\subsection{Structure of the paper}
The rest of the paper is structured as follows:
In Section \ref{sec:Background}, we provide a general background of MRI, particularly to the Bloch equations and MRF.
In Section \ref{sec:Math_MRF}, we relate MRF-type algorithms to inverse problems. We also perform stability analysis for inversion of the Bloch mapping.
Our new integrated physics approach leading to a single-step model in form of a non-linear operator equation is the focus of Section \ref{sec:Physicalmodel}. We analyse the differentiability of the associated operator and show the non-convexity of the Bloch manifold. Subsequently, we discuss several Newton-type methods for its numerical solution. Here, we particularly focus on the case of undersampled and noisy data. In order to illustrate the efficiency of the proposed method for qMRI, numerical tests and comparisons are presented in Section \ref{sec:Numerics}. A short description on solutions of Bloch equations in different cases is given in the Appendix.

\section{Background on MRI and MRF}
\label{sec:Background}

We provide here a brief summary of the principles underlying MRI as they are useful for our purpose of generating an integrated physics-based model for reconstruction; see \cite{Wri97} for more details.  Also, a mathematical description of MRF and the BLIP algorithms is given. 

\subsection{Bloch equations}
\label{subsec:Bloch}
The Bloch equations \cite{Blo46} characterize the key physics in nuclear magnetic resonance.
For the sake of their derivation, let $\Omega$ be a domain in $\mathbb{R}^{2}$ modelling a thin slice of tissue. Every element point (or voxel) of $\Omega$ is denoted by $r$. The main principles of MRI lie in the interaction between an externally applied dynamic magnetic field $B=(B_{x},B_{y},B_{z})^\top$ and the (net or bulk) magnetization which is equal to all the individual dipole moments of the proton spins within a voxel. This net magnetization is  proportional to the hydrogen proton density $\rho$. Correspondingly, letting $m=(m_{x},m_{y},m_{z})^\top$ denote the magnetization per unit density element, the net magnetization in a voxel of density $\rho$ equals $\rho m$.

In the case of a static magnetic field $B_0$, which is  typically regarded to lie in $z$-direction, the net magnetization is aligned to that field with its longitudinal component $m_{z}$ reaching an equilibrium $m_{eq}$. This alignment is not achieved instantaneously but it is controlled by the \emph{longitudinal relaxation time} $T_{1}$ (or $T_{1}(r)$ emphasizing the dependence on a specific voxel). The longitudinal magnetization evolves  according to 
$m_{z}(t)=m_{eq}(1-e^{-(t/T_{1})})$.
Furthermore, the part of the magnetization orthogonal to $B_{0}$, which is called the transverse magnetization $(m_{x},m_{y})^\top$, precesses about the $z$-axis at a frequency equal to $\gamma |B_{0}|$ where $\gamma$ denotes the gyromagnetic ratio. This precession emits an electromagnetic signal which can be detected and measured by the coils of the MR machine. The transverse magnetization decays exponentially at a rate $T_{2}$, the \emph{transverse relaxation time}. 

The overall macroscopic dynamics that dictate the relation between the magnetization $m$, the magnetic field $B$ and the relaxation times $T_{1}$, $T_{2}$, are governed by the Bloch equations, which is a system of linear ordinary differential equations (ODEs):
\begin{equation}
\label{eq:Bloch}
\begin{array}{lll}
\frac{\partial m(t,r)}{\partial t}& = &m(t,r)\times \gamma B(t,r)-\Theta(r) \bullet (m(t,r)-m_{e}),\\
m(0,r) & = & m_0(r),
\end{array}
\end{equation}
where $m_e=(0,0,m_{eq})^\top$ (without loss of generality we assume $m_{e}= (0,0,1)^\top$ in what follows), and ``$\times$'' denotes the outer product between vectors. For the ease of notation we use 
\[ \Theta(r):=(\Theta_1(r) ,\Theta_2(r),\Theta_3(r) )^\top:=
\left(\begin{matrix}
1/T_2(r),
1/T_2(r),
1/T_1(r)
\end{matrix} \right)^\top,
\]
and the operation $\bullet$ in \eqref{eq:Bloch} denotes Hadamard product (component-wise multiplication of vectors).
As introduced above, $m:(0,\tau)\times\Omega  \rightarrow \R^3$, for some time horizon $\tau>0$, denotes the  magnetization in a unit volume per unit proton density, and $m_0$ is a given initial state. Note that the dependence on $r$ is here intrinsic and does not enter the equation. 
As the latter is linear, one can simply multiply \eqref{eq:Bloch} by $\rho$ in order to get the net magnetization.

The total magnetic field $B(t,r)$  can be typically decomposed into
\begin{equation}\label{eq:magnetic_field}
B(t,r)=B_0(r)+ B_1(t,r) + (0,0,G (t)\cdot r)^\top.
\end{equation}
Here $B_0$ denotes the external constant magnetic field that points into the positive $z$ direction, and it is generally assumed to be spatially homogeneous. For the sake of generality, we, however, keep here the dependence on $r$.
The summand $B_1(t,r)= (B_{1,x}(t,r),B_{1,y}(t,r),0)^\top$ corresponds to a radio frequency (RF) pulse, which is sent periodically and lasts only  for a very short time. It is used to excite the magnetization from its equilibrium by turning the magnetization precession away from the direction of the main magnetic field with the so-called flip angle
\[\alpha(t)= \gamma \int_0^t\abs{B_1(s)}ds.\]
These pulses usually last only very briefly compared to $T_1$ and $T_2$.
Therefore, RF sequences can be completely characterized by sequences of flip angles, and time is normally omitted.
The interval between two consecutive pulses is called repetition time ($TR$). As we shall see later in Section \ref{sec:Physicalmodel}, we consider a specific flip angle sequence pattern referred to as \emph{Inversion Recovery balanced Steady State Free Precession} (IR-bSSFP) \cite{Sche99}. Through this choice, the solution of Bloch equations can be simulated by a discrete linear dynamical system; see Section \ref{subsec:Bloch_dynamic}.
In the Appendix, we provide case discussions concerning the discrete Bloch dynamics and the solution of Bloch equations. 
The factor $G(t)$ in \eqref{eq:magnetic_field} is a magnetic gradient field designed to differentiate the point-wise information from the detected signal.

In brief, the measured signal can be expressed by
\[S(t)=\int_{\Omega} \rho(x,y) (T_{x,y} m(t,x,y)) e^{-\mathrm{i}\gamma \abs{B_0}t} e^{-\mathrm{i}\gamma \int_0^t (xG_x +y G_y)d\tau} dxdy ,\]
where $T_{x,y} m :=m_x(x,y)+ \mathrm{i} m_y(x,y)$ stands for the transverse magnetization, and $\mathrm{i}$ is the imaginary unit. Alternatively, one can think of $T_{x,y} m$ as a pair of real-valued components.
The third component of $m$ can usually not be measured due to the position of coils.
Finally, up to a demodulation by $e^{\mathrm{i}\gamma \abs{B_0} t}$, the MR signal $D^{(t)}$ can mathematically be modelled as a collection of coefficients of a Fourier transform of the transverse magnetization, i.e.,
\[ P^{(t)}\mathcal{F}(\rho T_{x,y} m^{(t)})=D^{(t)},\]
where $\mathcal{F}$ denotes the Fourier transform and $P^{(t)}$ a sub-sampling operator.

\subsection{Sub-sampling}
\label{subsec:sub-sampling}

In MRI, and in particular in MRF, one does not wait for the signal to return to equilibrium between two excitation pulses and due to time constraints only a small proportion of the $k$-space is sampled. Reconstruction of the magnetization under such circumstances leads to the occurrence of aliasing artifacts, especially when this reconstruction uses the basic (but fast to apply) pseudo-inverse of $\mathcal{F}$.

In the literature, three different sub-sampling schemes are designed and are often practically employed: spiral, radial  and Cartesian sub-sampling. Each of these corresponds to a different variation in time of the selection gradients $G_{x}$ and $G_{y}$.
In the original version of MRF, the first two patterns were preferred as the associated aliasing artifacts appear to be uncorrelated, respectively, and can be roughly treated as random noise. The latter is not the case for Cartesian sub-sampling; see for instance the numerical examples in \cite{DavPuyVanWia14}. The BLIP method, reviewed in the next section and improving over MRF, however perfectly fits to Cartesian sub-sampling.
As our starting point is the BLIP method,  {we  focus here mainly on Cartesian sub-sampling based on multishot echo-planar imaging (EPI) \cite{Mck93}. Nevertheless,  we also present numerical tests using radial sub-sampling}; see Section \ref{subsec:subsampling} for details.
We note however that these choices are not limiting, and other sub-sampling patterns may be used as well.

\subsection{MRF and BLIP in some detail}
\label{subsec:MRF}

In MRF one initially considers a pre-designed excitation pattern of $L$ flip angles $\{\alpha_{\ell}\}_{\ell=1}^{L}$ separated by a repetition time $TR$. Here, for simplicity,  we consider $TR$ to be constant but this is not necessary. Also a subset $C_{ad}\subset \mathbb{R}^{m}$ of the space of tissue parameters to be estimated is predefined. For the ease of exposition, here we consider $C_{ad}$ to contain admissible $\theta=(T_{1},T_{2})$-values, yielding $m=2$. 
For example, values for $T_{1}$ would typically range from $685$ms (white matter on brain) to $4880$ms (cerebrospinal fluid), with the corresponding range for $T_{2}$ to be $65$ms--$550$ms \cite{Ma_etal13}.  As we shall see below, in our dictionary-free approach we choose $C_{ad}$ to be a convex subset of $\mathbb{R}^{+}\times \mathbb{R}^{+}$, in particular a box, thus admitting values between a minimum and a maximum value. Dictionary based methods then replace $C_{ad}$ by a sufficiently fine discretization yielding $J$ parameter values $\{\theta_{j}\}_{j=1}^{J}$. For simplicity, in this section we write $C_{ad}$ also for the discretization. Using this set of $J$ parameter values, the specific excitation pattern, the sequence of flip angles $\{\alpha_{\ell}\}_{\ell=1}^{L}$ and the repetition time $TR$, one can simulate the Bloch equations by using a discrete linear dynamical system. In this context, the solutions of the Bloch equations are evaluated at discrete times $t_{1}, t_{2},\ldots, t_{L}$; see Section \ref{subsec:Bloch_dynamic} for details.
This generates a dictionary $\mathrm{Dic}(C_{ad})$ of $J$ magnetization responses (i.e., trajectories of the solutions of Bloch equations evaluated at times $t_1,\;t_2,\ldots, t_L$) $\{m^{\theta_{j}}\}_{j=1}^{J}$:
\[\mathrm{Dic}(C_{ad})=\{m^{\theta_{j}}: \theta_{j}\in C_{ad}, \; j=1,\ldots,J \}\subset\left (\left (\mathbb{R}^{3} \right )^{L} \right )^{J}.\]

Next, MR data are collected at the respective $L$ read-out  times. Each component $D^{(\ell)}$ of the data $D=(D^{(1)},\ldots, D^{(L)})$, corresponds to a sub-sampling (resulting by $P^{(\ell)}$) of the Fourier coefficients of the net magnetization $X^{(\ell)}$. Here, the reconstruction of the transverse magnetization image is done via the least square solution and hence these images suffer from aliasing artifacts. This step therefore consists of solving $L$ least square solutions (using the pseudo-inverse Fourier transform $\mathcal{F}^{-1}(P^{(\ell)})^\top$) to obtain $X^{\ast}=(X^{(1)},\ldots, X^{(L)})$, where $X^{(\ell)}:\Omega \to \mathbb{R}^{2}$. Note that, instead of $\mathbb{R}^{2}$, one can also use the complex number representation of the reconstructed magnetization $X$ and the Bloch response $m$, i.e., $m=m_{x}+\mathrm{i}m_{y}$. Observe that in this section $\Omega$ denotes a set of discrete voxels, which in practice are represented by pixels $i:1,\ldots, N$. Summarizing, we have:
\begin{flushleft}
\textbf{Step 1 of the MRF process: Reconstruction of the magnetizations}\\
\emph{Reconstruct the vector of $L$ net magnetizations $X^{\ast}=(X^{(1)},\ldots, X^{(L)})$ by solving
\[X^{(\ell)}\in \underset{X:\Omega\to \mathbb{R}^{2}}{\operatorname{argmin}}\; \|P^{(\ell)} \mathcal{F} X- D^{(\ell)}\|_{2}^{2}\;\;\; \text{ using } \;X^{(\ell)}=\mathcal{F}^{-1}(P^{(\ell)})^\top D^{(\ell)},\quad \ell=1,\ldots,L \]}
\end{flushleft}

The second and final step of MRF identifies the transverse component of $m^{\theta_{j}}$ (denoted by $T_{x,y}m^{\theta_{j}}$) in the dictionary $\mathrm{Dic}(C_{ad})$ that best matches the reconstructed magnetization at every voxel. The desired parameter map $\theta:\Omega \to \mathbb{R}^{2}$ is then obtained by mapping every discrete voxel $i$ to the $\theta$-value that corresponds to the matched $m^{\theta}$, and the reconstructed magnetization sequence at voxel $i$, i.e., $(X_{i}^{\ell})_{\ell=1}^{L}$, contributes with density $\rho_{i}$ that is associated with this particular tissue element. Utilizing normalization and an $\ell_2$-projection onto the discrete Bloch manifold, the best approximation and following density computation yield
\begin{flushleft}
\textbf{Step 2 of the MRF process: Matching of the magnetizations to the dictionary}\\
\emph{For every discrete voxel $i=1,\ldots,N$, compute the projected magnetization $X_{i}=(X_{i}^{\ell})_{\ell=1}^{L}$ according to \[m^{\theta_{j_{i}}}= \underset{m^{\theta}\in \mathrm{Dic}(C_{ad})}{\operatorname{argmin}} \left \|\frac{T_{x,y}m^{\theta}}{\|T_{x,y} m^{\theta}\|_{2}}-X_{i}\right\|_{2}^{2}.\]
Then extract $\{\theta_{j_{i}}\}_{i=1}^{N}=\{(T_{1}(i), T_{2}(i))\}_{i=1}^{N}$ from a look-up table, and compute the density map $\{\rho_{i}\}_{i=1}^{N}$  as
\[\rho_{i}=\frac{\|X_{i}\|_{2}}{\|T_{x,y} m^{\theta_{j_{i}}}\|_{2}}.\] }
\end{flushleft}

One may notice that Step 1 very likely has non-unique minimizers due to sub-sampling.
In \cite{Ma_etal13} the specific minimizer $X^{(\ell)}=\mathcal{F}^{-1}((P^{(\ell)})^\top D^{(\ell)})$ was chosen, which, however, may not be suitable; compare, e.g., \cite{DavPuyVanWia14}. 
In the later work, the algorithm BLIP was introduced as an alternative MRF-approach.
\begin{algorithm}[!h]
\caption{\textbf{BLIP  \cite{DavPuyVanWia14}.}}
\begin{enumerate}
\item Generate a dictionary $\mathrm{Dic}(C_{ad})$.
\item Initialize the magnetization vector $X=0$ and choose an initial step size $\mu_{1}$.
\item For $n=1,2,3,\ldots $ iterate as follows:
\begin{enumerate}
\item For every $\ell=1,\ldots, L$, perform a gradient decent step yielding 
\[\left(X^{(\ell)}\right)_{n+1}=\left (X^{(\ell)}\right )_n - \mu_{n} \mathcal{F}^{-1}(P^{(\ell)})^\top\left(P^{(\ell)}\mathcal{F} ( X^{(\ell)} )_n-D^{(\ell)}\right).\]
\item Project each $(X_{i})_{n+1}=\big( \big (X_{i}^{(\ell)} \big)_{\ell=1}^{L} \big)_{n+1}$ onto the dictionary $\mathrm{Dic}(C_{ad})$ to obtain, as Step 2 in MRF, 
\[\left (m^{\theta_{j_{i}}}\right)_{n+1}=\left (\left ((m^{\theta_{j_{i}}})^{(\ell)}\right )_{\ell=1}^{L}\right)_{n+1} \text{ and } (\rho_{i})_{n+1}\] for every voxel $i,=1,\ldots, N$.
\item For every $\ell=1,\ldots, L$, update $\left (X^{(\ell)} \right)_{n+1}$ as follows
\[\left (X_{i}^{(\ell)} \right)_{n+1}\leftarrow (\rho_{i})_{n+1} \left ((T_{x,y} m^{\theta_{j_{i}}})^{(\ell)}\right)_{n+1},\quad i=1,\ldots, N .\]
\item Update the step size $\mu_{n}$ (see \cite{DavPuyVanWia14} for some rules).
\end{enumerate}
\item Upon termination of the iteration with outcome $X$, as in MRF, construct parameter maps from $X$ by using a look-up table.
\end{enumerate}
\label{alg:BLIP}
\end{algorithm}

 It aims to compute an approximate solution to
\begin{align}
\label{eq:BLIP_in_text}
&\min_{X }  \quad    \norm{P\mathcal{F} T_{x,y}X - D}_{2}^2 , \; \tag{BLIP} 
\text{s.t.  }
  X \in \R^+  \text{Dic}(C_{ad})  \nonumber 
\end{align}
by employing a projected gradient descent method, see Algorithm \ref{alg:BLIP}. Note that in contrast to MRF, BLIP integrates the dictionary constraint into a single minimization step and is shown in \cite{DavPuyVanWia14} to be superior to MRF, in particularly for Cartesian sub-samping.

 {It is worth mentioning here another approach \eqref{eq:LR_BLIP} which has been considered in  \cite{MazWeiTalEld18}. The authors posed the following optimization problem in a discrete setting:
\begin{align}
\label{eq:LR_BLIP}
&\min_{X }  \quad    \norm{P\mathcal{F} T_{x,y}X - D}_{2}^2  + \beta \text{rank}(X), \; 
\text{s.t.  }
  X \in \R^+  \text{Dic}(C_{ad}) , 
\end{align}
where $X\in \C^{N\times L}$ is the vectorized version of the magnetization, $N$ is the total number of voxels, and $\text{rank}(X)$ denotes the rank of the matrix $X$, which is a non-convex penalty.
With the second constraint in \eqref{eq:LR_BLIP}, the rank penalty enforces also a constraint on the total number of unique entries in the dictionary which are used to represent $X$.
In order to avoid difficulties on the non-convexity, the authors considered a convex relaxations of \eqref{eq:LR_BLIP}. There they used a penalty on the nuclear norm of $X$ which is the sum of the singular values of $X$, instead of $\text{rank}(X)$.}

\section{MRF as an inverse problem and its stability analysis}
\label{sec:Math_MRF}
\subsection{Towards a coupled inverse problem}
For the sake of generality, our starting point is the time continuous version of the Bloch equations. In order to fix our setting, let
$\mathcal{Y}:=[L^{2}(\Omega)]^{3}$ and $\mathcal{Z}:=[L^{\infty}(\Omega)]^{3}$.
The initial magnetization is given by $m_{0}\in \mathcal{Y}$, and $B\in L^{\infty}( 0,\tau; \mathcal{Z})$ denotes a given external magnetic field for some time horizon $\tau>0$. We recall the Bochner space
 \[L^{\infty}( 0,\tau; \mathcal{Z}):=\{f: (0,\tau)\to \mathcal{Z}:\; \|f\|_{L^{\infty}( 0,\tau; \mathcal{Z})}<+\infty\} ,\]
with $\|f\|_{L^{\infty}( 0,\tau; \mathcal{Z})}=\underset{0 < t< \tau}{\operatorname{ess \sup}} \;\|f(t)\|_{\mathcal{Z}}$. The space 
$L^{1}(0,\tau; \mathcal{Y})$ is defined similarly. The space $W^{1,1}(0,\tau;\mathcal{Y})$
consists of all the functions $f:(0,\tau)\to \mathcal{Y}$ such that both $f$ and $\frac{\partial f}{\partial t}$ belong to $L^{1}(0,\tau; \mathcal{Y})$. We refer to \cite{Eva10} for more on Lebesgue, Sobolev and Bochner spaces.

A natural space for the parameter $\theta=(T_{1},T_{2})$ is $ [L^{\infty}(\Omega)]^{2}$, and we also require this parameter to be bounded uniformly away from zero. Consequently, we have $\Theta=(1/T_{2},1/T_{2},1/T_{1})^\top\in [L^{\infty}(\Omega)]^{3}$, as well.
Finally, recall that $m_{e}\equiv(0,0,1)^{\top}$. 

For our further analysis, it is convenient  to introduce the operator 
\[\mathcal{B}_{m_{0},B}: [L^{\infty}(\Omega)]^{2}\to \{m:(0,\tau)\to \mathcal{Y}\},\]
where $\mathcal{B}_{m_{0},B}(\theta)$ denotes the solution mapping of the Bloch equations \eqref{eq:Bloch} up to time $\tau$. 
Equipped with this notation, we now state the following family of inverse problems which represents a continuous version of the MRF process:
\begin{itemize}
\item[-] Problem 1: For some $t_{\ell}\in (0,\tau)$, $\ell=1,\ldots,L$, in order to obtain $X^{(t_{\ell})}\in L^{2}(\Omega)$ solve the linear equation
\begin{equation}
\label{eq:problem1}
P^{(t_{\ell})} \mathcal{F}X^{(t_{\ell})} =D^{(t_{\ell})},
\end{equation}
where $D^{(t_{\ell})}\in [L^{2}(K)]^2$, $K$ is a bounded frequency domain which is usually called $k$-space, $\mathcal{F}:[L^{2}(\Omega)]^2\to [L^{2}(K)]^2$, $P^{(t_{\ell})}: [L^{2}(K)]^2\to [L^{2}(K)]^2$.

\item[-] Problem 2: For  every $r\in \Omega$, to obtain $\theta=\theta(r)\in \R^{+}\times \R^{+}$ solve
\begin{equation}
\label{eq:problem2}
\rho(r)T_{x,y}(\mathcal{B}_{m_{0},B}(\theta))(\cdot,r)=X^{(\cdot)}(r),
\end{equation}
\end{itemize}
where $\rho\in L^{\infty}(\Omega)$, and $T_{x,y}$ is the transverse projection.
Note that, strictly speaking, the coupling of \eqref{eq:problem1} and \eqref{eq:problem2} makes sense only when $P^{(t_\ell)}=\text{Id}$, i.e., there is no sub-sampling. This is because of the fact that under sub-sampling, uniqueness of solutions for \eqref{eq:problem1} is not guaranteed, and  $X^{(\cdot)}(r)$ may not belong to the Bloch manifold.

Here Problem 1 corresponds to the first step in MRF and aims to invert the Fourier transform for sub-sampled (and potentially noisy) data. 
This type of problem is the central mathematical problem in standard MRI and has been extensively studied in the literature. In particular, variational methods e.g., sparse regularization methods and optimal weighted total variational methods, have been successfully applied towards that \cite{GuoSonZha17,HinRauWuLan17, knoll2011second, LusDonSanPau08}, to mention only a few recent results.

In view of the parameter identification problem involving the Bloch equations in the second step of MRF
we now focus on equation \eqref{eq:problem2}.  But for the sake of ease of demonstration, we neglect the effect of the density map $\rho$ and the transverse projection operator $T_{x,y}$, i.e., we study 
\begin{equation}\label{eq:inv_Bloch}
\mathcal{B}_{m_{0},B}(\theta)=m. 
\end{equation}

\subsection{Stability analysis on inverting the Bloch mapping}
\label{subsec:Analysis}
With the aim of quantifying the influence of inaccuracies or noise in the solution of \eqref{eq:problem1} on solving \eqref{eq:problem2}, we next analyse stability of \eqref{eq:inv_Bloch}. This is of relevance for both, MRF and BLIP.

In order to simplify the discussion, in this section we consider the time domain $(0,\tau)$ to be the period between two consecutive pulses. From a modelling point of view,  $m_{0}$ will be the magnetization right after the first pulse, i.e., after the application of the flip angle displacement and, $m(\tau,\cdot)$ will be the magnetization right before the next pulse.
In that case the magnetic field $B\ne 0$ is considered to be time independent which means that $B$ is a constant function in $ L^{\infty}( 0,\tau; \mathcal{Y})$ with respect to time and with (possibly spatially varying) values in $\mathcal{Y}$. Also, the effect of the gradient field $G$ is ignored here, as it only encodes the MRI signal.

From a classical result for evolutionary equations in Banach spaces (see e.g. \cite[Proposition 3.3]{BenPraDelMit07}), we infer existence of a solution $m$ of \eqref{eq:Bloch} 
in $W^{1,1}(0,\tau;\mathcal{Y})$. In fact, $m$ enjoys even higher regularity, but for our purposes $W^{1,1}(0,\tau;\mathcal{Y})$ turns out to be sufficient. Hence, we consider
$\mathcal{B}_{m_{0},B}: [L^{\infty}(\Omega)]^{2}\to W^{1,1}(0,\tau;\mathcal{Y})$
for given $m_{0}$ and $B$.

Given the existence of solutions, our further analysis relies on the following assumptions:
\begin{assumption}
\label{ass:feasi_domain}
Let $\theta(r) \in C_{ad}$ for all $r\in \Omega$, where $C_{ad}\subset \R^+ \times \R^+ $ denotes a feasible domain which is convex and bounded away from zero.
\end{assumption}

For our next assumption, we define the range of the Bloch map, i.e.,
\[R(\mathcal{B}_{m_{0},B}):=\left\{m:
\;m=\mathcal{B}_{m_{0},B}(\theta)\,\text{ with }  \theta\in [L^{\infty}(\Omega)]^{2} \text{ and } \theta(r)\in C_{ad} \text{ for all } r\in \Omega \right\}.\]

\begin{assumption}
\label{ass:echotime}
Let 
\begin{equation}
\label{eq:regularity}
m\in R(\mathcal{B}_{m_{0},B})\subset W^{1,1}(0,\tau;\mathcal{Y}),
\end{equation}
be a solution of the Bloch equations \eqref{eq:Bloch}.
Then the quantity $(\omega_{\tau}^{1}(r),\omega_{\tau}^{2}(r),\omega_{\tau}^{3}(r))^\top:=\int_0^\tau m(t,r)dt -  m_e \tau$ is bounded away from zero, i.e., 
there is a constant $c_{\tau}>0$ such that 
\begin{equation}
\label{eq:tau_condition}
 \inf_{r\in \Omega} \abs{ \omega_{\tau}^{i}(\tau)} \geq c_\tau,\quad \text{  for } i=1,2,3.
\end{equation}
\end{assumption}

\begin{remark}
Assumption \ref{ass:feasi_domain} implies no factual limitation in practice. Assumption \ref{ass:echotime} is also justified in practice as we consider $(0,\tau)$ to be the time between two consecutive pulses which roughly equals to repetition time.
In this period, the net magnetization always satisfies $m_x>c_{\tau}'>0$, $m_y>c_{\tau}''>0$, and $m_z<m_e$, and these give the estimate  \eqref{eq:tau_condition}.
Since, in an MRI experiment, the time domain consists of the repetition of periodic radio pulses, (the excitation time of the pulse is usually very short) our assumption is satisfied during the entire experiment.
\end{remark}

\begin{theorem}
\label{thm:uniqueness}
Let Assumption \ref{ass:echotime} hold,  the magnetic field satisfy $B\neq 0$, and let $m=  \mathcal{B}_{m_{0},B}(\theta)$ for some $\theta$. Then the $\theta$-value associated with $m$ is unique.
\end{theorem}
\begin{proof}
Observe that by integrating the Bloch equations over the time domain $(0,\tau)$, we have
\begin{equation}
\label{eq:Bloch2}
\Theta(r)\!=\!\left(\! m(0,r)- m(\tau,r )+\!\!\int_{0}^{\tau} m(t,r)\times \gamma B(t,r) dt\!\right)\mathbf{./}\omega_\tau,
\end{equation}
where ``./'' denotes a component-wise quotient of vectors.
Note that the integrals are well-defined, since for almost every $r$, $m(r,\cdot)\in L^{1}(0,\tau)$. Also, due to Assumption \ref{ass:echotime} we have $\omega_\tau^i\ne 0$ for $i\in\{1,2,3\}$. The uniqueness of $\theta=(\frac{1}{\Theta_3},\frac{1}{\Theta_1})^\top$ follows readily. 
\end{proof}

We immediately have the next corollary.
\begin{corollary}
\label{cor:bijection}
Let Assumption \ref{ass:echotime} hold, and $B\ne 0$. Then the  Bloch mapping satisfies
\[\mathcal{B}_{m_{0},B}(\theta_{1})=\mathcal{B}_{m_{0},B}(\theta_{2}) \iff \theta_{1}=\theta_{2}.\]
\end{corollary}
\begin{proof}
For every fixed $\theta(r)\in C_{ad}$, the Bloch mapping is well-defined under Assumption \ref{ass:echotime}.
The mapping is injective by Theorem  \ref{thm:uniqueness}. The other direction follows from
the Picard--Lindel\"of Theorem (see, e.g., \cite{Tes12}). 
\end{proof}
\begin{remark}
The uniqueness result indicates that enforcing the magnetization function to be in the range of a Bloch mapping also guarantees a unique parameter.
This  explains the idea behind BLIP which aims at an improved solution (when compared to MRF) from undersampled data by using projection steps onto the Bloch manifold.
\end{remark}

The main stability result of this section is stated next.

\begin{theorem}
\label{thm:stability}
Let  the Assumption \ref{ass:feasi_domain} be satisfied, and let $m,m^{\delta}\in R(\mathcal{B}_{m_{0},B})$ with corresponding parameters $\theta, \theta^{\delta}$.
If Assumption \ref{ass:echotime} holds for both $m$ and $m^\delta$, and\\
 $\norm{m-m^\delta}_{W^{1,1}(0,\tau;\mathcal{Y})}\leq \delta$ for $\delta>0$, 
 then we have
\[\norm{\theta-\theta^\delta}_{[L^1(\Omega)]^{2}}\leq C(\tau,\theta,B) \delta,\]
where $C(\tau,\theta,B)$ is a constant depending on $\tau$, $\theta$ and $B$, but not on $\delta$.
\end{theorem}
\begin{proof}
Using equation \eqref{eq:Bloch2} with the obvious definition of $\omega_\tau^\delta$, we have
\begin{eqnarray*}
\Theta-\Theta^\delta
&= &
\left( m_{0}- m(\tau,r )+\int_{0}^{\tau} m(t,r)\times \gamma B(t,r) dt\right) ./ \omega_\tau \\
& & - \left( m_{0}^{\delta}- m^\delta(\tau,r )+\int_{0}^{\tau} m^\delta(t,r)\times \gamma B(t,r) dt\right)./\omega_\tau^\delta \\
& =&\Theta \bullet \left((\omega_\tau^\delta - \omega_\tau) ./ \omega_\tau^\delta\right)
- \left(m_{0}^{\delta} -m_{0} - m^\delta(\tau,r) + m(\tau,r)\right) ./ \omega_\tau^\delta \\
& & -\left( \int_0^\tau m^\delta(t,r)\times \gamma B(t,r) - m(t,r)\times \gamma B(t,r)dt\right) ./\omega_\tau^\delta.
\end{eqnarray*}
Note that $\int_0^\tau m^\delta(t,r)-m_{e}^\delta(r)dt=\int_0^\tau m^\delta(t,r)dt -m_e^\delta\tau$,
and $m_e^\delta=m_e.$
Due to \eqref{eq:tau_condition}, we have
\[\inf_{r\in \Omega} \abs{\int_0^\tau m(t,r)-m_{e}(r)dt }\geq c_\tau,\quad\text{and}\quad \inf_{r\in \Omega} \abs{\int_0^\tau m^\delta(t,r)-m_{e}(r)dt }\geq c_\tau. \]
As a consequence, we obtain the estimate 
\begin{align*}
\norm{\Theta-\Theta^\delta}_{[L^1(\Omega)]^{3}}
& \leq \frac{1}{c_\tau} \int \abs{\Theta(r)\bullet \int_0^\tau ( m^\delta(t,r)  - m(t,r) )dt }dr\\
& \quad +\frac{1}{c_\tau}\int \abs{\int_0^\tau \left(\frac{\partial m(t,r)}{\partial t}-  \frac{\partial m^{\delta}(t,r)}{\partial t} \right)d t  }dr  \\
& \quad + \frac{1}{c_\tau}\int \abs{ \int_0^\tau (m^\delta(t,r) - m(t,r) )\times \gamma B(t,r)dt}dr\\
& \le    \frac{1}{c_\tau} C \|m^{\delta}-m\|_{L^{1}(0,\tau;[L^{1}(\Omega)]^{3})}
+\frac{1}{c_\tau}  \left\|\frac{\partial m^{\delta}}{\partial t}-\frac{\partial m}{\partial t}\right\|_{L^{1}(0,\tau;[L^{1}(\Omega)]^{3})}\\
& \quad +\frac{1}{c_\tau} C  \|m^{\delta}-m\|_{L^{1}(0,\tau;[L^{1}(\Omega)]^{3})},
\end{align*}
with generic constants $C$ depending on $\gamma$, $B$ and $\Theta$.
Here we have used the facts that $\Theta\in [L^{\infty}(\Omega)]^{2}$ and the outer product with $B(t,r)$ can be written as the application of a linear operator with bounded norm (in $t$ and $r$) as $B$ is bounded in $L^{\infty}(0,\tau;\mathcal{Z})$ and also independent of time.

As $(\theta_1,\theta_2)=(\frac{1}{\Theta_3},\frac{1}{\Theta_1})$
and by Assumption \ref{ass:feasi_domain} we can find a constant $C_{1}>0$ such that
\[\norm{\theta-\theta^\delta}_{[L^1(\Omega)]^{2}} \leq C_{1} \norm{\Theta-\Theta^\delta}_{[L^1(\Omega)]^{3}}. \]
This follows from the fact that the function $h:[a,b]\to \R$, with $h(\beta)=1/\beta$, is Lipschitz for $0<a<b<\infty$.
The proof is completed by combining the last two inequality relations and the fact that the $[L^{1}(\Omega)]^{3}$ norm is bounded by $[L^{2}(\Omega)]^{3}$ norm.
\end{proof}

The above result can be interpreted as follows. Theorem \ref{thm:stability} shows that the inverse problem  \eqref{eq:inv_Bloch} is well-posed by restricting the right hand side to the range of the Bloch mapping.
That is, if the reconstructed magnetization is in the Bloch manifold (more precisely the positive cone of the manifold), then the values of the tissue parameters $\theta$ recovered from the dictionary should in principle be not too far away from the exact solutions.

The analytical properties of the Bloch mapping and its inverse not only support the application of MRF-type schemes, but they also motivate us to find yet more accurate solution techniques for quantitative MRI. This is our target in the next section.

\section{An integrated physics-based method for qMRI}
\label{sec:Physicalmodel}
We now propose a method for qMRI that integrates the physics model into the reconstruction process. In contrast to the previously discussed two-step procedures, it consists of a single step only.

On an abstract level, our model is associated with the non-linear operator equation
\begin{equation}\label{eq:Q_equation}
Q(\f x) =D,
\end{equation}
where $\f x(r)= (\rho(r), \theta(r))\in \tilde{C}_{ad}:=\R^+\times C_{ad}$ for all $r\in \Omega$, $D$ is the acquired MRI signal, and the qMRI-operator $Q$ is defined by
\begin{equation}
\label{eq:Q_function}
Q(\f x):= P\mathcal{F}(\rho T_{x,y}M(\theta)) .
\end{equation}
It  integrates the Bloch mapping within the data acquisition procedure.

Anticipating our subsequent development, $M(\theta)$ represents the discrete Bloch dynamics, which corresponds to the time continuous version $m(\theta)$ previously discussed. 
\subsection{Bloch mapping as discrete dynamics}
\label{subsec:Bloch_dynamic}
With the aim of employing a fast imaging protocol for absolute quantification of $T_1$ and $T_2$ post-contrast (e.g., upon administering Gadolinium (Gd)), we focus here on Inversion Recovery balanced Steady-State Free Precession (IR-bSSFP) flip angle sequence patterns; see \cite{Sche99} and compare also \cite{Karm14}. IR-bSSFP is a specific MRI excitation pulse sequence widely used in applications and it allows for a simple approximation of the solutions of the Bloch equations at the read out times.
In our subsequent analysis and numerical examples, we always use the associated discrete dynamics approximating the continuous Bloch equations.

To simplify the presentation, we will ignore the factor of off-resonance and only consider the homogeneous case of the flip angles and off-resonance frequency. 
In this case, the magnetization after each $n$-th excitation pulse is simulated by the following recursion formula \cite{Sche99}
\begin{equation}
\label{eq:IR_bSSFP}
\left\{\begin{array}{lll}
M_\ell  &= & E_1(TR_{\ell},\theta)R_{\phi_\ell} R_x(\alpha_{\ell})R_{\phi_\ell}^\top M_{\ell-1} +E_2(TR_{\ell},\theta)M_e,   \\
M_e  &= &(0,0,1)^\top ,\\
M_0  &= & -M_e=(0,0,-1)^\top.
\end{array}\right.
\end{equation}
Here $\{\alpha_{\ell}\}_{\ell=1}^{L}$ and $\{TR_{\ell}\}_{\ell=1}^{L}$ are the flip angles and repetition time sequences, and $\{M_\ell\}_{\ell=1}^{L}$ are the magnetizations at the middle of each $TR_{\ell}$ time interval.
Moreover we denote
\[E_1(TR_{\ell},\theta)=\left(
\begin{array}{ccc}
e^{-\frac{TR_{\ell}}{T_2}} & 0 & 0\\
 0  &e^{-\frac{TR_{\ell}}{T_2}} & 0\\
  0 & 0& e^{-\frac{TR_{\ell}}{T_1}} 
\end{array}\right),
\quad
\text{  }\;\;
E_2(TR_{\ell},\theta) = \left(1-e^{-\frac{TR_{\ell}}{T_1}}\right) 
\]
and also
\[R_{\phi_\ell}=\left(
\begin{array}{ccc}
\cos(\phi_\ell) & \sin(\phi_\ell) & 0\\
-\sin(\phi_\ell)  &\cos(\phi_\ell) & 0\\
  0 & 0& 1 
\end{array}\right)
\text{ and }\;\;
R_x(\alpha_{\ell})=\left(
\begin{array}{ccc}
1& 0 & 0 \\
0& \cos(\alpha_{\ell}) & \sin(\alpha_{\ell}) \\
0& -\sin(\alpha_{\ell})  &\cos(\alpha_{\ell})
\end{array}\right).
\]
The angle $\phi_\ell$ denotes a phase shift by the gradient magnetic fields \cite{Sche99} and is assumed to be known.

Writing \eqref{eq:IR_bSSFP} in a compact form, we are able to derive the evolution of the discrete system for the magnetization vectors
\begin{align}
\label{eq:Bloch_discrete}
M_\ell = &\left(\prod_{k=1}^\ell E_1(TR_{k},\theta) R(\alpha_{k})\right) M_0 +   E_2(TR_{\ell},\theta)M_e \\
 & \quad +  \sum_{k=1}^{\ell-1} \left( E_2(TR_{k},\theta)\prod_{j=k+1}^\ell E_1(TR_{j},\theta)R(\alpha_{j})\right)M_e, \nonumber
\end{align}
where we use the matrix notation $R(\alpha_{\ell}):= R_{\phi_\ell} R_x(\alpha_{\ell})R_{\phi_\ell}^\top$.
Note that \eqref{eq:Bloch_discrete} establishes a mapping between $\theta$ and $\{M_\ell\}_{\ell=1}^{L}$ yielding a discrete (in time) version of the operator $\mathcal{B}_{m_{0},B}$ associated with the IR-bSSFP pulse sequence.
For utilizing Gauss-Newton-type algorithms for solving \eqref{eq:Q_equation}, it is of interest to study differentiability of this mapping.
This and further properties are therefore the subjects of the following section.

\subsection{Properties of  the Bloch mapping and the qMRI-operator}
\label{subsec:differentiability}

Consider the discrete Bloch mapping as defined in \eqref{eq:IR_bSSFP}:
\[M: V \to \left[\mathcal{Y}:=[L^{2}(\Omega)]^{3}\right]^{L},\quad M(\theta):=\{M_\ell(\theta)\}_{\ell=1}^{L},\]
where $V$ is the open subset of $[L^{\infty}(\Omega)]^{2}$ that consists of all functions with strictly positive values almost everywhere.
 {Note here and later, we use the notation $\{\cdot\}$ to represent vector sequences.
Moreover, for small positive real value $\sigma$, we utilize $o(\sigma)$ to denote $\frac{o(\sigma)}{\sigma}\to 0$ as $\sigma\to 0$.}

\begin{proposition}
\label{prop:dis_bloch}
Let $\{M_\ell(\theta)\}_{\ell=1}^{L}$ be the sequence given in \eqref{eq:Bloch_discrete}, $\{\alpha_{\ell}\}_{\ell=1}^{L}$ the sequence of flip angles with $\alpha_\ell\in (0,\pi)$ for every $\ell=1,\ldots, L$, and $\{TR_{\ell}\}_{\ell=1}^{L}$ the sequence of repetition times with $TR_{\ell}>0$ for every $\ell$. 
Given $M_0\in \mathcal{Y}$, then the following statements hold true:
\begin{enumerate}
\item[(i)]  $M$ is Fr\'echet differentiable with bounded derivative. 
Moreover, given sufficiently small $h\in [L^{\infty}(\om)]^{2}$ we have the general estimate for all $q\geq 2$ and for $q=+\infty$:
\begin{equation}
\label{eq:Lq_estimate_M}
\norm{M(\theta+h)-M(\theta) - M'(\theta)h}_{[\mathcal{Y}]^L}= o\left(\norm{h}_{[L^{q}(\om)]^{2}}\right).
\end{equation}
In addition if $M_0\in [L^{\infty}(\om)]^3$, then there exist some constant $C$ independent of $h$ for the following estimate
\begin{equation}
\label{eq:L2_bound_M}
\norm{ M'(\theta)h}_{[\mathcal{Y}]^L}\leq C \norm{h}_{L^2(\om)]^{2}}.
\end{equation}
\item[(ii)]  Let $M_e=(0,0,1)^\top$ and either $M_0= M_e$ or $M_0= -M_e$, then the operator $M$ is injective, i.e., given $\theta^a,\;\theta^b\in V$, we have
\[ M(\theta^{a}) = M(\theta^{b}) \Longrightarrow  \theta^{a}= \theta^{b}. \]
\end{enumerate}
\end{proposition}

\begin{proof}
Due to the recursive nature of $M_\ell$, it suffices to analyse $M_{1}$:
\[M_1(\theta) =E_1(TR_1,\theta)R_\phi R_x(\alpha_{1})R_\phi^\top M_0 +E_2(TR_1,\theta)M_e.\]
\item[(i)] 
We start by considering the differentiability of $M_1(\theta)$.
This is readily derived when using the differentiability of $x\mapsto e^{-\frac{TR}{x}}$ for $x>0$.
We denote by $M^\prime_1(\theta)$ the Fr\'echet derivative of the map $M_{1}$ evaluated at $\theta$, that is $M_{1}'(\theta): [L^{\infty}(\Omega)]^{2}\to \mathcal{Y}$ bounded, linear such that
\begin{equation}\label{eq:M_frechet}
\lim_{h\to 0}\frac{\|M_{1}(\theta+h)-M_{1}(\theta) - M_{1}'(\theta)h\|_{\mathcal{Y}}}{\|h\|_{[L^{\infty}(\Omega)]^{2}}} =0.
\end{equation}

To simplify the formulas, for every $\ell=1,\ldots, L$, we denote
\[U_1(\ell):= \left(\begin{array}{ccc}
0 & 0 & 0\\
 0  &0& 0\\
  0 & 0& \frac{TR_{\ell}}{(T_1)^2} e^{-\frac{TR_{\ell}}{T_1 }} 
\end{array}\right)R(\alpha_{\ell})  ,\]
and
\[
U_2(\ell):= \left( \begin{array}{ccc}
\frac{TR_{\ell}}{(T_2)^2}e^{-\frac{TR_{\ell}}{T_2 }} & 0 & 0\\
 0  &\frac{TR_{\ell}}{(T_2)^2}e^{-\frac{TR_{\ell}}{T_2 }}  & 0\\
  0 & 0& 0 
\end{array}\right) R(\alpha_{\ell}). \]
We compute
\begin{equation}
\label{eq:M1_derivative}
M^\prime_1(\theta)
=\left(
M^\prime_{1,1}(\theta), \;  M^\prime_{1,2}(\theta)
\right):=
\left(
U_1(1) M_0  - \frac{TR_1}{T_1^2}e^{-\frac{TR_1}{T_1 }}  M_e, \;
U_2(1) M_0 
\right).
\end{equation}
Note that $M'_{1}(\theta)\in \mathcal{Y}\times \mathcal{Y}=[L^{2}(\Omega)]^{3}\times [L^{2}(\Omega)]^{3}$. It can be regarded as a bounded linear operator from $[L^{\infty}(\om)]^{2}\to \mathcal{Y}\times\mathcal{Y}$ which  is defined for every $h=(h_{1},h_{2})\in [L^{\infty}(\om)]^{2}$ as
\begin{align*}
M_{1}'(\theta)h
&=M_{1,1}'(\theta) h_{1}+ M_{1,2}'(\theta)h_{2}\\
&:=\left ([M_{1,1}'(\theta)]_{x} h_{1},[M_{1,1}'(\theta)]_{y}h_{1},[M_{1,1}'(\theta)]_{z}h_{1} \right)\\
&\qquad +\left ([M_{1,2}'(\theta)]_{x} h_{2},[M_{1,2}'(\theta)]_{y}h_{2},[M_{1,2}'(\theta)]_{z}h_{2} \right),
\end{align*}
where $[\cdot]_x$, $[\cdot]_y$, $[\cdot]_z$ denote components of a vector. The multiplication of $L^{2}(\om)$- and $L^{\infty}(\om)$-functions is understood in a pointwise sense, and the resulting product is in $L^{2}(\om)$.
Using the fact that $(e^{-\frac{TR}{x}})'=\frac{TR}{x^2}e^{-\frac{TR}{x}}$ is Lipschitz continuous over $x\in (0,\infty)$ for every fixed $TR>0$ (actually the derivative of every order of $e^{-\frac{TR}{x}}$ is Lipschitz continuous), then the following pointwise estimate holds true:
\begin{equation}\label{eq:point_derivative}
\abs{M_{1}(\theta(r)+h(r))-M_{1}(\theta(r)) - M_{1}'(\theta(r))h(r)} \leq C \abs{h(r)}^2,\; \text{ for all } r\in \om,
\end{equation}
where $C$ is a constant independent of $h$ and $\theta$, as well independent on $r$.
Note that 
\[\norm{\abs{h}^2}_{L^2(\om)}\leq C_q \norm{\abs{h}^2}_{L^q(\om)},\quad \text{ for all } \;q\geq 2 \text{ and } q=+\infty ,\]
where $C_q$ is a also constant independent of $h$ and $\theta$. 
Then  for all $h\in [L^{\infty}(\om)]^2$ sufficiently small
\[ \frac{\|M_{1}(\theta+h)-M_{1}(\theta) - M_{1}'(\theta)h\|_{\mathcal{Y}}}{\|h\|_{[L^{q}(\om)]^{2}}} \leq   \frac{C\norm{\abs{h}^2}_{L^2(\om)}}{\norm{h}_{[L^{q}(\om)]^{2}}}\leq \frac{CC_q\norm{\abs{h}^2}_{L^q(\om)}}{\norm{h}_{[L^{q}(\om)]^{2}}},
\]
and with this we get \eqref{eq:Lq_estimate_M} for $M_1$.

The Frech\'et differentiability of $M_1$ in the space $[L^{\infty}(\om)]^2$ is then a consequence of the above estimate, where the constant $C_q=C_\infty=\abs{\Omega}$. In this case, it then implies  \eqref{eq:M_frechet} which gives us the conclusion.

The derivative of $M_\ell(\theta)$ for $\ell>1$ can then be calculated by applying the chain rule to the recursion formula \eqref{eq:IR_bSSFP}, i.e., 
\begin{equation}
\label{eq:Mn_derivative}
M^\prime_\ell(\theta)=\left(
\begin{array}{ll}
\left(U_1(\ell) M_{\ell-1}(\theta)+ E_1(TR_{\ell},\theta)R(\alpha_{\ell}) M^\prime_{\ell-1,1}(\theta) - \frac{TR_{\ell}}{T_1^2}e^{-\frac{TR_{\ell}}{T_1 }}  M_e\right)^\top\\
\left(U_2(\ell) M_{\ell-1}(\theta) +  E_1(TR_{\ell},\theta)R(\alpha_{\ell})  M^\prime_{\ell-1,2}(\theta)\right)^\top
\end{array}\right)^\top.
\end{equation}
We get the boundedness of the derivatives because all the quantities $U_a(\ell)$, $R(\alpha_{\ell})$, $E_1(TR_{\ell},\theta)$ and $e^{-\frac{TR_{\ell}}{T_a }} \frac{TR_{\ell}}{T_a^2} $ for $a=1,2$ and  $\ell=1,\ldots,L$ are uniformly bounded.
If in addition we have $M_0\in [L^\infty(\om)]^3$, the iteration \eqref{eq:Bloch_discrete} will assure that $M_\ell\in [L^\infty(\om)]^3$. Then the estimate \eqref{eq:L2_bound_M} immediately follows.

\item[(ii)] 
We show that the map $M_{1}:[L^{\infty}(\Omega)]^{2}\to \mathcal{Y}$ is injective
for some non-zero $\alpha_1$ and $TR_1$.
We first note that $R=R(\alpha_1):=R_\phi R_x(\alpha_1)R_\phi^\top$ is unitary,
and $E_1(TR_1,\theta)$ and $E_2(TR_1,\theta)$ are contraction operators.
Assume now that $M_1(\theta^a)=M_1(\theta^b)$ for $\theta^a\neq \theta^b$. Then we have
\[\big (E_1(TR_1,\theta^a)-E_1(TR_1,\theta^b)\big)RM_0 +(E_2(TR_1,\theta^a)-E_2(TR_1,\theta^b))M_e=0\quad\text{in }\mathcal{Y}.\]
Assume further that $T_{1}^{a}\neq T_{1}^{b}$, then in those points of $\Omega$ where this occurs we have (suppressing spatial dependence $r$)
\begin{equation}\label{Me_M0}
M_e=\left(
\begin{array}{ccc}
\frac{e^{-\frac{TR_1}{T_2^a}} -e^{-\frac{TR_1}{T_2^b}}}{e^{-\frac{TR_1}{T_1^a}}-e^{-\frac{TR_1}{T_1^b}}} & 0 & 0\\
 0  &\frac{e^{-\frac{TR_1}{T_2^a}}-e^{-\frac{TR_1}{T_2^b}} }{e^{-\frac{TR_1}{T_1^a}}-e^{-\frac{TR_1}{T_1^b}}}& 0\\
  0 & 0& \frac{e^{-\frac{TR_1}{T_1^a}} -e^{-\frac{TR_1}{T_1^b}} }{e^{-\frac{TR_1}{T_1^a}}-e^{-\frac{TR_1}{T_1^b}}}
\end{array}\right) R M_0.
\end{equation}
Suppose now that $M_0=- M_e$. Then, since $R$ is unitary, \eqref{Me_M0} is satisfied if and only if $T_1^a=T_2^a$,  $T_1^b=T_2^b$,  $-RM_e=M_e$ and $\alpha_1=\pi$. 
This, however, contradicts $\alpha\in (0,\pi)$. The case $M_{0}=M_{e}$ is similar. 
If $T_{1}^{a}= T_{1}^{b}$ but $T_{2}^{a}\neq T_{2}^{b}$, then one uses the inverse relation of \eqref{Me_M0} and arrives at the same conclusion.
Thus we have injectivity for $M_{1}$ and hence also of $M$.
\end{proof}

Regarding non-convexity of the Bloch manifold we have the following result.
\begin{proposition}
\label{prop:dis_bloch_non_convex}
Suppose that the assumptions of Proposition \ref{prop:dis_bloch} hold true. Furthermore,
let the operator $M$ be restricted to some feasible set $\mathcal{C}_{ad}$ which is connected and convex: 
\[\mathcal{C}_{ad}:=\{\theta\in [L^{\infty}(\Omega)]^{2}\; |\; \theta(r)\in C_{ad}, \text{ for every }\; r\in \Omega \},\]
where $C_{ad}$ is a convex subset of $\R^{+}\times \R^{+}$ (typically a box) which is bounded and bounded away from zero. 
Then the image $M[\mathcal{C}_{ad}]$ of $M:\mathcal{C}_{ad} \to \mathcal{Y}^{L}$ is a non-convex subset of $\mathcal{Y}^{L}$.
\end{proposition}
\begin{proof}
Suppose that $M[\mathcal{C}_{ad}]$ is a convex subset of $\mathcal{Y}^{L}$.
Then, for arbitrary $\theta^a \neq \theta^b\in \mathcal{C}_{ad}$ and for every $\lambda \in (0,1)$, there exist $\theta^\lambda\in \mathcal{C}_{ad}$ such that
\begin{equation}\label{eq:convexity}
\lambda M_\ell(\theta^a)+ (1-\lambda) M_\ell(\theta^b)= M_\ell(\theta^\lambda) \text{ for all } \ell\in \set{1,\ldots,L}.
\end{equation}
We focus on the first two components $M_{1}, M_{2}\in \mathcal{Y}$ and
recall
\begin{align*}
M_1(\theta) =& E_1(TR_1,\theta)R(\alpha_1) M_0 +E_2(TR_1,\theta)M_e,\\
M_2(\theta) =& \left(\prod_{k=1}^2 E_1(TR_k,\theta) R(\alpha_k)\right) M_0 \\
            &\quad +\ \left( E_2(TR_1,\theta)+ E_2(TR_2,\theta) E_1(TR_1,\theta)R(\alpha_1)\right)M_e.
\end{align*}
 {With some straightforward calculations and simplifications}, the convexity condition \eqref{eq:convexity} can be equivalently written as the following system of equations (where $\theta^a=(T^a_1,T^a_2)$, $\theta^b=(T^b_1,T^b_2)$,  $\theta^\lambda=(T^\lambda_1,T^\lambda_2)$):
\begin{align*}
&\lambda e^{-\frac{TR_1}{T^a_1}} +(1-\lambda) e^{-\frac{TR_1}{T^b_1}} = e^{-\frac{TR_1}{T^\lambda_1}},\\
&\lambda e^{-\frac{TR_1}{T^a_2}} +(1-\lambda) e^{-\frac{TR_1}{T^b_2}} = e^{-\frac{TR_1}{T^\lambda_2}},\\
&\lambda e^{-\frac{TR_2+TR_1}{T^a_1}} +(1-\lambda)e^{-\frac{TR_2+TR_1}{T^b_1}}  = e^{-\frac{TR_2+TR_1}{T^\lambda_1}},\\
&\lambda e^{-\frac{TR_2+TR_1}{T^a_2}}  +(1-\lambda)e^{-\frac{TR_2+TR_1}{T^b_2}}  = e^{-\frac{TR_2+TR_1}{T^\lambda_2}},\\
&\lambda e^{-\frac{TR_1}{T^a_2}}e^{-\frac{TR_2}{T^a_1}}  +(1-\lambda)e^{-\frac{TR_1}{T^b_2}}e^{-\frac{TR_2}{T^b_1}} = e^{-\frac{TR_1}{T^\lambda_2}}e^{-\frac{TR_2}{T^\lambda_1}} .
\end{align*}
Since $TR>0$,  {taking $\lambda$ as the unknown of the above linear system, then it has a solution only if $\theta^a=\theta^b=\theta^\lambda$. This shows that there exists no $\theta^\lambda$ for $\theta^a\neq \theta^b$ which gives the contradiction.}
\end{proof}

The asserted non-convexity in Proposition \ref{prop:dis_bloch_non_convex} yields a disadvantage for methods based on projections of the magnetization reconstruction sequences onto the Bloch manifold, as projections need no longer be unique. One specific instance of such a method is BLIP.

Concerning the MRI data, it is quite natural to assume $D\in \left ([L^2(K)]^{2}\right)^{L}$ where $K$ denotes a compact frequency domain.
Now we are in a position to show the Fr\'echet differentiability of the qMRI-operator.
\begin{lemma}
Let $ \f {x} = (\rho,\theta)\in \tilde{V}\subset L^{\infty}(\om)\times [L^{\infty}(\om)]^{2}$, where $\tilde{V}$ is the open subset  with  functions of strictly positive values only. 
Then the qMRI-operator
\[Q: \tilde{V} \to \left([L^2(K)]^{2}\right)^{L}\]
is Fr\'echet differentiable.
Similarly, given $M_0\in [L^{\infty}(\om)]^3$ we have the following general estimate for sufficiently small $\f h\in  L^{\infty}(\om)\times [L^{\infty}(\om)]^{2}$:
\begin{equation}\label{eq:Lq_estimate_Q}
\norm{Q(\f x+\f h)-Q(\f x) - Q'(\f x)\f h}_{\left([L^2(K)]^{2}\right)^{L}}= o\left(\norm{\f h}_{[L^{2}(\om)]^{3}}\right).
\end{equation}
\end{lemma}
\begin{proof}
Recall that
\[Q(\f x)=\{Q^{(\ell)}(\f x)\}_{\ell=1}^{L}= \left\{P^{(\ell)}\mathcal{F}(\rho T_{x,y} M_{\ell}(\theta))\right\}_{\ell=1}^{L}.\]
To see the  Fr\'echet differentiability, we first notice that all $P^{(\ell)}$, $\mathcal{F}$ and $T_{x,y}$ are bounded, linear operators.
Then we consider
\begin{align*}
Q^{(\ell)}(\f x+\f h)  -Q^{(\ell)}(\f x) = &  P^{(\ell)}\mathcal{F}((\rho+h_\rho)  T_{x,y} M_{\ell}(\theta+h_\theta)) -P^{(\ell)}\mathcal{F}(\rho  T_{x,y}  M_{\ell}(\theta+h_\theta))  \\
  & \quad + P^{(\ell)}\mathcal{F}\left(\rho   T_{x,y} \left(M_{\ell}(\theta+h_\theta)- M_{\ell}(\theta)\right)\right).
\end{align*}
For every $\f {x}=(\rho, \theta) \in \tilde{V}$, and  $\f h=(h_\rho, h_\theta) \in L^{\infty}(\om)\times [L^{\infty}(\om)]^{2}$ small enough, applying the Fr\'echet differentiability of each $M_{\ell}$ from  Proposition \ref{prop:dis_bloch}, and using the estimates \eqref{eq:Lq_estimate_M} and \eqref{eq:L2_bound_M}, we get the estimates below:
\begin{align*}
     &\norm{ P^{(\ell)}\mathcal{F}((\rho+h_\rho)  T_{x,y} M_{\ell}(\theta+h_\theta)) -P^{(\ell)}\mathcal{F}(\rho  T_{x,y}  M_{\ell}(\theta+h_\theta)) }_{[L^2(K)]^{2}}\\
= & \norm{ P^{(\ell)}\mathcal{F}(h_\rho  T_{x,y} M_{\ell}(\theta+h_\theta)) }_{[L^2(K)]^{2}}\\
\leq &\norm{ P^{(\ell)}\mathcal{F}(h_\rho  T_{x,y} M_{\ell}(\theta))}_{[L^2(K)]^{2}}  + C\norm{h_\rho h_\theta }_{[L^2(\Omega)]^2} + o(\norm{h_\theta}_{[L^2(\Omega)]^2})
\end{align*}
and
\begin{align*}
 \norm{ P^{(\ell)}\mathcal{F}\left(\rho   T_{x,y} \left(M_{\ell}(\theta+h_\theta)- M_{\ell}(\theta)\right)\right)}_{[L^2(K)]^{2}}
 \leq & \norm{ P^{(\ell)}\mathcal{F}(\rho   T_{x,y} M'_{\ell}(\theta)h_\theta)}_{[L^2(K)]^{2}} \\
 & + o(\norm{h_\theta}_{[L^2(\Omega)]^2}).
\end{align*}
The two inequalities indicate  for sufficiently small $\f h\in  L^{\infty}(\om)\times [L^{\infty}(\om)]^{2}$
\begin{equation}\label{eq:es1}
\norm{Q^{(\ell)}(\f x+\f h)  -Q^{(\ell)}(\f x) - A^{(\ell)} \f h}_{[L^2(K)]^{2}}  = o\left(\norm{\f h}_{[L^{2}(\om)]^{3}}\right),
\end{equation}
where
\[
\begin{array}{llll}
A^{(\ell)}&:[L^{\infty}(\om)]^{3} &\to & [L^2(K)]^{2}, \\
          &:\f h &\mapsto  & P^{(\ell)}\mathcal{F}(h_\rho  T_{x,y} M_{\ell}(\theta)) +P^{(\ell)}\mathcal{F}(\rho   T_{x,y} M'_{\ell}(\theta)h_\theta),
\end{array}
\]
is a bounded linear operator.
Using \eqref{eq:es1} and the fact that
\[ \norm{\f h}_{[L^{2}(\om)]^{3}}\leq C \norm{\f h}_{[L^{\infty}(\om)]^{3}}  \]
we show that $Q^{(\ell)}$ is Fr\'echet differentiable, and $A^{(\ell)}$ is the derivative.
The derivative of $Q$ is obtained from derivatives of each $Q^{(\ell)}$ for $\ell\in \{1,\ldots,L\}$.
Finally the estimate \eqref{eq:Lq_estimate_Q} is obtained from \eqref{eq:es1}. 
\end{proof}
The above proof also presents a way of how to calculate the derivative of $Q$.
\begin{remark}
The estimates \eqref{eq:Lq_estimate_M} and \eqref{eq:Lq_estimate_Q} do not guarantee the differentiability of $M$ and $Q$ in the whole space $ [L^{2}(\om)]^{3}$ and  $ [L^{2}(\om)]^{3}$, respectively.
However, the Fr\'echet derivative $Q^{\prime}(\f x)$ at $\f x\in \tilde{V}$ as an operator from  $ [L^{2}(\om)]^{3}$ to $\left([L^2(K)]^{2}\right)^{L} $ is well defined given $M_0\in [L^{\infty}(\om)]^3$ which is applicable in practice.
\end{remark}

\begin{remark}
 {In order to simplify the presentation, we only consider $\rho:\Omega \rightarrow \R$ in the above analysis, that is $\rho$ is real-valued. It is, however, common practice to allow the proton density $\rho$ to be complex-valued in order to take into account the coil sensitivity and phase errors \cite{DavPuyVanWia14}. The analytical results and also the numerical algorithms can be easily extended to cover the case where $\rho:\Omega \rightarrow \C$ by just increasing the number of the unknown parameter functions. More precisely we can treat $\rho$ as two unknown parameters in the complex-valued case.
Later in the numerical part, Section \ref{subsec:exa_numerics},  we  give examples where $\rho$ is complex-valued.}
\end{remark}
\subsection{(Gauss-) Newton method for ideal data}
\label{subsec:Newton}

Next we turn towards iterative methods for computing a solution to the non-linear equation 
\begin{equation}\label{eq_Qx_D}
Q(\f x)=D,
\end{equation}
which in fact represents a system of equations
\[P^{(\ell)}\mathcal{F}(\rho T_{x,y} M_{\ell}(\theta))=D^{(\ell)},\quad \ell=1,\ldots,L.\]

Because of the regularity of the operator $Q$, a first idea to solve the non-linear operator equation \eqref{eq_Qx_D}  is using a Gauss-Newton method, which, given some approximate solution $\f x_n$, is based on the first-order approximation
\begin{equation}
\label{eq:Taylor}
Q(\f x_{n+1}) \simeq  Q(\f x_n) +  Q^{\prime}(\f x_n)\left( \f x_{n+1}  - \f x_{n} \right) = D. 
\end{equation}
By letting $D_n: = D- Q(\f x_n) +  Q^{\prime}(\f x_n) \f x_n $, \eqref{eq:Taylor} becomes
\begin{equation}\label{eq:Newton_equ}
Q^{\prime}(\f x_n)\f x_{n+1}  - D_n=0.
\end{equation}
Note that since $D=\{ D^{(\ell)}\}_{\ell=1}^{L}$ is a sequence of data frames of length $L$, so is $D_n$. 
Since typically, we have $L\geq 2$, (the space discrete version of) \eqref{eq:Newton_equ} in general contains redundant equations.
Thus, one considers \eqref{eq:Newton_equ} in a least-squares sense.
Taking into account also the physical constraint of the tissue parameters, we introduce the feasible set $\tilde{\mathcal{C}}_{ad}\subset [L^{\infty}(\Omega)]^{3}$ which is a connected and convex set (typically a box) and contains all feasible values for $\f x=(\rho,\theta)$. Finally, it leads to computing $\f x_{n+1}$ by solving
\begin{equation}
\label{eq:Gauss_Newton}
\f{x}_{n+1}=\underset{\f x\in \tilde{\mathcal{C}}_{ad}}{\operatorname{argmin}}\; \norm{  Q^{\prime}(\f x_n)\f x  - D_n}_{\left([L^2(K)]^2\right)^L}^2,\quad n=0,1,2\ldots \;  
\end{equation} 

The solution of the problem in \eqref{eq:Gauss_Newton} can be approximated by a  projection step to $\tilde{\mathcal{C}}_{ad}$, resulting to the following projected Gauss-Newton iteration:
\begin{align}
\label{eq:D_ite}
 D_n& = D- Q(\f x_n) +  Q^{\prime}(\f x_n) \f x_n, \\
\label{eq:GN_ite}
\f y_{n+1}&=(Q')^{\dag}(\f x_{n})D_{n} : = \left( (Q^{\prime}(\f {x}_n))^\top Q^{\prime}(\f {x}_n)\right)^{-1}(Q^{\prime}(\f {x}_n))^\top D_n ,\\
\label{eq:proj_GN}
\f x_{n+1}&=P_{\tilde{\mathcal{C}}_{ad}} \f y_{n+1}.
\end{align} 
We point out that the step in \eqref{eq:GN_ite} is regarded in a Hilbert space setting, i.e.,
\[Q^{\prime}(\f {x}_n):[L^2(\Omega)]^3 \rightarrow \left([L^2(K)]^2\right)^L,\;\text{ for }  n\in \N,\]
and $(Q^{\prime}(\f {x}_n))^\top$ is the Hermitian adjoint of the linear operator $Q^{\prime}(\f {x}_n)$. This can be done since,
as we have mentioned that $Q^{\prime}(\f {x}_n)$ is a well-defined linear operator for functions in $[L^2(\Omega)]^3$, and \eqref{eq:GN_ite} will give a solution $\f y_{n+1}\in [L^2(\Omega)]^3$.
The subsequent projection step \eqref{eq:proj_GN} assures that $\f x_{n+1}\in \tilde{\mathcal{C}}_{ad}\subset [L^\infty(\Omega)]^3$.
Supposing that $\tilde{\mathcal{C}}_{ad}:=\{\f x\in [L^2(\Omega)]^3:x_p(r)\in [\underline{C}_p,\overline{C}_p] \text{ for }p\in\{1,2,3\}\text{ a.e. }r\in\Omega\}$ for $\underline{C},\overline{C}\in\mathbb{R}^3$ with $\underline{C}_p<\overline{C}_p$ for $p\in\{1,2,3\}$ and $\f x=(x_1,x_2,x_3)^\top$, the projection can be realised by
\begin{equation}
\label{eq:threshold}
(P_{\tilde{\mathcal{C}}_{ad}} \f x)_p(r)=\left\{
\begin{array}{ll}
\underline{C}_p & \text{ for } x_p(r)\leq \underline{C}_p ,\\
x_p(r) & \text{ for } \underline{C}_p<x_p(r) <\overline{C}_p,\\
\overline{C}_p &\text{ for } \overline{C}_p \leq x_p(r)
\end{array}
\right. 
\end{equation}
for every $r\in\Omega$.
Different from the projection in BLIP algorithm, the projection in \eqref{eq:threshold} is uniquely defined because of the convexity of the feasible domain.
In particular, for an exact solution $\f x^{*}$ of \eqref{eq_Qx_D}, we assume that $\f x^* \in \tilde{\mathcal{C}}_{ad}$.
It is obvious that the non-expansiveness holds for the projection operator:
\begin{equation}\label{eq:non_expan}
\norm{\f x_{n+1}-\f x^*}_{[L^2(\om)]^3} \leq \norm{\f y_{n+1}-\f x^*}_{[L^2(\om)]^3}.
\end{equation}
We state here the result regarding superlinear convergence rate of the projected Gauss-Newton iteration \eqref{eq:D_ite}--\eqref{eq:proj_GN} given the Fr\'echet differentiability of the non-linear operator $Q$ and the general estimate \eqref{eq:Lq_estimate_Q}.
\begin{theorem}
\label{thm:superlinear}
Let $\f {x}^*\in \tilde{\mathcal{C}}_{ad}$ be an exact solution of \eqref{eq_Qx_D}, and
assume there exists a neighbourhood $N(\f x^\ast)\subset [L^\infty(\om)]^3 $ of $\f x^{\ast}$ such that $(Q^\prime)^{\dag}(\f {x})$ is uniformly bounded for all $\f x \in N(\f x^\ast)$. 
Then there exists a potentially smaller neighbourhood such that for every initial guess $\f x_{0}$ belonging there,  for the iterates in \eqref{eq:GN_ite} and \eqref{eq:proj_GN} we have that $\f x_{n}\to \f x^{\ast}$ with a superlinear rate of convergence, i.e., 
\begin{equation}
\label{eq:superlinear}
\norm{\f { x}_{n+1} -\f {x}^*}_{[L^2(\om)]^3} = o\left(\norm{\f { x}_{n} -\f {x}^*}_{[L^2(\om)]^3} \right) \; \text{ for all } n\in \N. 
\end{equation}
\end{theorem}
 {Since $\tilde{\mathcal{C}}_{ad}$ is convex and the projection is non-expansive, the proof of Theorem \ref{thm:superlinear} is rather similar to the proof for classical unconstrained problems, see, e.g., \cite{ItoKun08},  therefore it is omitted here.
}

Due to the non-linearity of the map $Q$ and non-convexity of $Q(\tilde{\mathcal{C}}_{ad})$, the iteration in \eqref{eq:GN_ite} will only converge for initial values $\f x_0\in \tilde{\mathcal{C}}_{ad}$ in a certain neighbourhood of the exact solution $\f x^*$,  provided that the data $D$ contains no noise.

For undersampled and noisy data, it is even more crucial to choose a good initial guess in order to obtain a robust and efficient numerical algorithm for solving the problem \eqref{eq:Q_equation}.
This would be the main task of the next section.

\subsection{A projected Levenberg-Marquardt method for undersampled and noisy data}
\label{subsec:LM_method}
Undersampling is often  unavoidable in the acquisition process of MRI due to time constraints. 
The main problem caused by undersampling is ill-posedness of the equation \eqref{eq:Q_equation} due to the properties of the operator $P$ composed into the qMRI-operator $Q$.
As a consequence, the solution of \eqref{eq:Q_equation} may be unreliable, even when the data is contaminated by noise of small intensity.

In order to address the problem of undersampling and noise, and to solve \eqref{eq:Q_equation} robustly, we turn to a projected \emph{Levenberg-Marquardt} (L-M) method  instead of the aforementioned projected Gauss-Newton scheme. 
Suppose that the ideal data $D$ has been corrupted by some noise, leading to perturbed data $D^{\delta}$.
Then the projected L-M iteration reads: Given $\f x_0\in \tilde{\mathcal{C}}_{ad}$ and a sequence $\{\lambda_n\}_{n\in\mathbb{N}}$ of positive real numbers, iterate for $n=\set{0,1,2,\ldots}$:
 \begin{align}
 \label{eq:data2}
 \tilde{D}^\delta_n& = D^\delta- Q(\f x_n),\\
 \label{eq:L-M}
\f h^\delta_n&=\underset{\f h}{\operatorname{argmin}}\; \norm{  Q^{\prime}(\f x_n)\f h  - \tilde{D}^\delta_n}_{\left([L^2(K)]^2\right)^L}^2 +\lambda_n \norm{\f h}_{[L^2(\om)]^3}^2, \; \\
 \label{eq:proj_LM}
\f x_{n+1}&= P_{\tilde{\mathcal{C}}_{ad}}(\f x_n + \f h^\delta_{n} ).
 \end{align}
where $P_{\tilde{\mathcal{C}}_{ad}}$ is the projection as defined in \eqref{eq:threshold}.

From a regularization point of view, the L-M iteration \eqref{eq:L-M} is nothing else but  an iterative Tikhonov regularization for solving a non-linear equation \cite{Han97,KalNeuSch08}.
Note that if $\lambda_n= 0$ for every $n$, then the L-M method becomes a Gauss-Newton method.
The convergence and convergence rates of L-M methods in the sense of regularization have been shown in many works; see, e.g., \cite{Han97}.
There, general rules of choosing the parameter of a form $\lambda_n=\lambda_0\beta^n$ for some $\lambda_0>0$, $\beta\in (0,1)$ are discussed, as well as a discrepancy principle of terminating the iterations at step $n=n_{e}$ where $n_{e}$ is the first iteration index such that the condition
\[\norm{Q(\f x_{n_e}) -D^\delta}_{\left([L^2(K)]^2\right)^L}\leq \varrho \delta \]
holds.
It is also shown that with these choices, the solution of the L-M method converges to a solution of the original  non-linear equation.
In our case this yields $\f x_{n_e} \rightarrow \f x^*$ as $\delta \rightarrow 0$.

The local and global convergence as well as rates of convergence of (projected) L-M algorithms have also been intensively studied; we refer to \cite{Deu04,FanYua05,YamFuk01,KanYamFuk04} for instance. In the absence of additive noise and with proper initial values,
the optimal convergence rates of the L-M algorithm are determined by the rates of the updated parameters $\lambda_n$, i.e., $ \norm{ \f x_n -\f x^*}_{[L^2(\om)]^3} =\mathcal{O} (\lambda_n)$.
In \cite{KanYamFuk04}, quadratic convergence rate of projected L-M algorithm for convex constraint has been proved in finite dimensional spaces.
For non-zero residual problems, i.e., in the presence of additive noise,
a standard L-M method with no projection usually only achieves a linear convergence rate
$ \norm{\f x_{n+1} -\f x_n}_{[L^2(\om)]^3} \leq C \norm{\f x_n -\f x_{n-1}}_{[L^2(\om)]^3}$
for some constant $C<1$.
With an additional convex constraint, in the case of non-zero residual problem, we expect that the projected L-M method will keep the convergence rate as the non-projected L-M for unconstrained problems, even though the convergence result seems to be more complicated than the zero residual problem.
We ignore the discussion in detail in this paper.

As for the (projected) Gauss-Newton iteration, initialization is  crucial for the (projected) L-M method.
Unfortunately, there is no general way to produce good initial guesses, rather this is a problem-dependent task. Here we suggest to use a very fast version of MRF (BLIP or other robust generations) in order to produce initial points in a neighbourhood of a solution. 
The low run-time of the initialization scheme is related to using a relatively coarse dictionary only. 
In this way, the dictionary is no longer refined in the L-M iterations. Having clarified this, our main proposed algorithm is summarized in Algorithm \ref{our_algorithm}.
\begin{algorithm}
\begin{itemize}
\item{Input and setting:}
\begin{itemize}
\item MRI data $D^{\delta}\in  \left([L^2(K)]^{2}\right)^{L}$;
\item Parameters for the physical setting of MRI, e.g., flip angle and repetition time sequences, $\{\alpha_{\ell}\}_{\ell=1}^{L}$, $\{TR_{\ell}\}_{\ell=1}^{L}$;
\item A coarse discretization of the set $C_{ad}=[T_{1}^{\min}, T_{1}^{\max}]\times [T_{2}^{\min}, T_{2}^{\max}]$.
\end{itemize}
\item{Initialization:} 
\begin{itemize}
\item Generate a dictionary $\mathrm{Dic}(C_{ad})$, using the coarse discretization of $C_{ad}$,
 the flip angles and the repetition times, with the help of the IR-bSSFP, formula \eqref{eq:Bloch_discrete} for magnetization;
 \item Use Algorithm \ref{alg:BLIP} (or other generations), to produce an initialization: $\f x_{0}=(\rho_{0},\theta_{0})\in  \tilde{\mathcal{C}}_{ad} \subset [L^{\infty}(\om)]^{3}$;
 \item Choose an initial parameter $\lambda_{0}\ge 1$.
\end{itemize}
\item{Projected L-M iteration:}
\begin{itemize}
\item[(1)] Do the projected L-M iteration step \eqref{eq:data2}--\eqref{eq:proj_LM};
\item[(2)] If stopping criteria are not fulfilled, set $n\leftarrow n+1$, update  $\lambda_{n}=\max \{\lambda_{0}\beta^{n},\mu_{n}\}$, where $\beta\in (0,1)$, and $\mu_{n}\geq 0$ and go back to (1); otherwise, give the output.
\end{itemize} 
\item{Output:} The estimated parameter map $\f x_{n_{e}}=(\rho_{n_{e}},\theta_{n_{e}})$, for some final iteration index $n_{e}$.
\end{itemize}
\caption{\newline\textbf{ Projected L-M iteration with MRF/BLIP-based initialization}}
\label{our_algorithm}
\end{algorithm}
There $(\mu_{n})_{n\in\mathbb{N}}$ is a sequence of  parameters that depend on the noise level in the data, and $\lambda_0$ depends on the sub-sampling rate.
In our numerical examples below, we set $\lambda_0=s^2$, where $1/s$ is the undersampling rate of the data.
A typical choice for $\mu_n$ is $ \mu_n= \epsilon \norm{Q(\f x_n)-D^{\delta}}_{\left([L^2(K)]^{2}\right)^{L}}$ where $\epsilon \in (0,1) $.

\subsection{Why more data frames can help}
\label{subsec:Chebyshev}
In the original MRF approach, in order to handle the problem of noisy data, the use of a large number $L$ of consecutive pulse sequences and acquisitions is proposed. Conceptionally, this technique should average out noise and thus support better reconstructions. We borrow this idea here and justify it theoretically in what follows. 
In this part we consider problems after discretization, that is, in finite dimensional spaces.

For this purpose, we first recall the so called Chebyshev's inequality for vector-valued random variables (see e.g. \cite{Fer82,Pap91}). In its formulation, $\mathbb{P}(\cdot)$ stands for the probability of an event and $\|\cdot\|_{\R^{p}}$ denotes the Euclidean norm in $\R^{p}$.
\begin{lemma}[Chebyshev's inequality]
\label{lem:Chebyshev}
Let $\phi=(\phi_1,\phi_2,\ldots,\phi_p)$ be a vector-valued  random variable, for some $p\in\mathbb{N}$, with expected value  and variance $E(\phi)=\chi=(\chi_1,\chi_2,\ldots,\chi_p)$, $V(\phi)=\Sigma^2=(\sigma_1^2,\sigma_2^2,\ldots,\sigma_p^2)$, respectively.
Then,  for every $\epsilon>0$, we have 
\begin{equation}
\label{eq:Chebyshev}
\mathbb{P}(\|\phi -\chi\|_{\R^{p}} > \epsilon ) \leq   \frac{\|\Sigma^2\|_{\R^{p}}}{\epsilon^2}.
\end{equation}
\end{lemma}

The following, main theorem of this section, states that if a family of $L$ linear systems has a common solution and the right hand sides are perturbed by noise, then by solving a least-squares problem one can get an approximation of the common solution, with a certain probability that gets improved as the number $L$ increases. Later we shall see how this applies to our proposed algorithm to qMRI.

\begin{theorem}
\label{thm:gaussian_data}
Let $A_\ell \zeta=b_\ell$, $\ell=1,\ldots, L$, be a family of $L$  linear systems of equations, where $\set{b_\ell}_{\ell=1}^{L}$, with $b_\ell\in\mathbb{R}^p$ for every $\ell$, and $\set{A_\ell}_{\ell=1}^{L}$, with $A_\ell\in\mathbb{R}^{d\times p}$, $p\le d$, and $\mathrm{rank}(A_{\ell})=p$ for every $\ell$, are given sequences of data and system matrices, respectively. 
Assume also that the singular values of all $A_\ell$ have a uniform lower and upper bound  $\sqrt{c}$ and $\sqrt{C}$, respectively, which are both independent of $L$. Further suppose that this family of equations has a common solution $\zeta^{\ast}\in \R^{p}$.
If $\tilde{b}_\ell=b_\ell+ \delta_\ell$, where $\{\delta_\ell\}_{\ell=1}^{L}$ are independently identically distributed (i.i.d.) random variables with expected value $(0,\ldots,0)\in \R^{p}$, and variance $(\sigma^2,\ldots,\sigma^2)\in \R^{p}$,
then the least-squares solution 
\begin{equation}
\label{eq:least_square_error}
\zeta_{ls}=\underset{\zeta\in \R^{p}}{\operatorname{argmin}}\; \norm{A \zeta-\tilde{b}}_{\R^{Ld}}^2,
\end{equation}
where 
\[A= (A_1,\, A_2,\cdots,\,A_L)^\top \quad  \text{and} \quad \tilde{b}=(\tilde{b}_1,\,\tilde{b}_2,\cdots,\,\tilde{b}_L)^\top \]
approximates the solution $\zeta^*$ with the following probability estimate 
\begin{equation}
\label{eq:error_probability}
\mathbb{P}(\norm{\zeta_{ls}-\zeta^*}_{\R^{p}}> \epsilon)< \frac{\sigma^2}{ \epsilon^2} \mathcal{O}\left( \frac{p}{L} \right),\;\quad \text{  for  every }\; \epsilon> 0.
\end{equation}
\end{theorem}
\begin{proof}
From \eqref{eq:least_square_error}, we get $\zeta_{ls}=(A^\top A)^{-1}A^\top\tilde{b}$, which is also a random variable.
Since $A$ is not random, we can compute the expected value of $\zeta_{ls}$ as follows:
\[E(\zeta_{ls}) = E((A^\top A)^{-1}A^\top\tilde{b})=(A^\top A)^{-1}A^\top E(\tilde{b})= (A^\top A)^{-1}A^\top b =\zeta^*.\]
Therefore $E(\zeta_{ls}- \zeta^*)=0$. Similarly, for the variance (diagonal of the covariance matrix) we have
\[V( \zeta_{ls}-\zeta^* ) =\sigma^2 \text{diag}\left((A^\top A)^{-1}\right),\]
where ``diag'' denotes the diagonal of a matrix.
Denoting by $\operatorname{Tr}$ the trace operator, i.e.,  the summation of the diagonal values 
and using Lemma \ref{lem:Chebyshev}, we get that for every $ \epsilon>0$ 
\begin{equation}
\mathbb{P}(\norm{ \zeta_{ls}-\zeta^*}_{\R^{p}} >  \epsilon )
< \frac{\sigma^2 \| \text{diag}\left((A^\top A)^{-1}\right)\|_{\R^{p}} }{\epsilon^{2}}
\le \frac{\sigma^2  \operatorname{Tr}\left((A^\top A)^{-1}\right) }{\epsilon^{2}}.\label{less_than_trace}
\end{equation}
Here we have used the fact that the matrix $(A^\top A)^{-1}$ is positive definite and hence it has strictly positive diagonal elements, together with the fact that the $\ell_1$ norm in $\R^{p}$ is larger than the Euclidean one.

From the form of $A$ we have
$A^\top A= \sum_{\ell=1}^L A_\ell^\top A_\ell$ with trace 
\[\operatorname{Tr}(A^\top A)=\sum_{\ell=1}^L \operatorname{Tr}(A_\ell^\top A_{\ell}).\] 
Since every $A_\ell^\top A_\ell$ is positive definite, so is $A^\top A$.
Let $\set{S_j}_{j=1}^{p}$ be the eigenvalues of $A^\top A$ allowing for the decomposition
\begin{equation}\label{decomp}
A^\top A=USU^{-1}  \quad \text{ and } \quad (A^\top A)^{-1}= US^{-1}U^{-1}, 
\end{equation}
where $S$ is the diagonal matrix with entries $\set{S_j}_{j=1}^{p}$, and $U$ is a unitary matrix.
Then, for the traces we have 
\[\operatorname{Tr}(A^\top A)=\sum_{j=1}^p S_j\quad \text{ and } \quad \operatorname{Tr}\left((A^\top A)^{-1}\right)=\sum_{j=1}^p \frac{1}{S_j}.\] 
Due to the uniform lower and upper bounds on the singular values of $\set{A_\ell}_{\ell=1}^L$, we get a corresponding uniform bound on the eigenvalues of the matrices $\set{A_\ell^\top A_\ell}_{\ell=1}^L$, i.e.,
\[ c L\leq S_j \leq C L, \quad \text{ for all }\; j=1,2,\ldots,p.\]
Consequently, we have  
\[\frac{1}{C L}\leq \frac{1}{S_j} \leq \frac{1}{cL} \quad \Longrightarrow\quad \frac{1}{S_j}=\mathcal{O}\left(\frac{1}{L}\right) ,\;\text{ for }\; j=1,2,\ldots,p.\]
From this we infer the following estimate
\[ \operatorname{Tr}\left((A^\top A)^{-1}\right)=\sum_{j=1}^p \frac{1}{S_j}=\mathcal{O}\left(\frac{p}{L}\right),\]
and combined with \eqref{less_than_trace} it proves the assertion
\[ \mathbb{P}(\norm{\zeta_{ls}-\zeta^*}_{\R^{p}} >  \epsilon ) =\frac{\sigma^2}{ \epsilon^2}\mathcal{O}\left(\frac{p}{L}\right). \]
\end{proof}

Theorem \ref{thm:gaussian_data} relates to our qMRI algorithm in several ways:

 (i)  Observe that regarding the setting of qMRI problems, the noise in the data obtained after each pulse sequence can be considered as realisations of i.i.d. random variables.
 
(ii) In Newton-type methods, if there is no sub-sampling for the qMRI-operator, then $Q^\prime$ is non-degenerate on the effective domain $\Omega$, i.e., on the part of the slices where the biological tissue is imaged. In this case, we may consider $A_{\ell}=(Q^{(\ell)})^\prime $, and $b_{\ell}= (D_k^{(\ell)})^\delta$, with both quantities
satisfying the assumptions of Theorem \ref{thm:gaussian_data} given that the data contains Gaussian noise.
This indicates that the result of Theorem \ref{thm:gaussian_data} can be applied to every Newton-type step for a given $\epsilon>0$, and an initial value $\norm{\f {x}_0 - \f {x}^*} \leq \epsilon $. 
Since we have restricted to a small neighbourhood of the exact solution $\f {x}^*$, we can take roughly the common solution $\zeta^{\ast}$ corresponding to \eqref{eq:least_square_error} of Theorem \ref{thm:gaussian_data} as the exact solution of the least-squares problem \eqref{eq:Gauss_Newton}. 

(iii) In the case of the Levenberg-Marquardt method, with the sub-sampling operators $P^{(\ell)}$, the results of Theorem \ref{thm:gaussian_data} can still  be applied as the involved matrices become $A_{\ell}= \left(((Q^{(\ell)})^\prime)^\top, \sqrt{\lambda_n }\text{Id} \right)^\top  $, and $b_{\ell}=\left(((\tilde{D}_n^{(\ell)})^\delta)^\top, 0\right)^\top$. Note that in this case $\zeta$ corresponds to $\f h$.
Further the matrices $A_{\ell}$ will always be of full rank with uniformly bounded singular values, whenever we let the sequence $(\lambda_n)_{n\in\mathbb{N}}$ be uniformly bounded away from zero. Such a uniform lower bound is indeed usually in place at the presence of noise. In such a case, we can treat $\zeta^{\ast}\equiv 0$ as the common solution of \eqref{eq:L-M} .

\section{Numerical results}
\label{sec:Numerics}
Now we report on numerical results obtained by our Algorithm \ref{our_algorithm} when applied to synthetic data. Our setting also allows for an extensive quantitative comparison with Algorithm \ref{alg:BLIP} (BLIP), which was shown in \cite{DavPuyVanWia14} to be superior to the original MRF.

\subsection{Generating test data}
Our tests are based on synthetic data from an anatomical brain phantom, publicly available from the Brain Web Simulated Brain Database \cite{Brainweb,Col_etal_98}. We use 
a $217\times 181$ slice completed by zero fill-in order to generate a $256\times 256$-pixel image. The selected ranges for $\theta=(T_1,T_2)^\top$ and  $\rho$ reflect natural values encountered in the human body \cite{DavPuyVanWia14},
with $T_1$ ranging from $530ms$--$5012ms$, $T_2$ from $41ms$--$512ms$, and  $\rho$ between $80$--$100$. As pixel units in practical images very likely contain multiple tissue types rather than only a pure one in a single volume of the observed pixels, we interpolate the values of each parameter $T_1$, $T_2$ and $\rho$ of the $256\times 256$ phantom, respectively, by averaging the values of every four neighboured pixels with non-zero parameter values.
This average process shrinks the $256\times 256$ image to a $128\times128$ image.
In Figure \ref{fig:real_solutions}, we display the interpolated parameters of $T_1$, $T_2$ and $\rho$ as coloured images. These serve as the ground truth for our numerical tests.
\begin{figure}[!ht]
\centering
\includegraphics[width=0.32\textwidth, trim={0 0 12.95cm 0},clip]{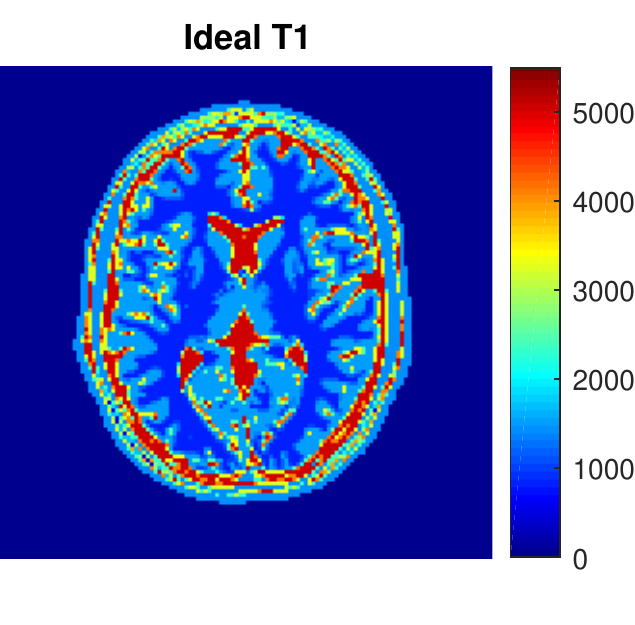}
\includegraphics[width=0.32\textwidth, trim={6.475cm 0 6.475cm 0},clip]{128real_solution-eps-converted-to.pdf}
\includegraphics[width=0.32\textwidth, trim={12.95cm 0 0 0},clip]{128real_solution-eps-converted-to.pdf}
\caption{The interpolated parameters serve as the ground truth for our algorithm. From left to right: $T_1$, $T_2$ and $\rho$.}
\label{fig:real_solutions}
\end{figure}

The IR-bSSFP pulse sequence scheme introduced in Section \ref{subsec:Bloch_dynamic} is applied to generate MRI data. It is based on constant flip angles $\alpha$ and repetition time $TR$ sequences of length $L$. The data $D$ are generated by using the prescribed parameters $T_1$, $T_2$  and $\rho$ with the pulse sequences characterized by $\alpha$ and $TR$. With this setting, we first simulate the magnetization, and then use FFT to generate the Fourier space data from it. Sub-sampling is implemented by using the scheme described in the next section. We also note that for the generation of the magnetization, we rely on \eqref{eq:Bloch_discrete} and take advantage of the MATLAB code provided in \cite{Ma_etal13}.
For simplicity, we set the phase shift $\phi \equiv 0$ in \eqref{eq:Bloch_discrete}.
 {The noise in the Fourier data is simulated as follows: we first add Gaussian noise of mean zero to the magnetization function $M$ over the effective domain $\Omega$, and then apply discrete Fourier transform to it. Note the Gaussian noise after the application of Fourier transform is still Gaussian.} 

\subsection{Sub-sampling patterns}\label{subsec:subsampling}
\subsubsection*{Cartesian sub-sampling}
Here we focus on Cartesian sub-sampling which is frequently used in practice; see, e.g.,  \cite{Mck93}.  This choice implies a specific form of the sub-sampling operator $P^{(\ell)}$ for $\ell=1,\ldots,L$. In the discrete setting, the full $k$-space data are given by a dense matrix of complex values or, equivalently, two real-valued dense matrices, respectively of size $N\times N$. According to our set-up above, we have $N=128$.
More specifically, we use here an $n$ multishot Echo Planar Imaging (EPI) scheme, which means that at every read-out time, $n$ rows of $k$-space are simultaneously filled. Hence, in every acquisition there will be $n<N$ rows of the matrix filled with Fourier coefficients. 
To simplify the discussion, we consider $(N\mod n)\equiv 0$, and further set $s:=N/n$, which gives a sub-sampling rate of $1/s$.
The sampling pattern $P^{(\ell)}$ is described in detail as follows:
\begin{itemize}
\item[(i)] For every $\ell$-th acquisition, define $\xi_\ell: =(\ell\mod s) $ for $\ell=1,\ldots,L$.
\item[(ii)] $P^\ell$ will include those rows of the full $k$-space matrices, indexed by numbers from the set $\iota$ with
\[\iota:= \set{i \in \set{1,\ldots,N}\;:\; (i \mod s )\equiv \xi_\ell}.\]
\end{itemize}
Thus, at every read-out time, $P^{(\ell)}$ samples $n$ rows from the full Fourier space to simulate the $n$ multishot EPI.
A simple example of such a sub-sampling pattern is shown in Figure \ref{fig:cartesian_sampling}.
\begin{figure}[!ht]
\includegraphics[width=\textwidth]{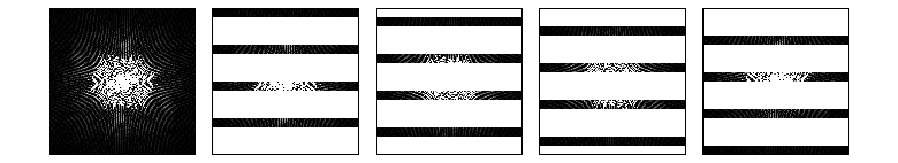}
\label{fig:cartesian_sampling}
\caption{A sub-sampling pattern example for $s=4$, $L=4$, and $N=16$. The first image depicts an example of fully sampled $k$-space data. The  second to the fifth images are example frames of the undersampled data, where the information indicated by blank rows  is not collected in that frame. The sub-sampling pattern follows the order periodically if $L>s$.}
\end{figure}

We note that this sub-sampling strategy differs from the one in \cite{DavPuyVanWia14}. There, $\xi_\ell$ is defined as a uniformly distributed random number in $\set{0,1,\ldots,s-1}$, whereas here we consider a deterministic periodical order.
After experimentation and when compared to the pseudo-random strategy of \cite{DavPuyVanWia14}, we found that the deterministic order is more stable and gives comparable or better results when the BLIP algorithm for Cartesian sub-sampled data is used.
Therefore, in our implementation of the BLIP algorithm we use the deterministic strategy as described above.

We consider different sub-sampling rates on the data using the above strategy.
By taking into account that longer processing time is needed for higher sampling rates, it follows that the flip angles and the repetition times must be increased proportionally.
Accordingly, we use the following repetition times  $TR=(TR_1,TR_2,\ldots, TR_L)$  and flip angles $\alpha=(\alpha_1,\alpha_2,\ldots, \alpha_L)$:
\begin{itemize}
\item[(a)] Fully sampled data:
Repetition time $TR_{\ell}= 40ms$ and flip angles $\alpha_\ell= \frac{40\pi}{180}$ for all $\ell=1,\ldots,L$.
\item[(b)] $1/4$ sampled data (sub-sampling rate $25\%$), e.g., a 32 multi-shot EPI:
Repetition time $TR_\ell= 20ms$ and flip angles $\alpha_\ell= \frac{20\pi}{180}$ for all $\ell=1,\ldots,L$.
\item[(c)] $1/8$ sampled data (sub-sampling rate $12.5\%$), e.g., a 16 multi-shot EPI:  a shorter repetition time $TR_\ell= 10ms$ and smaller flip angles $\alpha_\ell= \frac{10\pi}{180}$ are applied for all $\ell=1,\ldots,L$.
\end{itemize}
\subsubsection*{Radial sub-sampling}
 {We  use a similar strategy for radial sub-sampling patterns. There, we uniformly discretize the angular domain $[0,\pi)$. Each of the radial lines passing through the center point is fully sampled (full sample of radial direction).  The angular direction on the other hand, may be sub-sampled either randomly or, as this is shown in Figure \ref{fig:radial_sampling}, in a uniform fashion. 
\begin{figure}
\includegraphics[width=\textwidth]{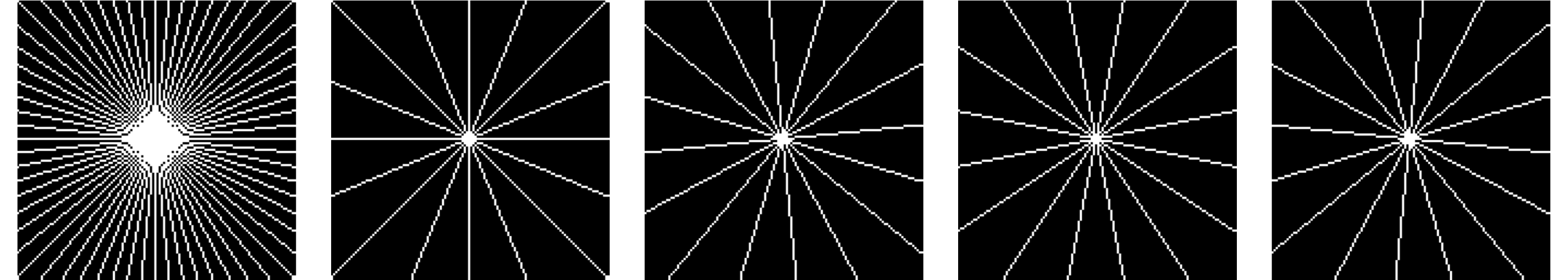}
\label{fig:radial_sampling}
\caption{A radial sub-sampling pattern example, where we have taken $p=32$, and $s=8$. Far left: The full $p$ resolution-angular radial sample obtained by rotating the sampling pattern in every acquisition step, with the    $4$ sampling  pattern frames depicted next. The $k$-space data are collected in each frame along the white strips.}
\end{figure}
As we use the  uniform sub-sampling of the angular direction in our numerical tests, we describe it in more detail: During each acquisition, $s$ angles (usually corresponding to $s$ coils) uniformly distributed over $[0,\pi)$ are selected. We thus obtain $k$-space data along the $s$ lines going through the center of the $k$-space. Then after every acquisition, we shift the angles by $\pi/p$, where $p$ is the angle-resolution number which we will consider to be $p=128$ in our examples later. This process is then periodically repeated till the end of the acquisition. 
However, we note that in contrast to Cartesian sub-sampling, in the radial case the angle-resolution $128$ can only provide a sampling rate of $ 74.02\% $ of the full $k$-space data of size $128\times128$.
Note that in our numerical experiments later, using this kind of mask, it generates $k$-space data in a squared grid where for non-sampled areas we fill with zeros.
The  Fourier transform over the radial sub-sampled space, and its inverse, can be easily implemented by using the MATLAB function ``fft2'' (``ifft2'' for the inverse). 
We mention that in practice radially sub-sampled data are often treated with non-uniform fast Fourier transform (NUFFT) algorithms. This is particularly relevant regarding the speed of the algorithm.
As far as the quality of the reconstructions is concerned, the approach we follow here is expected to be comparable to NUFFT.
}
\vspace{1em}

 {To illustrate the efficiency of the proposed method, we compare it with other methods from the literature. In particular, we choose BLIP and also the algorithm in \cite{MazWeiTalEld18} which approximates \eqref{eq:LR_BLIP}. 
In order to relax the non-convex penalty in \eqref{eq:LR_BLIP}, the authors in \cite{MazWeiTalEld18} utilized the nuclear norm, i.e. the sum of the singular values, of the discrete matrix that represents the magnetization variable.
The algorithm, named FLOR (MRF with LOw Rank), uses soft-thresholding of the singular values, and the details can be found in \cite[Algorithm 4]{MazWeiTalEld18}.
We also remark that in the literature on MRF, Cartesian sub-sampling is not used frequently since the pertinent artefacts pose extra challenges when compared to the ones due to spiral and radial sub-sampling. In fact, BLIP is one of the most successful examples for Cartesian sub-sampling among MRF techniques. According to our experiments, the FLOR algorithm seems not to work reasonably well for Cartesian sub-sampling. Therefore, in this case we do not present the results of FLOR but only compare to the BLIP algorithm.}

In order to compare our method with BLIP and FLOR algorithms, we used for the latter two a very fine dictionary where $T_1$ was discretized from $15ms$ to $5500ms$ with increments of $15ms$, and $T_2$ was discretized from $1.5ms$ to $550ms$ with increments of $1.5ms$.
This means that the discretized feasible domain $C_{ad}$ for $\theta=(T_1,T_2)$ was (in MATLAB notation)
\[C_{ad}=\left\{[15:15:5500]\times [1.5:1.5:550]\right\}.\]
In this case, the dictionary had $366\times366=133956$ entries, and it required memory for a storage matrix of dimension $133956\times L$. The deterministic sub-sampling pattern was used in all numerical examples for the BLIP algorithms, including the generation of the initial values.
For the projection onto the feasible set $\tilde{\mathcal{C}}_{ad}$ we use the following thresholds for each parameter values: (the constants $\underline{C} $  and  $\overline{C}$ here refer to  \eqref{eq:threshold})
\begin{equation*}
\begin{matrix}
    & \underline{C} & \overline{C}\\
T_1 &  0            & 5500       \\
T_2 &  0            & 550        \\
\rho&  0            & 100
\end{matrix}
\end{equation*}
Note that the  value $0$ is assigned to the marginal area in the tested images in Figure \ref{fig:real_solutions}, 
where there is no tissue information.
It is reported in  \cite{Ma_etal13} that adding random noise to the flip angles and to  repetition times may improve the final results of MRF (and BLIP). However, in our experiments we did not find significant differences.
Therefore, we do not add noise to the angles and repetition times in our numerical tests.
Further, the linear systems in the L-M iterations were solved by employing MATLAB's backslash command. For our test runs, we used a CPU with an Intel Core i5-7500, 3.40GHz, 2 cores, and RAM of 8GB DDR4, 2400 MHz, as well as
MATLAB of version 2018a under the operating system openSUSE 42.3.

\subsection{Experiments on Cartesian sub-sampled data}
\subsubsection{Undersampled data with no additive noise}
The first set of examples addresses noiseless undersampled data (Cartesian sub-sampling at rate $1/8$), and totally $L=80$ data frames.
In these tests, we used a coarse dictionary for initializing Algorithm \ref{our_algorithm}. Here $T_1$ was discretized from $200ms$ to $5500ms$ with increments of $200ms$, and $T_2$ was discretized from $20ms$ to $550ms$ with increments of $20ms$. 
Note we not only compare our results to the solutions of the BLIP algorithm, but we also plot the initial guesses produced by BLIP.
Concerning BLIP, following the findings in \cite{DavPuyVanWia14} we applied $20$ steps of a Landweber iteration. On the other hand, our method was stopped after $25$ Levenberg-Marquardt steps as then no significant change in the iterates was observed.
The regularization parameters had the following values: $\mu=\mu_{n}=0$, for every $n\in \mathbb{N}$, $\lambda_0=s^2$, and $\beta=0.01$.

The reconstructed parameter maps are presented  in Figure \ref{fig:undersampled_solutions}. In the first row we depict the parameter maps $T_{1}$, $T_{2}$, $\rho$ of the BLIP algorithm, computed with the coarse dictionary. These quantities were subsequently used for the initialization of our new algorithm. In the second row, the corresponding results for the fine dictionary are shown. These are the ones that should be compared with the images of the third row, which are the results of our algorithm.
In order to make the differences clearer, we also provide the corresponding error maps in Figure \ref{fig:undersampled_error}. In fact, we show the pointwise error maps 
$\abs{\theta_{computed}-\theta_{gt}}$,
where $\theta_{gt}$ are the ground-truth parameter maps shown in Figure \ref{fig:real_solutions}, and draw the reader to observing the scale of error as depicted in the vertical bar.
We observe that the accuracy of the estimated parameters, especially for $T_{1}$, is much higher in our method when compared to BLIP. Note that the error in BLIP is actually larger than the dictionary mesh size, which indicates that this is not a matter of the fineness of dictionary, but it could also be due to the projection onto a non-nonvex set as discussed above.

The rate of convergence of the proposed algorithm turns out to be linear for this example; see Figure \ref{fig:undersampled_convergence}. The figure depicts the ratio $\frac{\norm{x_{n+1}-x_n}_{2}}{\norm{x_n-x_{n-1}}_{2}}$ versus the number of iterations.
Note that $x$ stands here either for $T_1$, $T_2$, or $\rho$. 
\begin{figure}[!ht]
\centering
\includegraphics[width=0.32\textwidth, trim={0 0 12.95cm 0},clip]{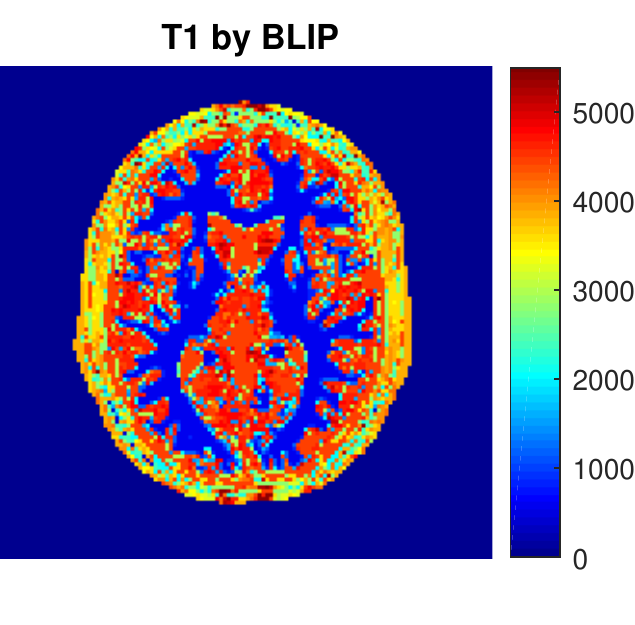}
\includegraphics[width=0.32\textwidth, trim={6.475cm 0 6.475cm 0},clip]{8th_sampled_initial_solution-eps-converted-to.pdf}
\includegraphics[width=0.32\textwidth, trim={12.95cm 0 0 0},clip]{8th_sampled_initial_solution-eps-converted-to.pdf}\\
\includegraphics[width=0.32\textwidth, trim={0 0 12.95cm 0},clip]{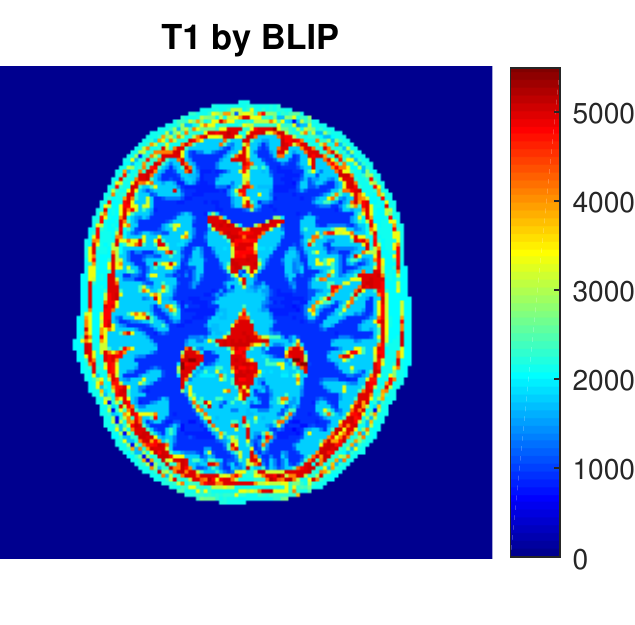}
\includegraphics[width=0.32\textwidth, trim={6.475cm 0 6.475cm 0},clip]{8th_sampled_BLIP_solution-eps-converted-to.pdf}
\includegraphics[width=0.32\textwidth, trim={12.95cm 0 0 0},clip]{8th_sampled_BLIP_solution-eps-converted-to.pdf}\\
\hspace{0.61pt}
\includegraphics[width=0.32\textwidth, trim={0 0 12.95cm 0},clip]{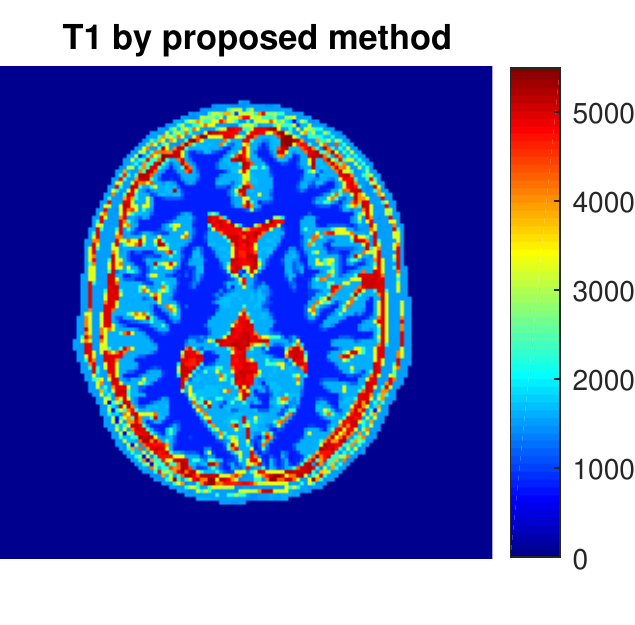}
\includegraphics[width=0.32\textwidth, trim={6.475cm 0 6.475cm 0},clip]{8th_sampled_pLM_solution-eps-converted-to.pdf}
\includegraphics[width=0.32\textwidth, trim={12.95cm 0 0 0},clip]{8th_sampled_pLM_solution-eps-converted-to.pdf}
\caption{Experiment with noiseless undersampled data. First row: Initialization of our algorithm, computed by BLIP with a coarse dictionary. Middle row: Result by BLIP with fine dictionary. Last row: Solution by proposed algorithm.}
\label{fig:undersampled_solutions}
\end{figure}
\begin{figure}[!ht]
\centering
\includegraphics[width=0.32\textwidth, trim={0 0 12.95cm 0},clip]{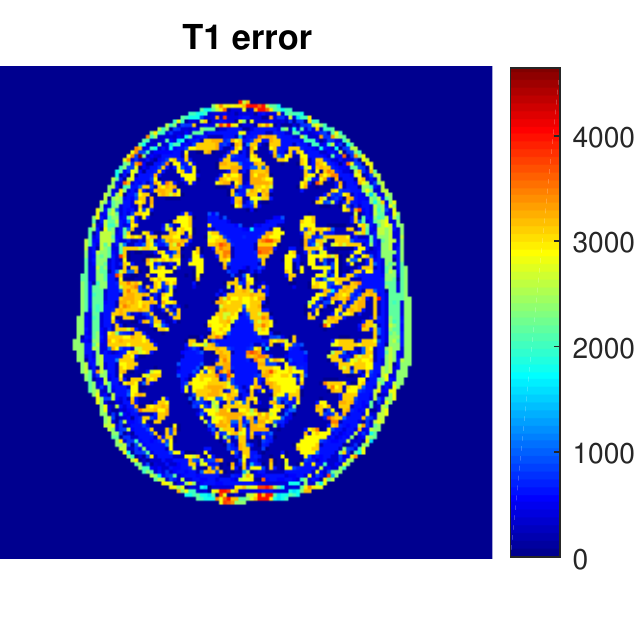}
\includegraphics[width=0.32\textwidth, trim={6.475cm 0 6.475cm 0},clip]{8th_sampled_initial_error-eps-converted-to.pdf}
\includegraphics[width=0.32\textwidth, trim={12.95cm 0 0 0},clip]{8th_sampled_initial_error-eps-converted-to.pdf}\\
\includegraphics[width=0.32\textwidth, trim={0 0 12.95cm 0},clip]{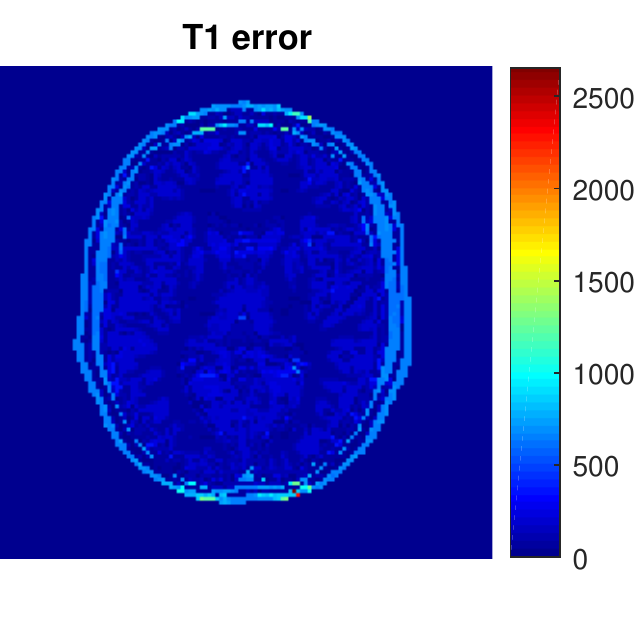}
\includegraphics[width=0.32\textwidth, trim={6.475cm 0 6.475cm 0},clip]{8th_sampled_BLIP_error-eps-converted-to.pdf}
\includegraphics[width=0.32\textwidth, trim={12.95cm 0 0 0},clip]{8th_sampled_BLIP_error-eps-converted-to.pdf}\\
\hspace{0.61pt}
\includegraphics[width=0.32\textwidth, trim={0 0 12.95cm 0},clip]{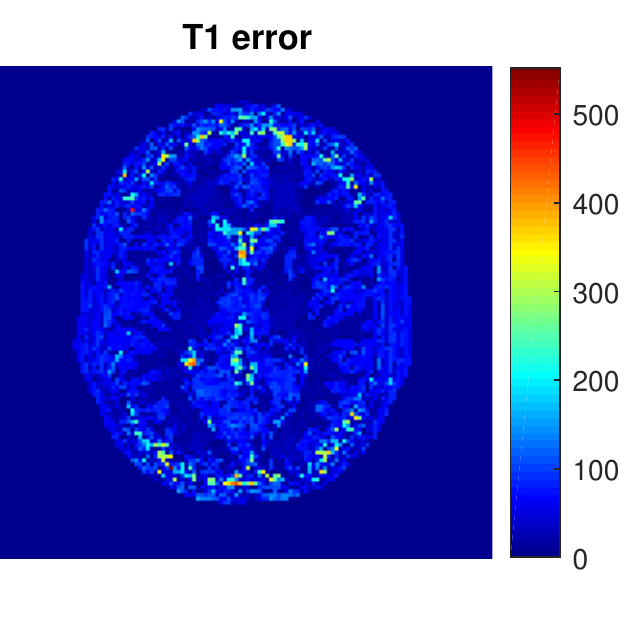}
\includegraphics[width=0.32\textwidth, trim={6.475cm 0 6.475cm 0},clip]{8th_sampled_pLM_error-eps-converted-to.pdf}
\includegraphics[width=0.32\textwidth, trim={12.95cm 0 0 0},clip]{8th_sampled_pLM_error-eps-converted-to.pdf}
\caption{Experiment with noiseless undersampled data. Pointwise distance of the solutions of  Figure \ref{fig:undersampled_solutions} to the corresponding ground truths of Figure \ref{fig:real_solutions}. First row: Initial error of BLIP with a coarse dictionary. Middle row: error of BLIP with fine dictionary. Last row: Error of the proposed algorithm.}
\label{fig:undersampled_error}
\end{figure}

\begin{figure}[!ht]
\centering
\includegraphics[width=\textwidth]{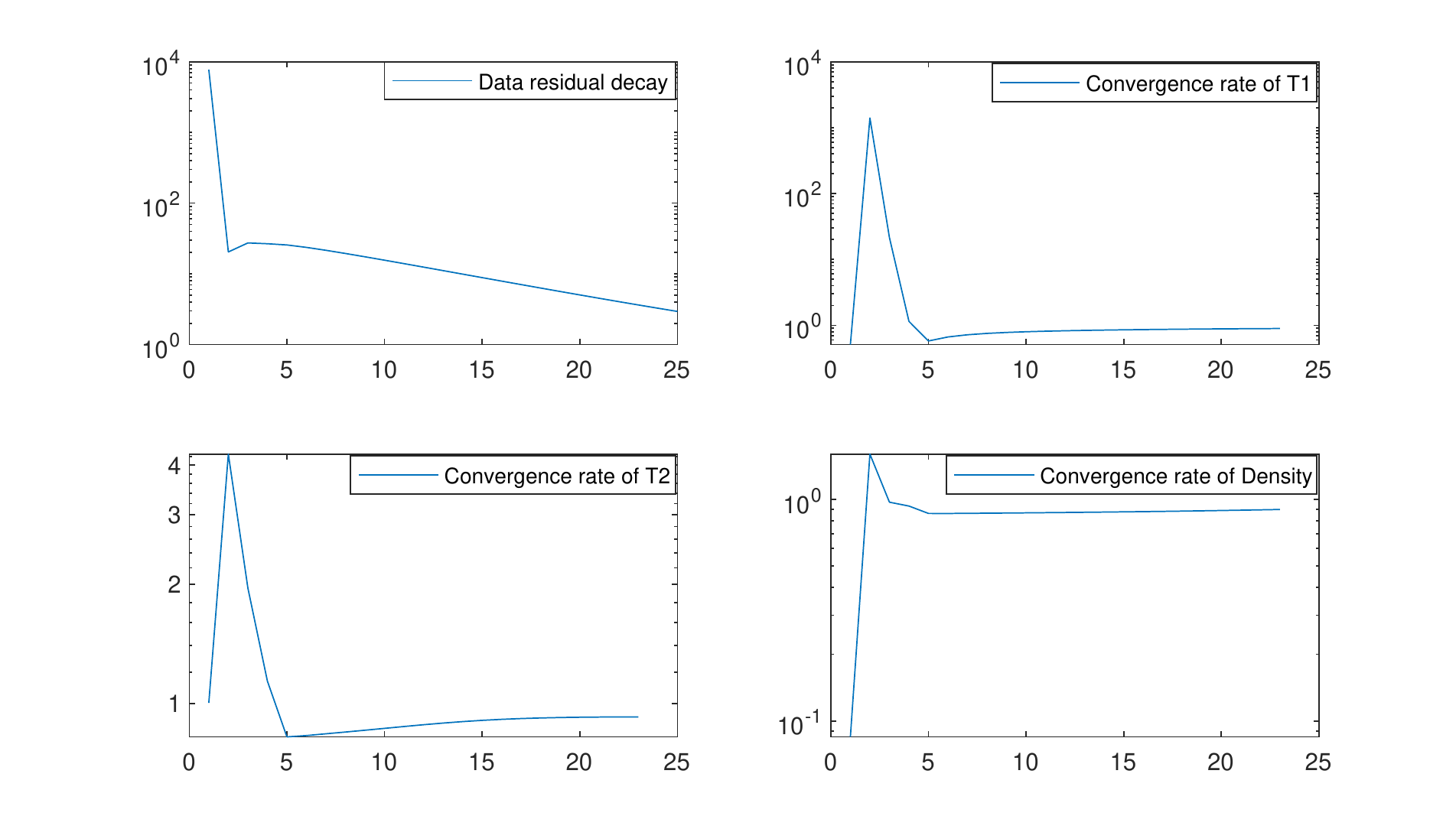}
\caption{Experiment with undersampled data. From left to right and from above to bottom: Convergence of the data residual $\|Q(\f x_{n})-D\|_{2}$,
convergence rates via plots of the iterate ratios $\frac{\norm{x_{n+1}-x_{n}}_{2}}{\norm{x_{n}-x_{n-1}}_{2}}$ for $x= T_1, T_2, \rho$, respectively.}
\label{fig:undersampled_convergence}
\end{figure}

\subsubsection{Undersampled data with additive noise}
\label{subsubsec:subsample_noise}

Now we present results for undersampled noisy data with a sub-sampling rate of 1/4 and additive Gaussian white noise of variance $\sigma^{2}=0.8$ and mean $0$. The total signal to noise ratio of the Fourier data is $SNR=35$. As before, we use here $L=80$ data frames. 

The coarse dictionary employed in order to generate the initial value $\f x_0$ used
$T_1$ discretized from $400ms$ to $5500ms$ with increments of $400ms$, and $T_2$ from $40ms$ to $550ms$ with increments of $40ms$. 
This resulted in a dictionary with $169$ entries only, and needed a complex-valued matrix of dimension $169\times L$ for its representation.
As in the previous example, the refined dictionary had a dimension $133956\times L$.
Again, we used $20$ Landweber iterations for BLIP, and  $25$ iterations for our L-M algorithm.
The regularization parameters were chosen as $\mu_{n}=10^{-8}\norm{Q\f x_n-D^\delta}_{2}$ for every $n\in\mathbb{N}$, $\lambda_0=s^2$, and $\beta=0.01$.
Note that because of noise, here we used a fixed $\mu$ strictly larger than zero.

We depict the results in Figure \ref{fig:noise_undersampled_solutions} and the corresponding pointwise errors in Figure \ref{fig:noise_undersampled_error}, using the same row system as in Figures \ref{fig:undersampled_solutions} and \ref{fig:undersampled_error}, respectively. 
The result of the proposed algorithm again outperforms the refined BLIP algorithm, especially in the reconstruction of the density map, but not as significantly as in the noiseless case. In addition, our method consumes much less memory and requires  much less CPU-time; see Table \ref{tab:Q_result}.

In Figure \ref{fig:noise_undersampled_convergence}, the residual ratio plots again show linear rates of convergence.

\begin{figure}[!ht]
\centering
\includegraphics[width=0.32\textwidth, trim={0 0 12.95cm 0},clip]{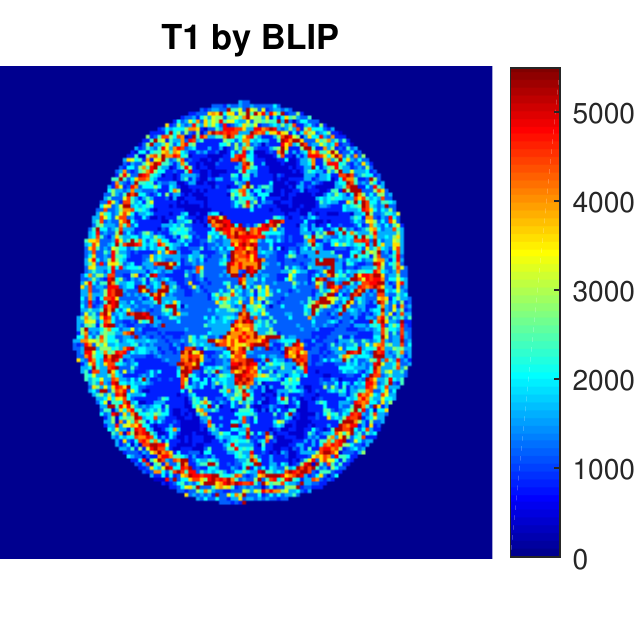}
\includegraphics[width=0.32\textwidth, trim={6.475cm 0 6.475cm 0},clip]{n_4th_sampled_initial_solution-eps-converted-to.pdf}
\includegraphics[width=0.32\textwidth, trim={12.95cm 0 0 0},clip]{n_4th_sampled_initial_solution-eps-converted-to.pdf}\\
\includegraphics[width=0.32\textwidth, trim={0 0 12.95cm 0},clip]{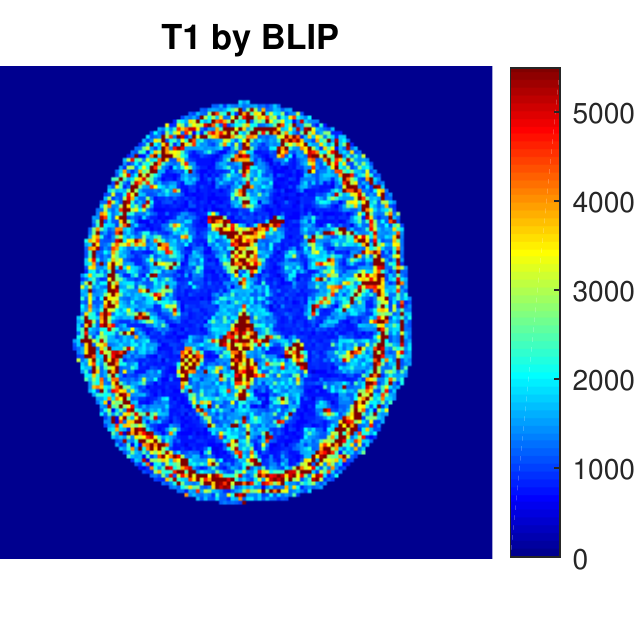}
\includegraphics[width=0.32\textwidth, trim={6.475cm 0 6.475cm 0},clip]{n_4th_sampled_BLIP_solution-eps-converted-to.pdf}
\includegraphics[width=0.32\textwidth, trim={12.95cm 0 0 0},clip]{n_4th_sampled_BLIP_solution-eps-converted-to.pdf}\\
\hspace{0.61pt}
\includegraphics[width=0.32\textwidth, trim={0 0 12.95cm 0},clip]{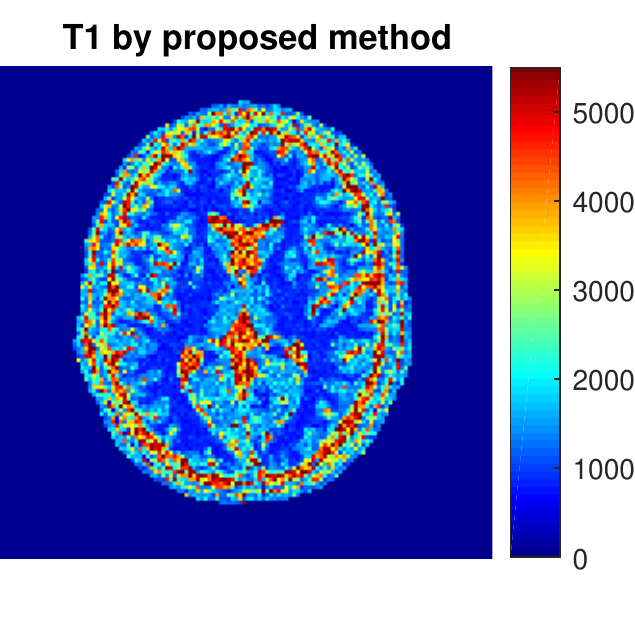}
\includegraphics[width=0.32\textwidth, trim={6.475cm 0 6.475cm 0},clip]{n_4th_sampled_pLM_solution-eps-converted-to.pdf}
\includegraphics[width=0.32\textwidth, trim={12.95cm 0 0 0},clip]{n_4th_sampled_pLM_solution-eps-converted-to.pdf}
\caption{Experiment with undersampled and noisy data.   First row: Initialization of our algorithm, computed by BLIP with a coarse dictionary. Middle row: Result by BLIP with fine dictionary. Last row: Solution by proposed algorithm.}
\label{fig:noise_undersampled_solutions}
\end{figure}
\begin{figure}[!ht]
\centering
\includegraphics[width=0.32\textwidth, trim={0 0 12.95cm 0},clip]{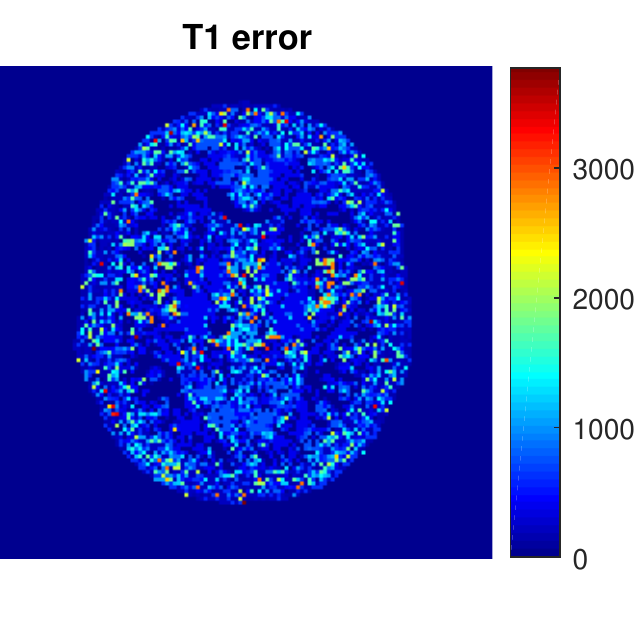}
\includegraphics[width=0.32\textwidth, trim={6.475cm 0 6.475cm 0},clip]{n_4th_sampled_initial_error-eps-converted-to.pdf}
\includegraphics[width=0.32\textwidth, trim={12.95cm 0 0 0},clip]{n_4th_sampled_initial_error-eps-converted-to.pdf}\\
\includegraphics[width=0.32\textwidth, trim={0 0 12.95cm 0},clip]{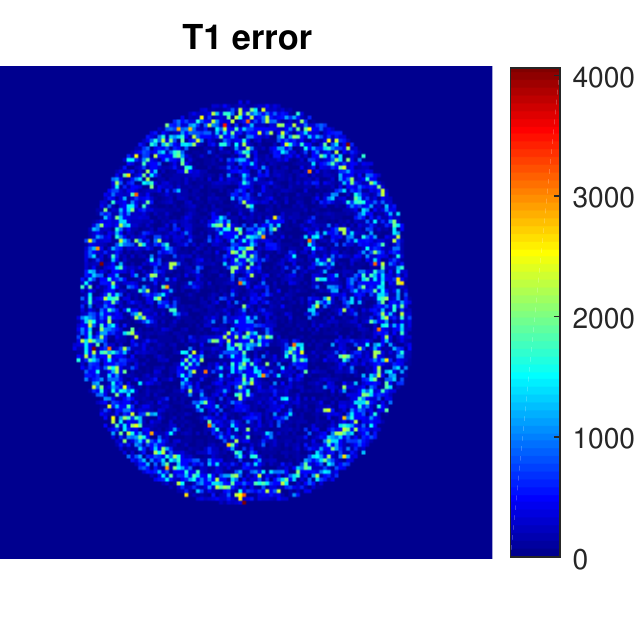}
\includegraphics[width=0.32\textwidth, trim={6.475cm 0 6.475cm 0},clip]{n_4th_sampled_BLIP_error-eps-converted-to.pdf}
\includegraphics[width=0.32\textwidth, trim={12.95cm 0 0 0},clip]{n_4th_sampled_BLIP_error-eps-converted-to.pdf}\\
\includegraphics[width=0.32\textwidth, trim={0 0 12.95cm 0},clip]{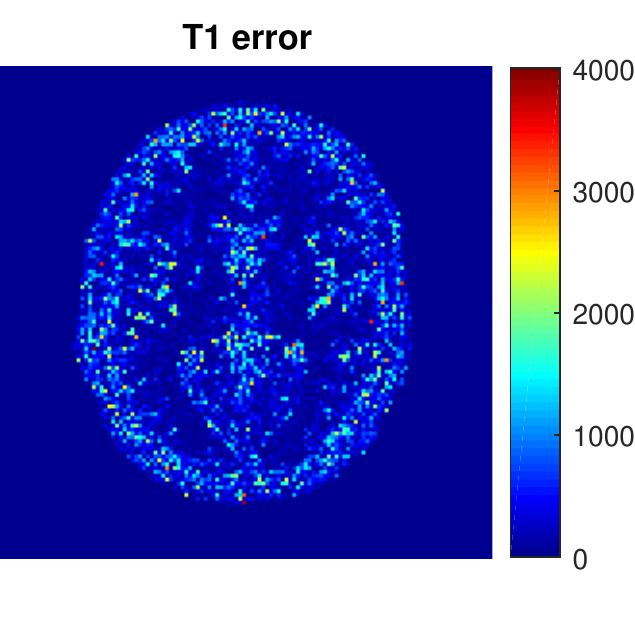}
\includegraphics[width=0.32\textwidth, trim={6.475cm 0 6.475cm 0},clip]{n_4th_sampled_pLM_error-eps-converted-to.pdf}
\includegraphics[width=0.32\textwidth, trim={12.95cm 0 0 0},clip]{n_4th_sampled_pLM_error-eps-converted-to.pdf}\\
\caption{Experiment with undersampled and noisy data. Pointwise distance of the solutions of  Figure \ref{fig:noise_undersampled_solutions}, to the corresponding ground truths of Figure \ref{fig:real_solutions}. First row: Initial error of BLIP with a coarse dictionary. Middle row: error of BLIP with fine dictionary. Last row: Error of the proposed algorithm.}
\label{fig:noise_undersampled_error}
\end{figure}
\begin{figure}[!ht]
\centering
\includegraphics[width=\textwidth]{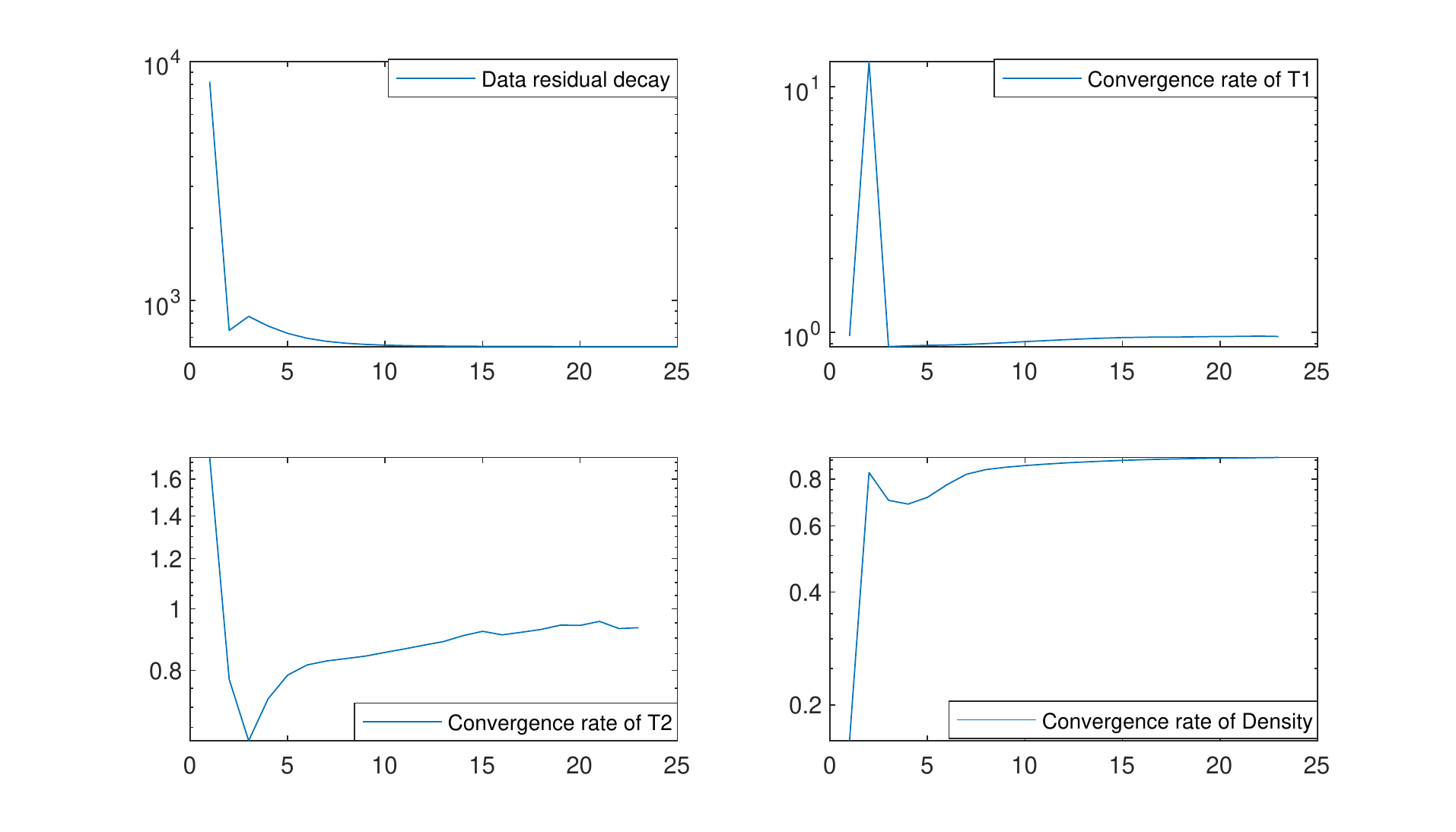}
\caption{Experiment with noisy data.  From left to right and from above to bottom: Convergence of the data residual $\|Q(\f x_{n})-D^\delta\|_{2}$,
convergence rates via plots of the iterate ratios $\frac{\norm{x_{n+1}-x_{n}}_{2}}{\norm{x_{n}-x_{n-1}}_{2}}$ for $x= T_1, T_2, \rho$ respectively.}
\label{fig:noise_undersampled_convergence}
\end{figure}

\subsubsection{Ideal data--fully sampled and no noise}
We also discuss the results for the case of fully sampled and noise-free data.
Here, we only need $L=3$ data frames, which actually equals the number of unknown parameters. Thus, the resulting discrete system is non-singular. For both the  BLIP  and our algorithm, we execute $5$ iterations.
The regularization parameters were chosen as $\mu_{n}=0$ for every $n\in\mathbb{N}$, $\lambda_0=s^2=1$, and $\beta=0$. Note that, as discussed earlier, this choice makes the L-M iteration equivalent to the Gauss-Newton method.

Here, we only show the error maps of the results  in Figure \ref{fig:full_sampled_error}.
We observe that the Gauss-Newton algorithm essentially recovers the ground truth as expected, while the accuracy of BLIP is limited by the discretization mesh of the dictionary.

In contrast to the previous case, as we verify numerically in Figure \ref{fig:full_sampled_convergence},  the convergence rate of the algorithm is superlinear.

\begin{figure}[!ht]
\centering
\includegraphics[width=\textwidth]{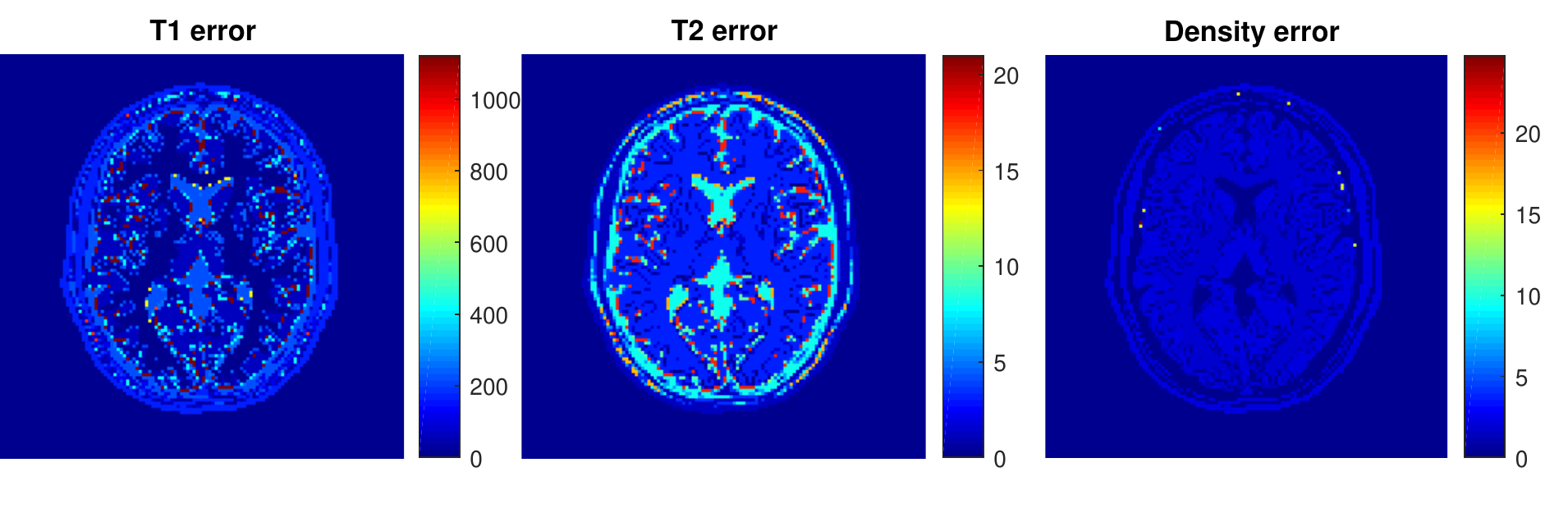}\\
\includegraphics[width=\textwidth]{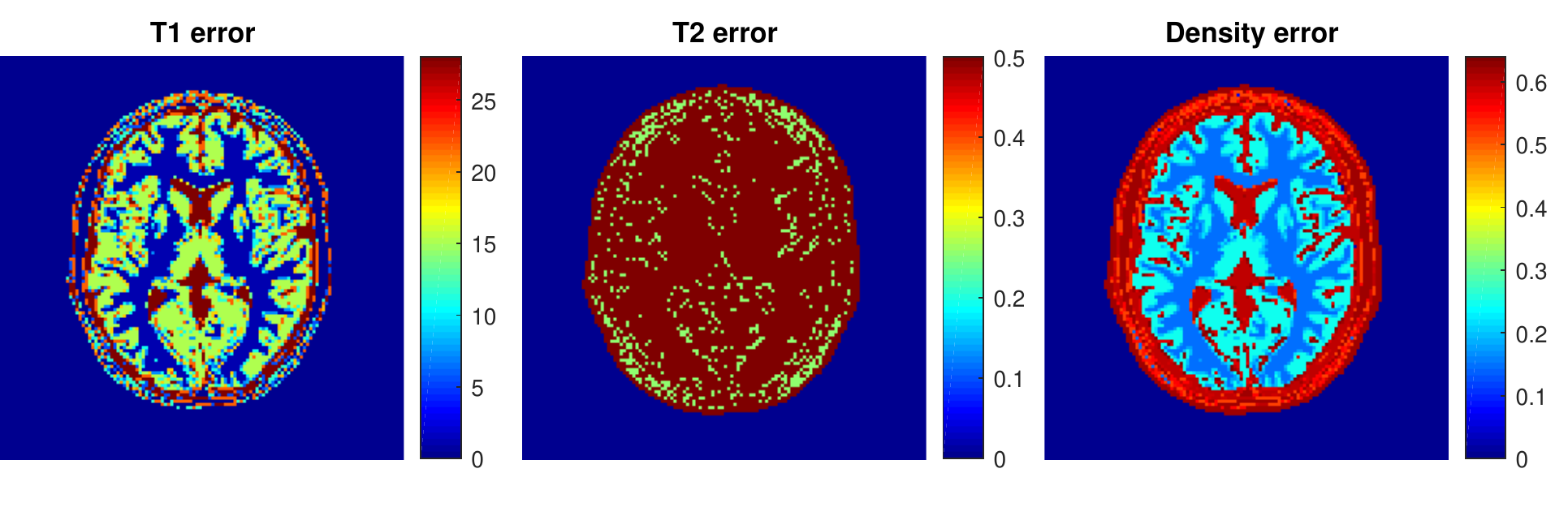}\\
\includegraphics[width=\textwidth]{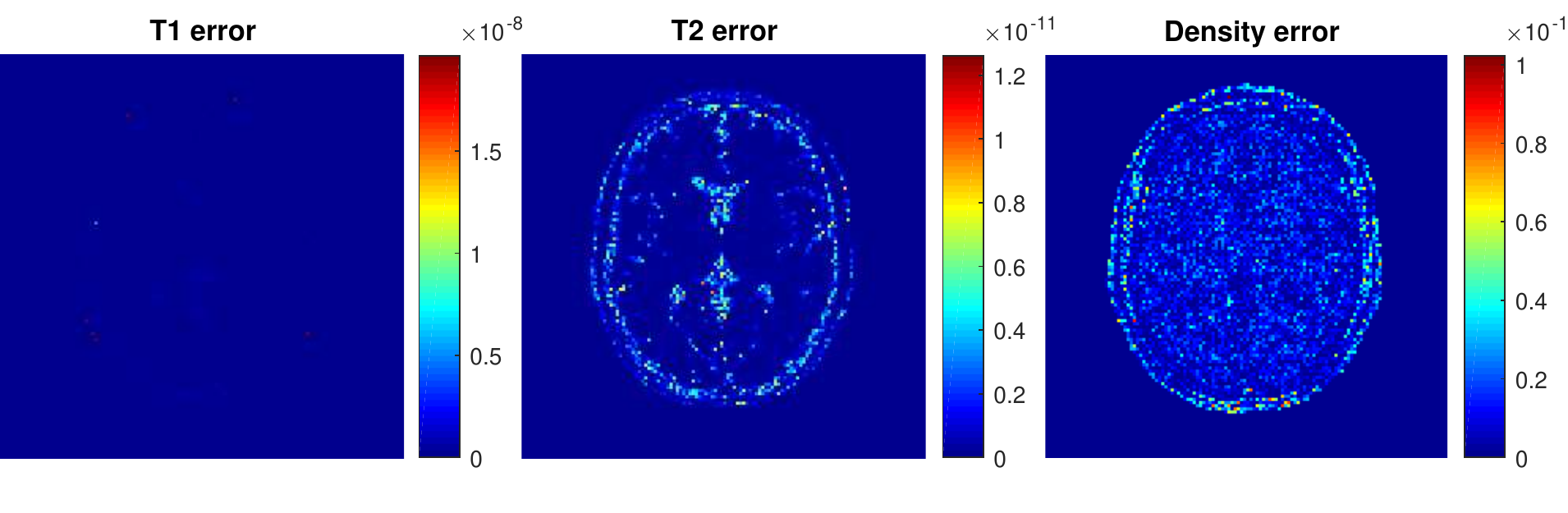}
\hspace{0.61pt}
\caption{Fully sampled data.  Pointwise distance of the solutions of BLIP algorithm and Newton algorithm to the corresponding ground truths. First row: Initial error of BLIP with a coarse dictionary. Middle row: error of BLIP with fine dictionary. Last row: Error of the proposed algorithm.}
\label{fig:full_sampled_error}
\end{figure}

\begin{figure}[!ht]
\centering
\includegraphics[width=\textwidth]{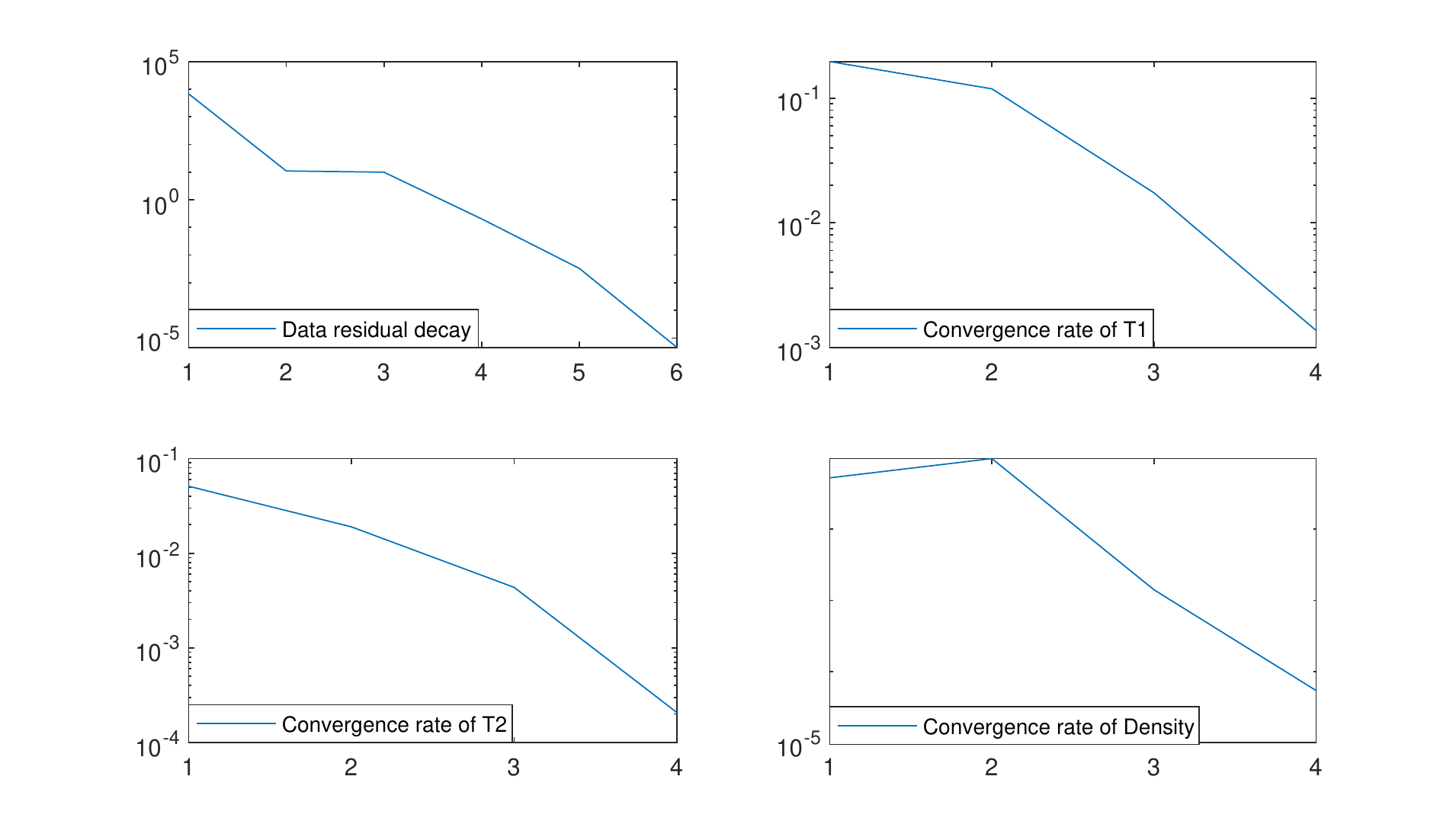}
\caption{Fully sampled data and superlinear convergence of the Newton method. From left to right and from above to bottom: Convergence of the data residual $\|Q(\f x_{n})-D\|_{2}$,
convergence rates via plots of the iterate ratios $\frac{\norm{x_{n+1}-x_{n}}_{2}}{\norm{x_{n}-x_{n-1}}_{2}}$ for $x= T_1, T_2, \rho$ respectively.}
\label{fig:full_sampled_convergence}
\end{figure}

\subsubsection{Quantitative comparisons}
\label{subsubsec:quantitative}
In  Table \ref{tab:Q_result} we provide a summary of further qualitative comparisons for all of the previous tests.
The index in our comparison is the cost in CPU-time as well as the error rates of each algorithm, with the latter defined as
\[ \frac{\|x_{computed}-x_{gt}\|_{2}}{\|x_{gt}\|_{2}},\]
where $x=T_{1}, T_{2}$ or $\rho$.
We provide comparisons with the initial value $\f x_{0}$ and also with the results of the L-M algorithm without projection.
Note that the CPU-time costs for the algorithm proposed here include the time needed for computing the initial value.
\begin{table}[!ht]
\caption{Quantitative summary of the results: computational times and error rates }
\begin{center}
\resizebox{\textwidth}{!}{
\begin{tabular}{ c| llll|llll|llll}
\hline
    &\multicolumn{4}{c|}{Full data}  &\multicolumn{4}{c|}{ $1/8$ sampled data} & \multicolumn{4}{c}{ $1/4$ sampled and noisy data} \\ 
                & time ($s$) & \multicolumn{3}{c|}{ error rate $\frac{\norm{e}_{2}}{\norm{x^*}_{2}}$} &time ($s$) & \multicolumn{3}{c|}{ error rate $\frac{\norm{e}_{2}}{\norm{x^*}_{2}}$}  & time ($s$) &  \multicolumn{3}{c}{ error rate $\frac{\norm{e}_{2}}{\norm{x^*}_{2}}$}  \\[5pt] \hline
                &   & $T_1$ & $T_2$ &  $\rho$  &   & $T_1$ & $T_2$ &  $\rho$  &   & $T_1$ & $T_2$ &  $\rho$   \\ \hline
Initial  & $1.20$  &   $0.036 $ & $0.009 $ & $0.008 $ &$15.26$  & $0.472 $  &  $0.010 $ &  $0.003 $ & $13.67$  &  $0.148 $ &  $0.088 $&  $0.188 $\\ 
BLIP & $78.94$ & $ 0.005 $  &  $0.002 $ & $0.003 $ & $964.04$ &  $0.072 $ & $0.002 $ & $0.001$  &   $1073.86 $ &  $0.078 $ &  $0.019 $ &  $0.028 $ \\ 
L-M & $8.41$  &   $ 1.6*10^{-13} $ & $2.6*10^{-15} $  &  $6.1*10^{-16} $  &$489.82$ & $0.009$  &   $0.001 $ &  $0.0002 $ &  $493.23 $ & $0.072 $ &  $0.014 $&  $0.020 $ \\ 
Proposed&  $8.47$  & $ 1.6*10^{-13} $ & $2.4*10^{-15} $  &  $5.6*10^{-16} $ &$494.56$ & $0.015$  &   $0.002 $ &  $0.0002 $ &  $495.77 $ & $0.070 $ &  $0.011 $&  $0.009 $  \\ 
\hline
\end{tabular}}
\label{tab:Q_result}
\end{center}
\end{table}
\begin{remark}
 {
We note that a latest version of BLIP, called CoverBLIP \cite{GolCheWiaDav18}, appeared recently, and is reported to be $2$-$6$ times faster than the one we have used in our tests. However since the outcomes of BLIP and CoverBLIP are the same, the error rates that we report here will essentially not change. 
}
\end{remark}
From the table we observe that the initial guess has been significantly improved by both algorithms, but in the end our proposed algorithm outperforms the refined BLIP in all of the indices.
We can see from the table that the non-projected L-M algorithm provides comparable results as the projected L-M method in the cases the data with no additive noise. This is not surprising since the initialization produces values well located in the interior of the feasible set. As a consequence the non-projected L-M iterations can almost be reside in the feasible domain. Hence, the projection appears unnecessary in the noise-free case.

Finally, we would like to verify the fact that larger frames sequences can help to get more accurate estimations; compare the discussion at the end of Section \ref{subsec:Chebyshev}. For the results shown in Table \ref{tab:Sequence_comparison} we performed a set of experiments for data frames of increasing length $L$. All data were $1/4$ sub-sampled and corrupted by additive i.d.d. Gaussian noise as described before with variance $1$ and mean $0$. This gives a total signal noise ratio $SNR=15$.
We then ran our L-M algorithm for 20 iterations always using the same initial value $\f x_{0}$, which was generated by using BLIP with $160$ frames and the coarse dictionary as described in Section \ref{subsubsec:subsample_noise}.
The parameter had values $\lambda_n=\mu_n=\lambda_0\norm{Q\f x_n-D^\delta}_{2}$ for all $n\in \mathbb{N}$, and $\lambda_0=10^{-8}$.

The results in Table  \ref{tab:Sequence_comparison} clearly indicate that an increasing number of sequences improves the accuracy of the estimated parameters.
There we have also computed the errors of the standard L-M method, i.e., with no projection.
It can be observed  (also in Table \ref{tab:Q_result} for the noisy case) that the projected L-M method outperforms the standard L-M method of no projection.
Note that the CPU-time costs that we show here do not include initialization as the latter was the same in each case.
\begin{table}[!ht]
\caption{The length of data sequences and its influence on the solution accuracy }
\begin{center}
\resizebox{\textwidth}{!}{
\begin{tabular}{l|llll|llll}
\hline
   & \multicolumn{4}{c|}{Standard L-M method}  &\multicolumn{4}{c}{ Projected L-M method (proposed)}  \\ \hline
 & ER $T_1$ & ER $T_2$ & ER $\rho$  &  time ($s$) & ER $T_1$ & ER $T_2$ & ER $\rho$  &  time ($s$) \\ \hline
$L=5$ & $ 0.2267 $  &  $0.4923 $ & $0.1682 $ & $36.66$ & $ 0.1743 $  &  $0.2028 $ & $0.0424 $ & $36.40$ \\ 
$L=10$ & $ 0.1818 $  &  $0.0805 $ & $0.0757 $ & $59.69$ & $ 0.1699 $  &  $0.0348 $ & $0.0275 $ & $59.33$ \\ 
$L=20$ & $ 0.0542$  &  $0.0182 $ & $0.0317 $ & $104.56$ & $ 0.0290 $  &  $0.0072 $ & $0.0099 $ & $104.65$ \\ 
$L=40$ & $ 0.0413 $  &  $0.0131 $ & $0.0276 $ & $193.65$ & $ 0.0211 $  &  $0.0051 $ & $0.0090 $ & $194.28$ \\ 
$L=80$ & $ 0.0268 $  &  $0.0117 $ & $0.0268 $ & $375.90$ & $ 0.0121 $  &  $0.0043 $ & $0.0087 $ & $374.41$ \\  
$L=160$ & $ 0.0193 $  &  $0.0112  $ & $0.0266  $ & $736.77 $ & $ 0.0078 $  &  $0.0041 $ & $0.0085 $ & $737.45$ \\ \hline
\end{tabular}}
\label{tab:Sequence_comparison}
\end{center}
\end{table}

\subsection{Experiments on radial sub-sampled data}
\label{subsec:exa_numerics}
Next we demonstrate that the proposed method works efficiently also for other types of sub-sampling schemes. 
Additionally, here we consider the proton density function $\rho$ to be complex-valued.
This is particularly relevant when taking into account the effect of coil sensitivities and phase shifts. In such cases, the complex-valued proton density absorbs the coil sensitivity map and phase shift (usually by a multiplication of the actual density function with the sensitivity map). 
Note that in general the sensitivity map can be estimated for a given coil under fixed physical settings, see e.g. the method in \cite{KeeHinKnoKraLau11}. Therefore we may consider them as known quantities. After computing the complex-valued proton density function, and taking into account the known quantities of the coil sensitivity map, one is able to estimate  the actual proton density function by solving algebraic linear equations.
As we deal with synthetic data, this step is omitted in the paper.
For the sake of visualization, we set the imaginary part to be a constant $C$ minus the real part $imag(\rho)=C-real(\rho)$.
The real part is set to be the same as in the previous examples, and we choose $C=180$.

We use again synthetic data but with radial sub-sampling at the sampling rate ($12.1\%$). That is, $16$ out of $128$ angle-strips are used to collect data in each frame. 
We corrupt the data again with additive zero-mean Gaussian noise to the effective part of the magnetization function. 
Three levels of noise are considered: small (SNR=$78.24$), medium (SNR=$19.88$) and large  (SNR=$3.41$), of variance $\sigma^{2}=1$,  $\sigma^{2}=2$, and $\sigma^{2}=5$, which correspond to
the number of acquisition sequences $L=80$, $L=200$ and $L=500$, respectively.
For comparison, noise-free data are also tested with a sampling rate $74.02\%$. 
In such a case, we use an acquisition sequence of length  $L=3$ for reconstruction.

In this set of examples, we compare our results not only to BLIP but also to the FLOR algorithm \cite{MazWeiTalEld18}.
In FLOR, an additional regularisation parameter  $\beta$ is used to control the rank of the matrix representation of the magnetization function.
We choose the parameter $\beta$ manually in order to provide sufficiently good results in the tests. The number of iterations is fixed to be $30$, as suggested in \cite{MazWeiTalEld18}, while for BLIP we keep the same  iteration numbers and step lengths  as in the previous tests.
We adjust the regularization parameter of FLOR to be $\beta=0$, $\beta=50$, $\beta=100$, $\beta=150$,  for noise-free, small, moderate, and large noise-level experiments, respectively.

For the initialization we use a coarse dictionary to generate an initial value $\f x_0$.
For this coarse dictionary,
$T_1$ is discretized from $500ms$ to $5500ms$ with increments of $500ms$, and $T_2$ from $50ms$ to $550ms$ with increments of $50ms$. 
This resulted in a dictionary with $128$ entries only, and needed a complex-valued matrix of dimension $128\times L$ for its representation.
Here, we employ the BLIP algorithm using the coarse dictionary in order to generate the initial guess for our algorithm.
For the BLIP and FLOR algorithm, we use the same refined dictionary as in the previous examples.
The regularization parameter for the projected L-M method is chosen as $\lambda_n= \mu_n= 10^{-10}\norm{Q\f x_n-D^\delta}_{2}$, which turns out to be efficient.

The summary of the numerical tests is presented in Table \ref{tab:Additional}.
Upon inspecting the table, we find that when the level of noise in the data increases, the performance of BLIP  using a finer dictionary deteriorates significantly as it is more affected by the noisy information. 
The FLOR algorithm works more stably under moderate and large noise.  It is particularly efficient for recovering the $T_1$ parameter when compared to the other methods. This is because $T_1$ is more sensitive to noise than the other parameters.
Our proposed method exhibits a relatively high accuracy for estimating $T_2$ and the density $\rho$ (both real and  imaginary parts) in all cases, namely noise-free, small, medium and large noise levels.
We also point out that in the medium and large noise cases, the estimation of $T_1$ by the proposed method is not as good as the results obtained by the FLOR algorithm.
This is not surprising given the analytical expressions of $M'(\theta)$ with respect to $T_1$.
Due to the large magnitude of $T_1$, the linearised operator with respect to $T_1$ is more ill-posed than with respect to the other parameters.
This shows that more sophisticated  schemes for regularizing $T_1$ are required in order to take care of data with strong noise in the current framework, which may  serve as a topic for future investigation.

\begin{table}[!ht]
\caption{Numerical results for radial sub-sampled data with complex proton density function}
\begin{center}
\resizebox{\textwidth}{!}{
\begin{tabular}{ c|llll|llll|llll|llll}
\hline
  &\multicolumn{4}{c|}{Sampling rate $74.02\%$} &\multicolumn{12}{c}{Sampling rate $12.02\%$ }  \\ \hline
  &\multicolumn{4}{c|}{No-noise ER $\frac{\norm{e}_{2}}{\norm{x^*}_{2}}$}&\multicolumn{4}{c|}{S-noise ER $\frac{\norm{e}_{2}}{\norm{x^*}_{2}}$} & \multicolumn{4}{c|}{ M-noise ER $\frac{\norm{e}_{2}}{\norm{x^*}_{2}}$}  &   \multicolumn{4}{c}{ L-noise ER $\frac{\norm{e}_{2}}{\norm{x^*}_{2}}$}  \\  
  &\multicolumn{4}{c|}{SNR=$\infty$, L=$3$} &\multicolumn{4}{c|}{SNR=$78.24$, L=$80$} & \multicolumn{4}{c|}{ SNR=$19.88$, L=$200$}  &   \multicolumn{4}{c}{ SNR=$3.41$, L=$500$}  \\ \hline 
&  $T_1$ & $T_2$ &  r-($\rho$) & i-($\rho$) & $T_1$ & $T_2$ &  r-($\rho$) & i-($\rho$) & $T_1$ & $T_2$ &  r-($\rho$) & i-($\rho$)  & $T_1$ & $T_2$ &  r-($\rho$) & i-($\rho$)    \\ \hline
Initial & $0.156$ & $0.063$  &  $0.025 $ & $0.029$ & $0.207 $ & $ 0.072 $  &  $0.089  $ & $ 0.105 $ & $0.196 $ &  $0.089$ & $ 0.086 $ & $ 0.101$  &   $  0.200$ &  $0.093  $ &  $ 0.095 $ &  $0.112 $ \\ \hline
BLIP & $0.096$ & $0.039$  &  $0.008$ & $  0.008 $ & $0.445 $ & $0.073 $  &  $0.025 $ & $0.033$ & $0.575 $ &  $0.090  $ & $ 0.030 $ & $0.039 $  &   $ 0.773 $ &  $ 0.125 $ &  $ 0.032 $ &  $ 0.040$ \\ \hline
FLOR & $0.095$ & $0.082$  &  $ 0.007$ & $0.008$  &  $0.120 $ & $ 0.061  $  &  $ 0.056 $ & $ 0.076 $ & $0.065 $ &  $ 0.046 $ & $ 0.062 $ & $0.079 $  &   $0.069  $ &  $ 0.051 $ &  $ 0.070$ &  $ 0.086$ \\ \hline
Proposed & $0.034$ & $0.023$  &  $ 0.001 $ & $0.001$ &   $0.125 $ & $ 0.019 $  &  $0.006$ & $ 0.006$ & $0.134 $ &  $ 0.016 $ & $ 0.008 $ & $0.007$  &   $0.207  $ &  $ 0.027 $ &  $ 0.014 $ &  $ 0.013$ \\ 
\hline
\end{tabular}}
\label{tab:Additional}
\end{center}
\end{table} 

For the sake of space, we only show here the visualization comparisons in the case of medium noise level. Figure \ref{fig:m_radial_solution} presents the ground truth solution and also solutions of the estimated parameter functions using the different methods, and Figure \ref{fig:m_radial_error} shows the relative error map of each method. Furthermore, Figure \ref{fig:m_radial_convergence} verifies the linear convergence rate of the proposed method for complex-valued density functions. 

\begin{figure}[!ht]
\centering
\includegraphics[width=\textwidth]{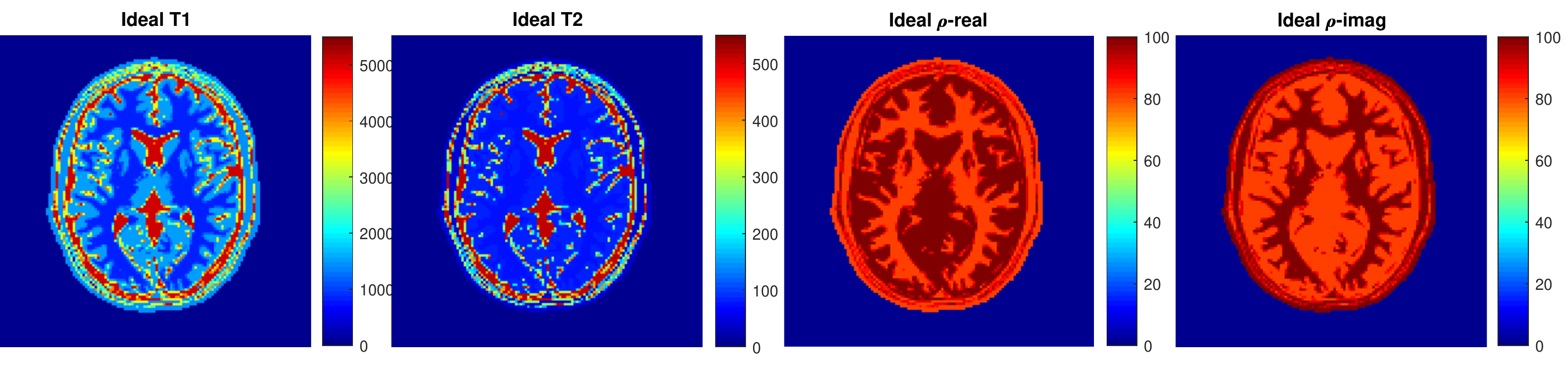}\\
\includegraphics[width=\textwidth]{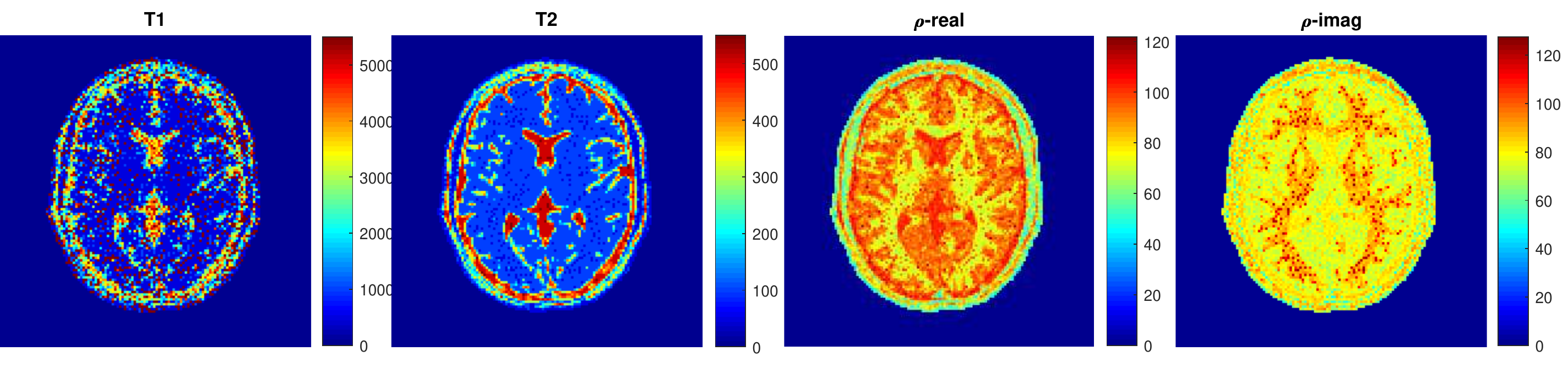}\\
\includegraphics[width=\textwidth]{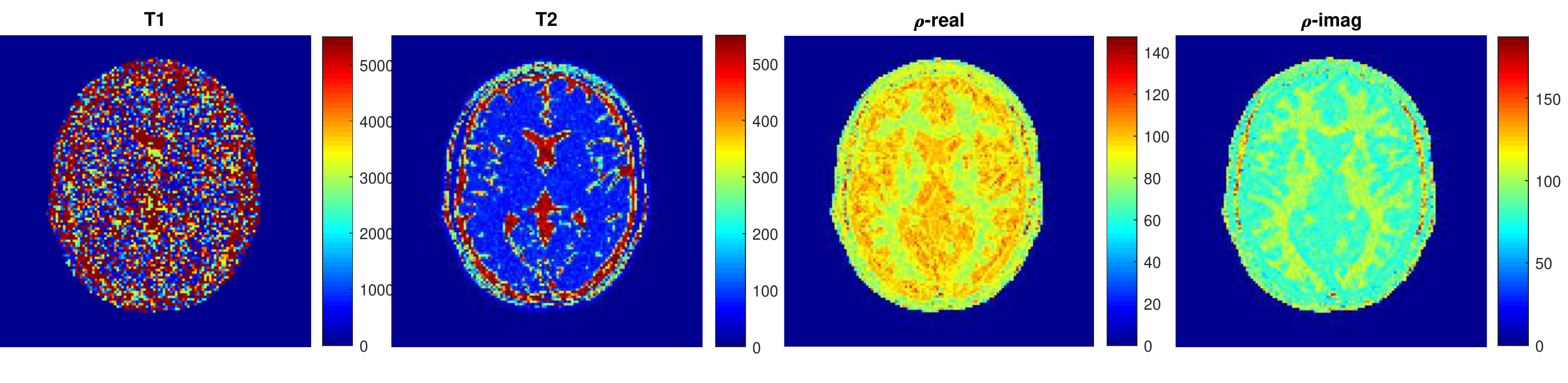}\\
\includegraphics[width=\textwidth]{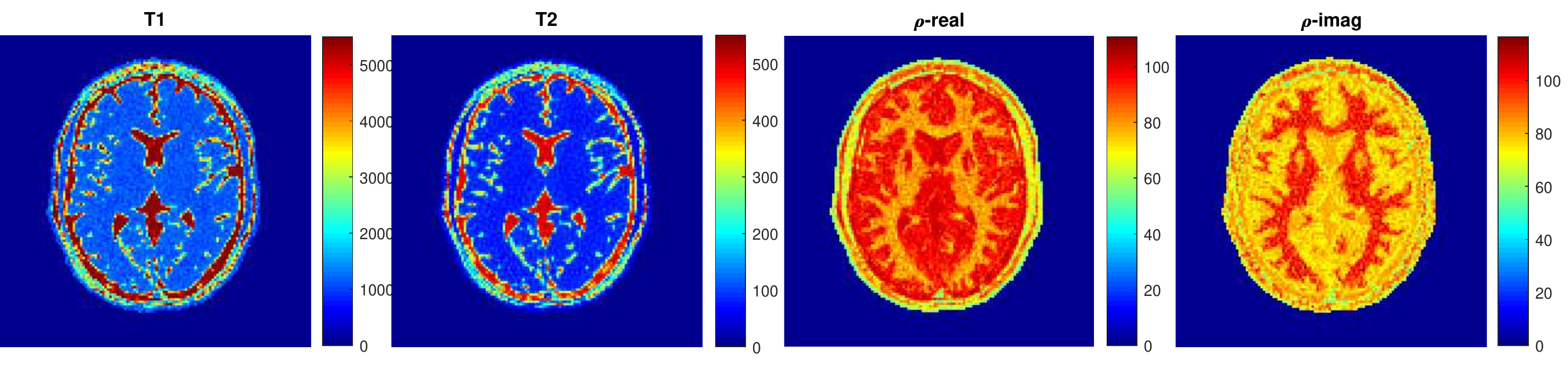}\\
\includegraphics[width=\textwidth]{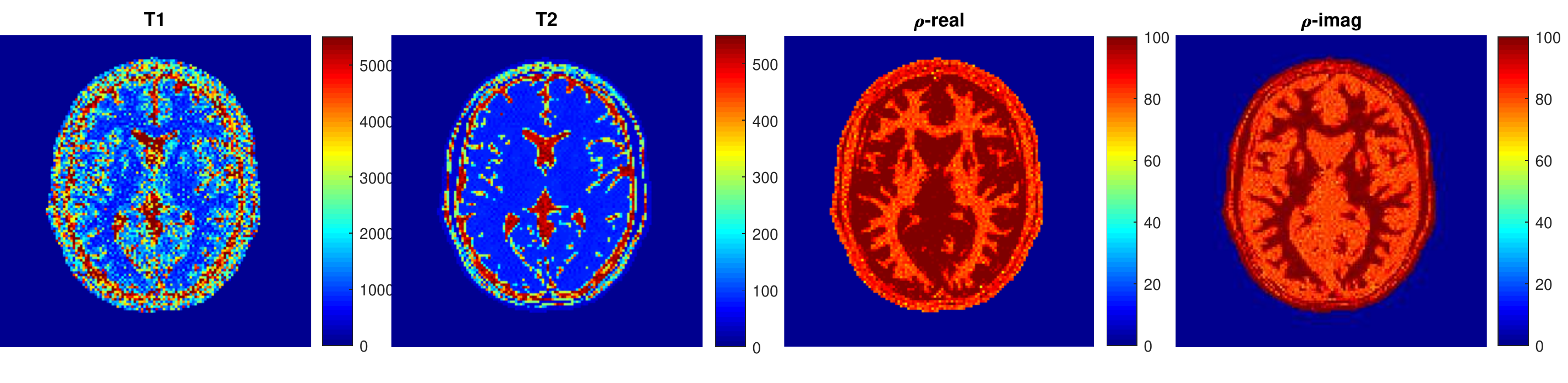}
\label{fig:m_radial_solution}
\caption{Experiment with medium-noise (SNR$\simeq 20$) of radial-sampled data at rate $12.1\%$. First row: Ground truth.
Second row: Initial solution of BLIP with a coarse dictionary. Third row: Solution of BLIP with fine dictionary. Fourth row: Solution of FLOR with fine dictionary. Last row: Solution of the proposed algorithm.}
\end{figure}

\begin{figure}[!ht]
\centering
\includegraphics[width=\textwidth]{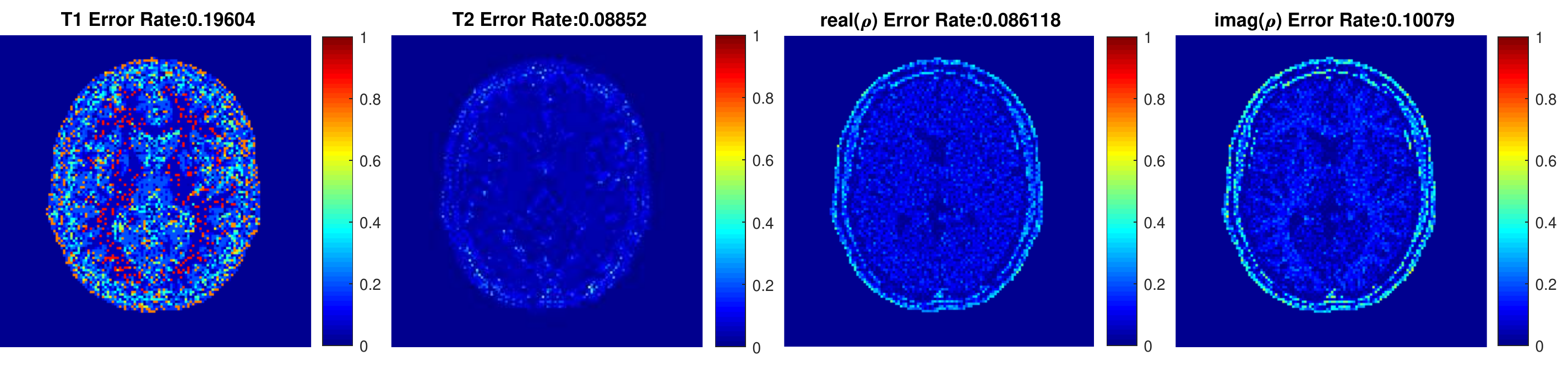}\\
\includegraphics[width=\textwidth]{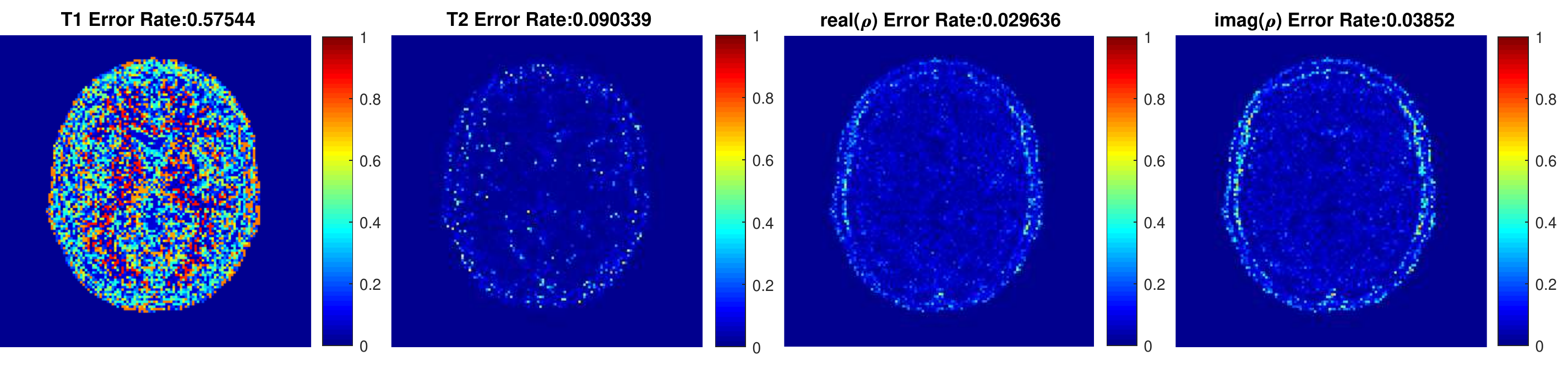}\\
\includegraphics[width=\textwidth]{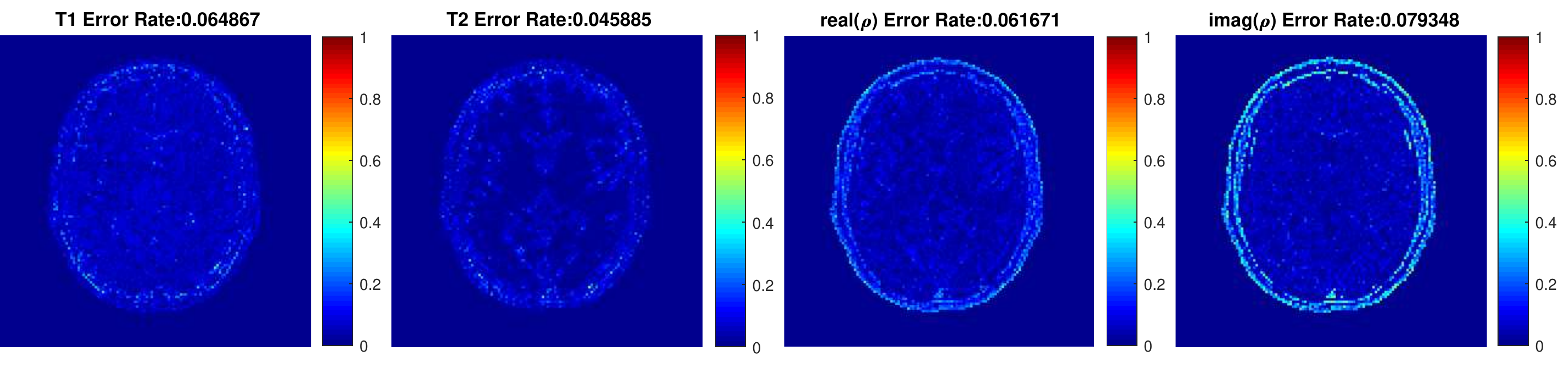}\\
\includegraphics[width=\textwidth]{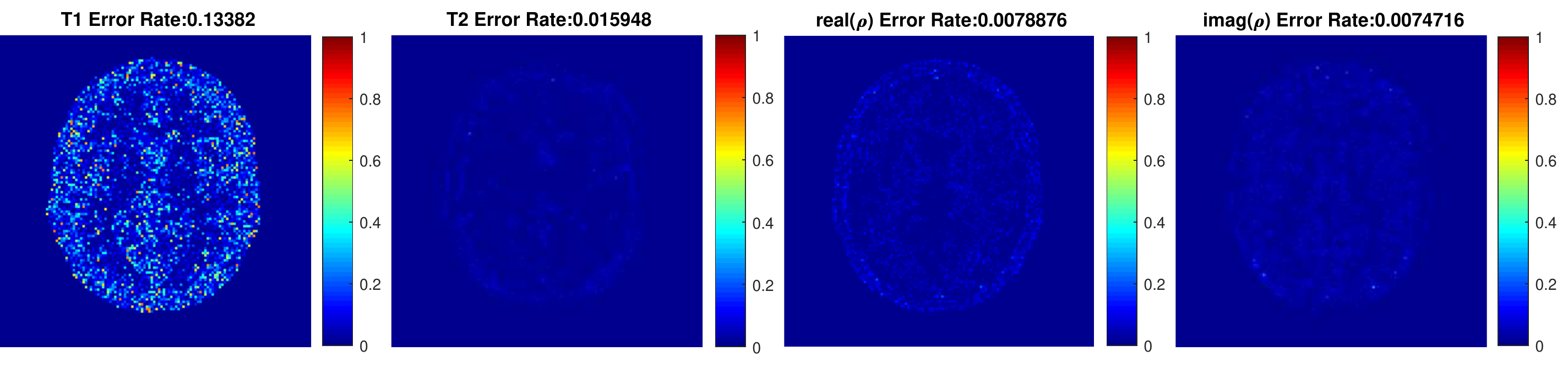}
\label{fig:m_radial_error}
\caption{Experiment with medium-noise (SNR$\simeq 20$) of radial-sampled data at rate $12.1\%$. Relative error map of solutions presented in Figure \ref{fig:m_radial_solution}. First row: Initial error of BLIP with a coarse dictionary. Second row: Error of BLIP with fine dictionary. Third row: Error of FLOR with fine dictionary.
Last row: Error of the proposed algorithm.}
\end{figure}

\begin{figure}[!ht]
\centering
\includegraphics[width=\textwidth]{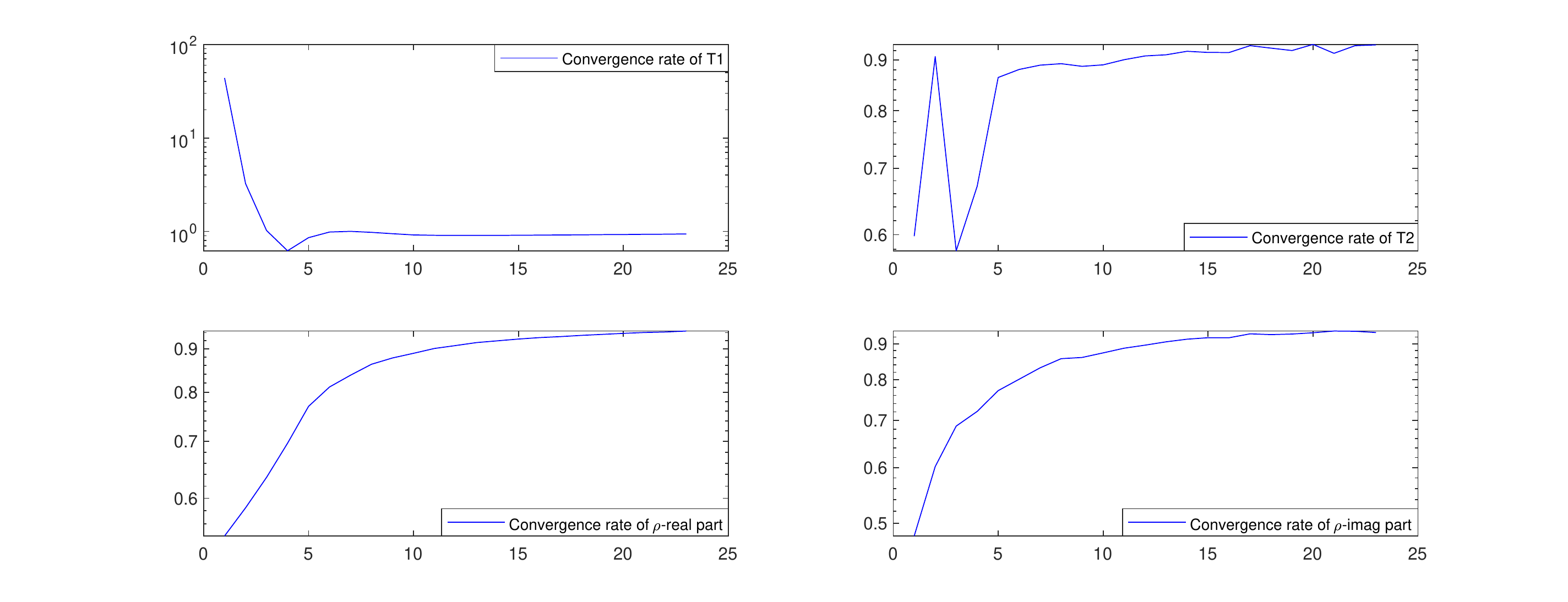}\\
\label{fig:m_radial_convergence}
\caption{Experiment with medium-noise(SNR$\simeq 20$) of radial-sampled data at rate $12.1\%$.  From left to right and from above to bottom: Convergence rates via plots of the iterate ratios $\frac{\norm{x_{n+1}-x_{n}}_{2}}{\norm{x_{n}-x_{n-1}}_{2}}$ for $x= T_1, T_2, \text{real}(\rho), \text{imag}(\rho)$ respectively.}
\end{figure}

\section{Concluding remarks}
In this paper,  we have analysed MRF from the perspective of inverse problems, and we were able to provide some mathematical insights in order to better understand the functionalities of the method.
Subsequently, we have proposed a novel model for quantitative MRI which is in accordance  with the standard routine of the MRI experiment setting. The model is dictionary-free and incorporates the physical setting of MRI into one single non-linear equation. We have proposed a robust algorithm that was shown to be capable of estimating the tissue parameters with high precision.
In contrast to the original MRF method and many of its variants, it does not rely on refining a dictionary to improve the accuracy. Even though our new algorithm is based on a specific Bloch dynamics referred to as IR-bSSFP,  this constitutes by no means a limitation for the method.
Rather, other type of discrete dynamics or approximations to Bloch equations can be fitted to this approach as well. Furthermore,  we have considered  the relaxation parameters $T_1$, $T_2$ and the proton density $\rho$ as unknowns  in the present paper, but as long as parameters can be related to the Bloch dynamics, there would be no difficulty to extending the algorithm to incorporate the further parameters into the current framework. 

 {Regarding future work, numerical results suggest that more sophisticated regularization schemes are needed in order to better estimate $T_1$  when data contain strong noise.
Furthermore, the partial volume effect for low resolution images, which has been considered in the literature, may also be addressed in our framework. Indeed, one potential way of doing so is to enforce certain regularity on the variation of the parameter functions. For instance, one may invoke total variation or total generalized variation regularization priors, among many others. }

\subsection*{Acknowledgements}
This work has been conducted within the \textsc{Matheon} Research Center project CH12
funded by the Einstein Center for Mathematics (ECMath) Berlin within its Innovation Area ``Mathematics in Clinical Research and Health Care''.
GD thanks Barbara Kaltenbacher for a helpful discussion on Levenberg-Marquardt methods for ill-posed problems during an ESI workshop in Vienna.
The authors would also like to thank the anonymous reviewers for their valuable remarks and comments that resulted in an improved version of the paper.

\appendix
\subsection*{Appendix: Solutions of Bloch equations with different cases}
Here we briefly review several simplified cases towards the solutions of the Bloch equations,
which are helpful in order to understand the simulations based on discrete dynamics.
More detailed descriptions can be found in \cite{Kup00,Nis10}.
Note that here we omit the position dependence in Bloch equations.
\subsubsection*{Only main field with no relaxation}
The Bloch equations, in a setting which only takes into account the main magnetic field and with no relaxation, represent an autonomous dynamical system, that is
\begin{equation*}
 \frac{\partial m}{\partial t} =  m \times\gamma B_0.
\end{equation*}
The solution in this case is
\[m(t)=P_{\omega_0}(t)m(0), \]
where
\[
P_{\omega_0}(t)=\left(
\begin{array}{ccc}
\cos(\omega_0 t ) & \sin(\omega_0 t ) & 0\\
-\sin(\omega_0 t )  &\cos(\omega_0 t ) & 0\\
  0 & 0& 1 
\end{array}\right),\quad \text{ and } \omega_0 =\gamma\abs{B_0}.
\]
It can be interpreted in a way that the magnetization precesses about the main magnetic field at a  frequency $\omega_0$, called Larmor frequency.
\subsubsection*{Main field with relaxation}
This is the case of Bloch equations \eqref{eq:Bloch_intro} with $B(t,r)=B_0$,
meaning that radio frequencies and gradient fields are not considered here. After some change of variable and further calculations, the solution turns out to be
\[m(t)=P_{\omega_0}(t)E(t)  m(0) +(1-e^{-\frac{t}{T_1}})m_e, \]
where
\[
E(t)=\left(
\begin{array}{ccc}
e^{-\frac{t}{T_2}} & 0 & 0\\
 0  &e^{-\frac{t}{T_2}} & 0\\
  0 & 0& e^{-\frac{t}{T_1}} 
\end{array}\right).\]
Note that the matrices $P_{\omega_0}(t)$ and $E(t)$ are commutable.
\subsubsection*{With perturbations and without relaxation}
By perturbation we mean that there is a $B_1$ field which rotates at the Larmor frequency, and it is always orthogonal to the main field, such that $\langle B_0,B_1\rangle =0$. This models the excitation of radio pulses in the MRI machine.
By convention, the direction of the $B_1$ field can be defined to be along the $x$-axis in space.
Since in reality, the excitation pulse only lasts for a very short length of time in comparison with $T_1$ and $T_2$, we can ignore the relaxation terms.
The solution of \eqref{eq:Bloch_intro} in the case of no relaxation terms but with perturbation is
\[ m(t) =P_{\omega_0}(t) R_x(\alpha(t)) m(0),\]
where  $\alpha(t):=\gamma \int_0^t\abs{B_1(s)}ds$ is the flip angle, and
\[R_x(t)=\left(
\begin{array}{ccc}
1& 0 & 0 \\
0& \cos(\alpha(t)) & \sin(\alpha(t)) \\
0& -\sin(\alpha(t))  &\cos(\alpha(t))
\end{array}\right).\]

\subsubsection*{With perturbations and relaxation}
Finally we are able to simulate the solutions of \eqref{eq:Bloch_intro} in the case where both the perturbations of the main field and relaxation terms are taken into account.
This is based on the assumption that the excitation pulse is turned on at the time period $(0,t_0)$, where $t_0$ is a very small number in comparison to the relaxation parameters.
Therefore, we can estimate the solution of \eqref{eq:Bloch_intro} with the following formula:
\[m(t)= P_{\omega_0}(t)E(t)R_x(\alpha(t_0))  m(0)+ (1-e^{-\frac{t}{T_1}})m_e.\]
The main tool in all of the above calculations is to change variables to a rotating frame of reference in order to match the Larmor precession.


\bibliographystyle{amsplain}
\bibliography{mybibfile}

\end{document}